%% file: article.tex
\documentclass[letter, 11pt]{article}

\usepackage{graphicx,psfrag,amsfonts,verbatim,fullpage,mathtools}
\usepackage[hidelinks]{hyperref}

\usepackage{abstract} %
\usepackage{xurl}
\usepackage{lmodern}

\usepackage[T1]{fontenc}
\usepackage{amsmath}
\usepackage{amsthm}
\usepackage{amssymb}
\usepackage{hyperref}
\usepackage{cleveref}
\usepackage{xcolor}
\usepackage{nicefrac}
\usepackage{booktabs}
\usepackage{multirow}
\usepackage{mathtools}
\usepackage{cancel}
\hypersetup{
    hidelinks,
    colorlinks = true,
    urlcolor = black,
    citecolor = black,
    linkcolor = black
}
\usepackage{microtype}
\usepackage{graphicx}
\usepackage[labelfont={bf,sf},font=small]{caption}
\usepackage[labelfont={bf,sf},font=small]{subcaption}

\usepackage{float}

\usepackage{babel}

\usepackage{parskip}
\begingroup
\makeatletter
   \@for\theoremstyle:=definition,remark,plain\do{%
     \expandafter\g@addto@macro\csname th@\theoremstyle\endcsname{%
        \addtolength\thm@preskip\parskip
     }%
   }
\endgroup
\usepackage{parskip}

\usepackage{diagbox}

\definecolor{commentcolor}{RGB}{74,112,35}

\usepackage{titlesec}
\def\sectionfont{\sffamily\Large\bfseries\boldmath}
\def\subsectionfont{\sffamily\large\bfseries\boldmath}
\def\paragraphfont{\sffamily\normalsize\bfseries\boldmath}
\titleformat*{\section}{\sectionfont}
\titleformat*{\subsection}{\subsectionfont}
\titleformat*{\subsubsection}{\paragraphfont}
\titleformat*{\paragraph}{\paragraphfont}
\titleformat*{\subparagraph}{\paragraphfont}

\makeatletter

\usepackage{fullpage, url}

\usepackage{abstract} %

\usepackage{thmtools}

\declaretheorem[numberwithin=section,style=plain]{theorem}
\declaretheorem[numberwithin=section,style=plain]{corollary}
\declaretheorem[numberwithin=section,style=plain]{lemma}
\declaretheorem[numberwithin=section,style=plain]{definition}

\declaretheorem[numbered=no,style=plain]{remark}

\usepackage{enumitem}
\setlist{nolistsep}

\usepackage{multirow}

\usepackage{booktabs}           %
\usepackage{adjustbox} %
\usepackage{diagbox}

\makeatother

\usepackage{babel}

\renewcommand{\leq}{\leqslant}
\renewcommand{\geq}{\geqslant}
\renewcommand{\succeq}{\succcurlyeq}

\newtheorem{exmp}{Example}

\usepackage {tikz}
\usepackage{pgfkeys}
\usetikzlibrary{shapes,arrows}
\usepackage{pgfplots}
\pgfplotsset{compat=1.13}
\pgfplotsset{plotOptions/.style={
	legend pos=outer north east,
	legend style={draw=none},
	legend cell align={left},
	xlabel={iterations},
	xmin=-1,
	xmax=1000,
	width=.37\linewidth,
	height=.35\linewidth,
		label style={font=\scriptsize},
		legend style={font=\scriptsize},
		tick label style={font=\scriptsize},
		solid,
		very thick
	}}
\pgfplotsset{plotOptions2/.style={
	legend pos=outer north east,
	legend style={draw=none},
	legend cell align={left},
	xmin=0.05,
	xmax=0.85,
	width=0.75\textwidth,
	height=0.5\textwidth,
		label style={font=\scriptsize},
		legend style={font=\scriptsize},
		tick label style={font=\scriptsize},
				no markers,
		solid,
		ultra thick
	}}	
\pgfplotsset{plotOptions3/.style={
	legend pos=outer north east,
	legend style={draw=none},
	legend cell align={left},
	xmin=0,
	xmax=0.9,
	width=0.7\textwidth,
	height=0.5\textwidth,
		label style={font=\scriptsize},
		legend style={font=\scriptsize},
		tick label style={font=\scriptsize},
				no markers,
		solid,
		ultra thick
	}}	
	
\pgfplotsset{plotOptions4/.style={
	legend pos=outer north east,
	legend style={draw=none},
	legend cell align={left},
	xlabel={iterations},
	xmin=0,
	xmax=0.9,
	width=.37\linewidth,
	height=.35\linewidth,
		label style={font=\scriptsize},
		legend style={font=\scriptsize},
		tick label style={font=\scriptsize},
				no markers,
		solid,
		very thick
	}}	

 \pgfplotsset{plotOptions5/.style={
	legend pos=outer north east,
	legend style={draw=none},
	legend cell align={left},
	xmin=0.05,
	xmax=0.85,
	width=0.75\textwidth,
	height=0.5\textwidth,
		label style={font=\scriptsize},
		legend style={font=\scriptsize},
		tick label style={font=\scriptsize},
	}}	

  \pgfplotsset{plotOptions6/.style={
	legend pos=outer north east,
	legend style={draw=none},
	legend cell align={left},
	xmin=0,
	xmax=0.85,
	width=0.75\textwidth,
	height=0.5\textwidth,
		label style={font=\scriptsize},
		legend style={font=\scriptsize},
		tick label style={font=\scriptsize},
		no markers,
		ultra thick
	}}	

    \pgfplotsset{plotOptions7/.style={
	legend pos=outer north east,
	legend style={draw=none},
	legend cell align={left},
	xmin=0,
	xmax=1,
	width=0.75\textwidth,
	height=0.5\textwidth,
		label style={font=\scriptsize},
		legend style={font=\scriptsize},
		tick label style={font=\scriptsize},
		no markers,
		ultra thick
	}}	
	
\definecolor{colorP1}{RGB}{55,126,184}  %
\definecolor{colorP2}{RGB}{228,26,28}  %
\definecolor{colorP3}{RGB}{152,78,163} %
\definecolor{colorP4}{RGB}{77,175,74}  %
\definecolor{colorP5}{cmyk}{1.,1.,1,0}
\definecolor{colorP6}{RGB}{250, 150, 10} %

\usepgfplotslibrary{groupplots}
\usepgfplotslibrary{polar}
\usepgfplotslibrary{smithchart}
\usepgfplotslibrary{statistics}
\usepgfplotslibrary{dateplot}
\usepgfplotslibrary{ternary}
\usetikzlibrary{spy,backgrounds}

\newcommand{\tr}{\mathop{\bf tr}}
\newcommand{\rank}{\mathop{\bf rank}}
\newcommand{\cond}{q}

\newcommand{\appref}[1]{{Appendix~\ref{#1}}}

\newcommand{\argmin}[0]{\mathrm{argmin}}
\newcommand{\PRP}{PRP} 
\newcommand{\FR}{FR}
\newcommand{\GDEL}{GDEL}
\global\long\def\pg{d}%
\global\long\def\pgb{\mathbf{d}}%

\begin{document}

\title{
	\begin{minipage}{\textwidth}
		\fbox{%
			\parbox{\textwidth}{%
				
				{\small Published in Mathematical Programming Series A. DOI: \url{https://doi.org/10.1007/s10107-024-02127-7}}
			}%
		}
	\end{minipage} \\[10pt]
\textbf{\textsf{Nonlinear conjugate gradient methods: worst-case convergence rates via computer-assisted analyses}}
}

\author{Shuvomoy Das Gupta\footnote{Operations Research Center, Massachusetts Institute of Technology. Email: \textit{sdgupta@mit.edu}.}, Robert M. Freund\footnote{Sloan School of Management, Massachusetts Institute of Technology. Email: \textit{rfreund@mit.edu}.}, Xu Andy Sun\footnote{Sloan School of Management, Massachusetts Institute of Technology. Email: \textit{sunx@mit.edu}.}, Adrien Taylor\footnote{INRIA, École Normale Supérieure, CNRS, PSL Research University, Paris. \textit{adrien.taylor@inria.fr}.}}

\date{}
\maketitle

\begin{abstract}
We propose a computer-assisted approach to the analysis of the worst-case convergence of nonlinear conjugate gradient methods (NCGMs). Those methods are known for their generally good empirical performances for large-scale optimization, while having relatively incomplete analyses. Using our computer-assisted approach, we establish novel complexity bounds for the {Polak-Ribi\`ere-Polyak}~(PRP) and the Fletcher-Reeves~(FR) NCGMs for smooth strongly convex minimization. In particular, we construct mathematical proofs that establish the first
non-asymptotic convergence bound for {\FR} (which is historically the first developed NCGM), and a much improved non-asymptotic convergence bound for {\PRP}. Additionally, we provide simple adversarial examples on which these methods do not perform better than gradient descent with exact line search, leaving very little room for improvements on the same class of problems.

\end{abstract}

\section{Introduction}
We consider the standard unconstrained convex minimization problem 
\begin{equation}\label{eq:opt}
     f_\star\triangleq\min_{x\in\mathbb{R}^n} f(x),
\end{equation}
where $f$ is $L$-smooth (i.e., it has an $L$-Lipschitz gradient) and $\mu$-strongly convex. We study the worst-case performances of a few famous variants of \emph{nonlinear conjugate gradient methods} (NCGMs) for solving~\eqref{eq:opt}. More specifically, we study {Polak-Ribi\`ere-Polyak}~(PRP)~\cite{polyak1969conjugate,polak1969note} and Fletcher-Reeves~(FR)~\cite{fletcher1964function} schemes with exact line search. With exact line search, many other NCGMs such as the Hestenes and Stiefel method~\cite{hestenes1952methods}, the conjugate descent method due to Fletcher~\cite{fletcher1987practical}, and the Dai and Yuan method~\cite{dai1999nonlinear} reduce to either~{\PRP} or~{\FR}. Under exact line search,~{\PRP} and~{\FR} can be presented in the following compact form:
\begin{equation}
\begin{aligned} & \gamma_{k}\in\underset{ {\gamma \, \in \, \mathbb{R}}}{\mathrm{argmin}}\,f(x_{k}-\gamma\,d_{k}),\\
 & x_{k+1}=x_{k}-\gamma_{k}d_{k},\\
 & \beta_{k}=\frac{\|\nabla f(x_{k+1})\|^{2}-\eta\,\left\langle \nabla f(x_{k+1});\,\nabla f(x_{k})\right\rangle }{\|\nabla f(x_{k})\|^{2}},\\
 & d_{k+1}=\nabla f(x_{k+1})+\beta_{k}d_{k},
\end{aligned}
\tag{$\mathcal{M}$}\label{eq:NCG}
\end{equation}
where PRP and FR are respectively obtained by setting $\eta=1$ and $\eta=0$. NCGMs have a long history (see, e.g., the  survey~\cite{hager2006survey} and monograph~\cite{andrei2020nonlinear}), but are much less studied compared to their many first-order competitors. For instance, even though FR is generally considered the first NCGM \cite[\S 1]{hager2006survey}, we are not aware of non-asymptotic convergence results for it. On a similar note, some variants of NCGMs are known for their generally good empirical behaviors 
(which we illustrate in~\Cref{fig:log_reg}) 
with little of them being backed-up by classical complexity analyses.  
In this work, we apply the performance estimation approach~\cite{drori2014performance,taylor2017smooth} to~\eqref{eq:NCG} for filling this gap by explicitly computing some worst-case convergence properties of {\PRP} and {\FR} with exact line search. 
This work focuses on exact line search, as it is arguably the most logical starting point to understand the non-asymptotic convergence behavior of NCGMs. In certain cases, the minimizer associated with exact line search has an analytical form, while in others, it can be computed efficiently \cite[$\mathsection$9.7.1]{boyd2004convex}. However, in many practical implementations, inexact line searches are employed that try to either approximately minimize $f(x_{k}-\gamma\,d_{k})$ or even just reduce $f$ enough along the ray $x_{k}-\gamma d_{k}$. These inexact methods can be either monotone, which ensures a decrease in $f$ but converges slowly, or nonmonotone, which may allow faster convergence but risks nonrobust tuning \cite[$\mathsection$1.2]{andrei2020nonlinear}. Examples of notable monotone inexact line search schemes include backtracking \cite{armijo1966minimization}, Goldstein \cite{goldstein1965steepest}, and Wolfe line searches \cite{wolfe1969convergence, wolfe1971convergence}, and their variants \cite{hager2005new}. Nonmonotone schemes include \cite{grippo1986nonmonotone, zhang2004nonmonotone, huang2015nonmonotone} and many others; see \cite[pp.~10-14]{andrei2020nonlinear} for a brief review. Despite the computational differences between the two types of line searches, both aim to emulate the exact line search method. Consequently, when using an inexact line search process, any convergence guarantees—defined in terms of iteration numbers—are likely to be worse compared to the exact line search (neglecting the cost of performing exact line search).

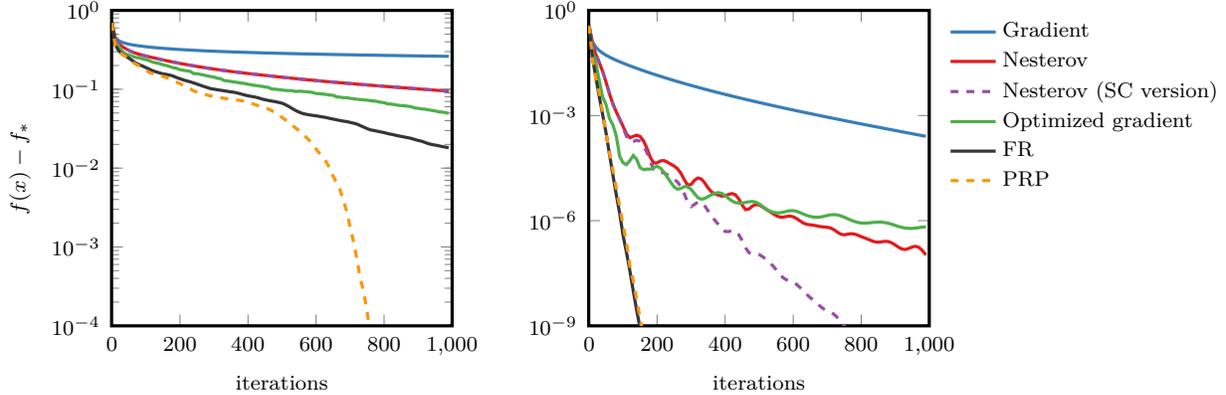
\begin{figure}[!ht]
    \centering
	    \begin{tabular}{cc}
		\begin{tikzpicture}
		\begin{semilogyaxis}[ylabel={$f(x)-f_*$}, plotOptions, ymin=1e-4, ymax=1]
		\addplot [color=colorP1] table [x=k,y=SteepestDescent] {NormalizedLogisticRegr_logregularizationInf.txt};
		\addplot [color=colorP2] table [x=k,y=AccGradient] {NormalizedLogisticRegr_logregularizationInf.txt};
		\addplot [color=colorP3, dashed] table [x=k,y=AccGradient_strCvx] {NormalizedLogisticRegr_logregularizationInf.txt};
		\addplot [color=colorP4] table [x=k,y=OGM] {NormalizedLogisticRegr_logregularizationInf.txt};
			
		\addplot [color=colorP5] table [x=k,y=NCG_FR] {NormalizedLogisticRegr_logregularizationInf.txt};
		\addplot [color=colorP6, dashed] table [x=k,y=NCG_PRP] {NormalizedLogisticRegr_logregularizationInf.txt};
		\end{semilogyaxis}
		\end{tikzpicture}
		&
		\begin{tikzpicture}
		\begin{semilogyaxis}[plotOptions, ymin=1e-9, ymax=1]
		\addplot [color=colorP1] table [x=k,y=SteepestDescent] {NormalizedLogisticRegr_logregularization4.txt};
		\addplot [color=colorP2] table [x=k,y=AccGradient] {NormalizedLogisticRegr_logregularization4.txt};
		\addplot [color=colorP3, dashed] table [x=k,y=AccGradient_strCvx] {NormalizedLogisticRegr_logregularization4.txt};
		\addplot [color=colorP4] table [x=k,y=OGM] {NormalizedLogisticRegr_logregularization4.txt};
		\addplot [color=colorP5] table [x=k,y=NCG_FR] {NormalizedLogisticRegr_logregularization4.txt};
		\addplot [color=colorP6, dashed] table [x=k,y=NCG_PRP] {NormalizedLogisticRegr_logregularization4.txt};

		\addlegendentry{Gradient}
		\addlegendentry{Nesterov}
		\addlegendentry{Nesterov (SC version)}
		\addlegendentry{Optimized gradient}
		\addlegendentry{FR}
		\addlegendentry{PRP}
		\end{semilogyaxis}
		\end{tikzpicture}
		\end{tabular}
    \caption{Convergence of a few first-order methods on a logistic regression problem on the small-sized Sonar dataset~\cite{gorman1988analysis}. Experiments with normalized features (zero mean and unit variance). Left:~without regularization. Right:~with an $\ell_2$ regularization of parameter $10^{-4}$. All methods were featured with an \textit{exact} line search ({performed numerically using the bisection method with a tolerance of $10^{-8}$}): (i)~gradient descent, (ii)~Nesterov's accelerated gradient~\cite{nesterov1983method} (exact line search instead of backtracking), (iii)~Nesterov's accelerated method for strongly convex problems, version~\cite[Algorithm 28]{d2021acceleration} with exact line search instead of the gradient step, (iv)~optimized gradient method~\cite[Algorithm  (OGM-LS)]{drori2020efficient}, (v)~{\FR}, and (vi)~{\PRP}. 
    }
    \label{fig:log_reg}
\end{figure}

\subsection{Contributions}

The contribution of this paper is twofold. First, we compute worst-case convergence bounds and counter-examples for {\PRP} and {\FR}. These bounds are obtained by formulating the problems of computing worst-case scenarios as nonconvex quadratically constrained quadratic optimization problems (QCQPs), and then by solving them to global optimality. Second, these computations enable us to construct mathematical proofs that establish an improved non-asymptotic convergence bound for {\PRP}, and, to the best of our knowledge, the first non-asymptotic convergence bound for~{\FR}. Furthermore, the worst-case bounds for {\PRP} and {\FR} obtained numerically reveal that there are simple adversarial examples on which these methods do not perform better than gradient descent with exact line search ({\GDEL}), leaving very little room for improvements on this class of problems. Since we demonstrate that the convergence results of NCGMs associated with exact line search are already disappointing, we conclude that inexact line searches, which approximate exact line search, are unlikely to offer improvement.

From a methodological point of view, our approach to computing worst-case scenarios and bounds through optimization is part of what is often referred to as \emph{performance estimation} framework. {This framework models the computation of the worst-case performance of a first-order method as an optimization problem itself; such optimization problems are called performance estimation problems~(PEP).}  While these PEPs are usually amenable to convex semidefinite programs~\cite{drori2014performance,taylor2017smooth,taylor2017exact}, this is generally not the case for \emph{adaptive} first-order methods such as {\PRP} and {\FR} \cite{barre2020complexity, barre2021worst}. To study these methods, we evaluate the worst-case performances of \eqref{eq:NCG} by solving nonconvex QCQPs, extending the standard SDP-based approach from~\cite{drori2014performance,taylor2017smooth,taylor2017exact} developed for non-adaptive methods. This contribution aligns with the spirit of~\cite{gupta2022branch}, developed for devising optimal (but non-adaptive) first-order methods.

\paragraph{Organization.} The paper is organized as follows. In Section~\ref{s:NCG_as_GD}, we establish non-asymptotic convergence
rates for {\PRP} and {\FR} by viewing the search direction $d_{k}$
in \eqref{eq:NCG} as an approximate gradient direction. In Section~\ref{sec:analysis}, we compute the exact numerical values of
the worst-case $\nicefrac{f(x_{N})-f_{\star}}{f(x_{0})-f_{\star}}$
and $\nicefrac{f(x_{k+N})-f_{\star}}{f(x_{k})-f_{\star}}$ for {\PRP}
and {\FR} by formulating the problems as nonconvex QCQPs and then
solving them to certifiable global optimality using a custom spatial branch-and-bound algorithm. The solutions to these QCQPs allow us to construct low-dimensional (dimension~$4$, observed a posteriori) counter-examples indicating that there is essentially no room for further improvement of the rates that we provide.  In Section \ref{app:ncg-pep-alg}, we discuss implementation details for solving
the nonconvex QCQPs in this paper.

\paragraph{Code.}

All the numerical results in this paper were obtained on the \texttt{MIT
Supercloud Computing Cluster} with Intel-Xeon-Platinum-8260 processor
with 48 cores and 128 GB of RAM running Ubuntu 18.04.6 LTS with Linux
4.14.250-llgrid-10ms kernel \cite{reuther2018interactive}. We used
\texttt{JuMP}---a domain specific modeling language for mathematical
optimization embedded in the open-source programming language \texttt{Julia}
\cite{JuMPDunningHuchetteLubin2017}---to model the optimization
problems. To solve the optimization problems, we use the following
solvers: \texttt{Mosek 9.3} \cite{mosek}, \texttt{KNITRO 13.0.0}
\cite{byrd2006k}, and {\texttt{Gurobi 10.0.0} \cite{Gurobi2024}}, which are free for
academic use. The relative feasibility tolerance and relative optimality tolerance of
all the solvers are set at $\textrm{1e-6}$. We validated the ``bad'' worst-case scenarios produced by our methodology using the \texttt{PEPit} package \cite{goujaud2022pepit},
which is an open-source \texttt{Python} library allowing to use the performance estimation problem (PEP) framework. %

The code used to generate and validate the results in this paper is available at: 
\begin{center}
\url{https://github.com/Shuvomoy/NCG-PEP-code}.
\end{center}

\subsection{Related works} 
Conjugate gradient (CG) methods are particularly popular choices for solving systems of linear equations and quadratic minimization problems; in this context, they are known to be information-optimal in the class of first-order methods~\cite[Chapter 12 \& Chapter 13]{nemirovski1994efficient} or~\cite[Chapter 5]{nemirovski1999optimization}. There are many extensions beyond quadratics, commonly referred to as \emph{nonlinear conjugate gradient methods} (NCGMs). They are discussed at length in the textbooks~\cite[Chapter 5 \& Chapter 7]{nocedal1999numerical} and~\cite[Chapter 5]{bonnans2006numerical} and in the nice survey~\cite{hager2006survey}. In particular, when exact line searches are used, many variants become equivalent and can be seen as instances of quasi-Newton methods, see~\cite[Chapter 7, \S``Relationship with conjugate gradient methods'']{nocedal1999numerical} or~\cite[Chapter 5,~\S5.5]{bonnans2006numerical}. For instance, it is well known that standard variants such as Hestenes-Stiefel~\cite{hestenes1952methods} and Dai-Yuan~\cite{dai1999nonlinear} are equivalent to~\eqref{eq:NCG} when exact line searches are used, while being different in the presence of more popular line search procedures (such as Wolfe's~\cite[Chapter 3]{nocedal1999numerical}). Beyond quadratics, obtaining convergence guarantees is often reduced to the problem of ensuring the search direction to be a descent direction, see for instance~\cite[\S 5.5 ``Extensions to non-quadratic problems'']{nemirovski1999optimization} or~\cite{al1985descent,hager2005new}. Without exact line searches, even when $f$ is strongly convex, there are counter-examples showing that even popular variants may not generate descent directions~\cite{dai1997analysis}. Note that NCGMs are often used together with restart strategies, which we do not consider here; see, e.g.,~\cite{royer2022nonlinear} and the references therein. Also, in~\cite[\S 5]{carmon2017convex}, the authors empirically demonstrate that NCGMs work very well in training deep learning models. 

In this work, we use the performance estimation framework, which models the computation of the worst-case performance of a first-order method as an optimization problem called PEP~\cite{drori2014performance,taylor2017smooth,taylor2017exact}. This PEP methodology is essentially mature for analyzing ``fixed-step'' (i.e., non-adaptive) first-order methods (and for methods whose analyses are amenable to those of fixed-step methods), whose stepsizes are essentially chosen in advance. This type of method includes many common first-order methods and operator splitting schemes, including the heavy-ball method~\cite{polyakintroduction} and Nesterov's accelerated gradient~\cite{nesterov1983method,beck2009fast}. Only very few adaptive methods were studied using the PEP methodology, namely gradient descent with exact line searches~\cite{de2017worst}, greedy first-order methods~\cite{drori2020efficient}, and Polyak stepsizes~\cite{barre2020complexity}. A premise to the study of NCGMs using PEPs was done in~\cite[\S 4.5.2]{barre2021worst}, where an upper bound on the worst-case $\nicefrac{(f(x_2) - f_\star)}{(f(x_0)-f_\star)}$ of {\FR} was numerically computed for two iterations and two condition number values, $q = 0.1$ and $q=0.01$, where $q\triangleq \nicefrac{\mu}{L}$. This was achieved by numerically solving an SDP relaxation through a grid search on $\beta_k$. In comparison, we compute the worst-case $\nicefrac{(f(x_N) - f_\star)}{(f(x_0)-f_\star)}$ by solving the nonconvex PEPs associated with both {\FR} and {\PRP} to global optimality across a broader range of condition numbers over $q \in [0,1]$ for $N=1,2,3,4$. Furthermore, for both methods, we also compute ``Lyapunov''-type bounds on $\nicefrac{(f(x_{k+N}) - f_\star)}{(f(x_{k})-f_\star)}$ that holds for any $k$ for $N=1,2,3,4$, and also establish their analytical complexity bounds offering a more comprehensive understanding of their performance. {We limit  our numerical experiments to $N \in \{1, 2, 3, 4\}$ for the following reasons. One reason is primarily computational: solving the underlying QCQPs to global optimality in a reasonable amount of time for $N = 3, 4$ is already quite challenging as the number of nonconvex constraints grows quadratically with $N$. To account for this numerical bottleneck,  we consider two types of complementary bounds: (i) ``Lyapunov'' bound~$\nicefrac{(f(x_{k+N}) - f_\star)}{(f(x_{k})-f_\star)}$ and (ii) bound that incorporates the initial condition~$\nicefrac{(f(x_N) - f_\star)}{(f(x_0)-f_\star)}$, where both bounds provide similar conclusions. Considering the computational challenges and also similar conclusions from two different but complementary bounds, we accepted to stop at $N=4$. } Our work is also closely related in spirit with the technique developed in~\cite{gupta2022branch} for optimizing coefficients of fixed-step first-order methods using nonconvex optimization.

\subsection{Preliminaries}

In this section, we recall the definition and
a result on smooth strongly convex functions, as well as a base result
on steepest descent with an exact line search.

\paragraph{Properties of smooth strongly convex functions.}

We use the standard notation $\langle\,\cdot\,;\,\cdot\,\rangle:\mathbb{R}^{n}\times\mathbb{R}^{n}\rightarrow\mathbb{R}$
to denote the Euclidean inner product, and the corresponding induced
Euclidean norm $\|\cdot\|$. The class of $L$-smooth $\mu$-strongly
convex functions is standard and can be defined as follows.

\begin{definition}Let
$f:\mathbb{R}^{n}\rightarrow\mathbb{R}$ be a proper, closed, and
{differentiable} convex function, and consider two constants $0\leq\mu<L<\infty$.
The function $f$ is $L$-smooth and $\mu$-strongly convex (notation
$f\in\mathcal{F}_{\mu,L}(\mathbb{R}^{n})$), if 
\begin{itemize}
\item ($L$-smooth) {for all $x,y\in\mathbb{R}^{n}$, it
holds that $\|\nabla f(x) -\nabla f(y)\| \leq L \|x-y\|$,}  
\item ($\mu$-strongly convex) for all $x,y\in\mathbb{R}^{n}$,
it holds that $f(x)\geq f(y)+\langle\nabla f(y);x-y\rangle+\frac{\mu}{2}\|x-y\|^{2}$. 
\end{itemize}

We simply denote $f\in\mathcal{F}_{\mu,L}$ when the
dimension is either clear from the context or unspecified. We also
denote by $\cond\triangleq\frac{\mu}{L}$ the inverse condition number.
For readability, we do not explicitly treat the (trivial) case $L=\mu$.
\end{definition} Smooth strongly convex functions satisfy many inequalities,
see e.g.,~\cite[Theorem 2.1.5]{nest-book-04}. For the developments
below, we need only one specific inequality characterizing functions
in $\mathcal{F}_{\mu,L}$. The following result can be found in~\cite[Theorem 4]{taylor2017smooth} and is key in our analysis. 

\begin{theorem}\cite[Theorem 4, $\mathcal{F}_{\mu,L}$-interpolation]{taylor2017smooth}\label{thm:interp}
Let $I$ be an index set and $S=\{(x_{i},g_{i},f_{i})\}_{i\in I}\subseteq\mathbb{R}^{n}\times\mathbb{R}^{n}\times\mathbb{R}$
be a set of triplets. There exists $f\in\mathcal{F}_{\mu,L}$ satisfying
$f(x_{i})=f_{i}$ and $\nabla f(x_{i})=g_{i}$ for all $i\in I$ if
and only if {
\begin{equation}
f_{i}\geq f_{j}+\langle g_{j};\,x_{i}-x_{j}\rangle+\frac{1}{2\left(1-\frac{\mu}{L}\right)}\left(\frac{1}{L}\|g_{i}-g_{j}\|^{2}+\mu\|x_{i}-x_{j}\|^{2}-2\frac{\mu}{L}\left\langle g_{i}-g_{j};\,x_{i}-x_{j}\right\rangle \right)\label{eq:interp}
\end{equation}}
holds for all $i,j\in I$. \end{theorem}

Another related result from \cite[\S 2.1]{drori2022oracle}
that we record next involves constructing a smooth and strongly convex function from a given set of triplets. In this theorem, number of elements of index set $I$ is denoted by $|I|$.

\begin{theorem}\cite[\S 2.1, smooth and strongly convex extension]{drori2022oracle}\label{thm:scvx-extension}
Suppose $I$ is a set of indices and $S=\{(x_{i},g_{i},f_{i})\}_{i\in I}\subseteq\mathbb{R}^{n}\times\mathbb{R}^{n}\times\mathbb{R}$
is a set of triplets such that \eqref{eq:interp} holds for all $i,j\in I$
for some $0\leq\mu<L<\infty$. Then the function $f:\mathbb{R}^{n}\to\mathbb{R}$
defined by 
\begin{align}
f(y)= & \max_{\alpha\in\Delta}\Big[\frac{L}{2}\|y\|^{2}-\frac{L-\mu}{2}\|y-\frac{1}{L-\mu}\sum_{i\in I}\alpha_{i}(g_{i}-\mu x_{i})\|^{2}\nonumber \\
 & \quad+\sum_{i\in I}\alpha_{i}\Big(f_{i}+\frac{1}{2(L-\mu)}\|g_{i}-Lx_{i}\|^{2}-\frac{L}{2}\|x_{i}\|^{2}\Big)\Big]\label{eq:scvx-extension}
\end{align}
where { $\Delta=\{\alpha\in\mathbb{R}^{|I|}\mid\alpha\geq0,\,\sum_{i=1}^{n}\alpha_{i}=1\},$}
satisfies $f\in\mathcal{F}_{\mu,L}(\mathbb{R}^{n})$, $f(x_{i})=f_{i}$
and $\nabla f(x_{i})=g_{i}$ for all~$i\in I$. \end{theorem}

\paragraph{Approximate steepest descent method.}

Consider a function $f\in\mathcal{F}_{\mu,L}$ and the approximate
steepest descent method: 
\begin{equation}
\begin{aligned} & \gamma_{k}=\underset{\gamma \, \in \, \mathbb{R}}{\argmin}\,f(x_{k}-\gamma d_{k}),\\
 & x_{k+1}=x_{k}-\gamma_{k}d_{k},
\end{aligned}
\tag{\ensuremath{\mathcal{ASD}}}\label{eq:GD_ELS}
\end{equation}
where the search direction $d_{k}$ satisfies a relative error criterion:
\begin{equation}
\|d_{k}-\nabla f(x_{k})\|\leq\epsilon\|\nabla f(x_{k})\| \textup{ where }\epsilon\in[0,1). \tag{$\mathcal{REC}$}\label{eq:rel_dir}
\end{equation}

Note that the relative tolerance $\epsilon$ needs to satisfy $\epsilon \in [0,1)$ for \eqref{eq:GD_ELS} to converge. If $\epsilon \geq 1$, then $d_k = 0$ becomes feasible and \eqref{eq:GD_ELS} cannot be guaranteed to converge anymore. In practice, this means that we can pick $d_k$ to be orthogonal to $\nabla f(x_k)$ \cite[$\mathsection$5]{de2017worst}.

The iterates of \eqref{eq:GD_ELS} satisfies the following two necessary
(weaker) conditions for $x_{k+1}$ to follow~\eqref{eq:GD_ELS}:
\begin{equation}
\begin{aligned} & \left\langle \nabla f(x_{k+1});\,d_{k}\right\rangle =0,\\
 & \left\langle \nabla f(x_{k+1});\,x_{k}-x_{k+1}\right\rangle =0,
\end{aligned}
\tag{\ensuremath{\mathcal{ASD}_{\textup{relaxed}}}}\label{eq:relaxed-line-search-cond-1}
\end{equation}
where the first condition follows from optimality of $\gamma_{k}$
in the line search condition in~\eqref{eq:GD_ELS} as follows 
\begin{align}
0 & =\left[\nabla_{\gamma}f(x_{k}-\gamma\pg_{k})\right]_{\gamma=\gamma_{k}}\nonumber \\
 & =-\left\langle \nabla f(x_{k}-\gamma_{k}\pg_{k});\,\pg_{k}\right\rangle \nonumber \\
 & =-\left\langle \nabla f(x_{k+1});\,\pg_{k}\right\rangle ,\label{eq:pseudo-orthogonality-proof-ASD}
\end{align}
and the second condition comes from putting $d_{k}=\nicefrac{(x_{k}-x_{k+1})}{\gamma_{k}}$
in \eqref{eq:pseudo-orthogonality-proof-ASD}.

\paragraph{Convergence of approximate steepest descent method.}

We will use the following
result in our analysis. Note that similar results (without line searches)
to that of Theorem~\ref{thm:approx_GD} can be found in~\cite{de2020worst},
which might help in future analyses of NCGMs without exact line searches.

\begin{theorem}[{\cite[Theorem 5.1]{de2017worst}}]
\label{thm:approx_GD} Let $f\in\mathcal{F}_{\mu,L}(\mathbb{R}^{n})$,
$x_{\star}\triangleq\argmin_{x\in\mathbb{R}^{n}}f(x)$ be {the} minimizer
of $f$, and $f_{\star}\triangleq f(x_{\star})$. For any $x_{k}\in\mathbb{R}^{n}$,
search direction $d_{k}\in\mathbb{R}^{n}$, and $x_{k+1}\in\mathbb{R}^{n}$
computed using \eqref{eq:GD_ELS} such that the relative error criterion~\eqref{eq:rel_dir}
holds, we have 
\begin{equation}
f(x_{k+1})-f_{\star}\leq\left(\frac{1-\cond_{\epsilon}}{1+\cond_{\epsilon}}\right)^{2}\left(f(x_{k})-f_{\star}\right),\label{eq:ELS_inex_init}
\end{equation}
where $\cond_{\epsilon}\triangleq\nicefrac{\mu(1-\epsilon)}{L(1+\epsilon)}$. 
\end{theorem}

{Next, we show that the relative error criterion~\eqref{eq:rel_dir}
can be interpreted in simple geometric fashion in the context of exact
line searches in~\eqref{eq:GD_ELS}. As \eqref{eq:GD_ELS}
uses exact line search, it is the direction of $d_{k}$ that influences
the convergence, not the magnitude. Scaling the magnitude of
$d_{k}$ by a suitable nonzero factor $\alpha=\nicefrac{\left\langle \nabla f(x_{k});\,d_{k}\right\rangle }{\|d_{k}\|^{2}}$, i.e.,  the scaled direction being $\alpha d_k$, leads to $\gamma_{k}$ getting scaled to
$\gamma_{k}/\alpha$, but this scaling does not alter $x_{k},x_{k+1}$ and we have $|\sin \angle (\nabla f(x_k), \alpha d_k)| = |\sin \angle (\nabla f(x_k), d_k)|$ (details in the proof to Corollary \ref{thm:approx_GD_sine-1} below); here we have used the notation $\angle (a,b)$ to denote the angle between two vectors $a,b$.
Hence, by appropriate scaling of search direction to $\alpha d_{k}$,
we can ensure that $|\sin \angle (\nabla f(x_k), \alpha d_k)|=\nicefrac{\|\alpha d_{k}-\nabla f(x_{k})\|}{\|\nabla f(x_{k})\|}$. Since the iterates remain the same, under the angle condition,
we will have the same convergence guarantees that hold for \eqref{eq:rel_dir} in Theorem \ref{thm:approx_GD}. Hence, ~\eqref{eq:rel_dir} implies $|\sin\theta_{k}|\leq\epsilon$. On the other hand, if $|\sin\theta_{k}|\leq\epsilon$ holds, then there exists a scaling of the search direction such that \eqref{eq:rel_dir} also holds with the same parameter $\epsilon$.   In short,~\eqref{eq:rel_dir}
is equivalent to $|\sin\theta_{k}|\leq\epsilon$ in the context of
exact line search used in \eqref{eq:GD_ELS}, which we detail in Corollary \ref{thm:approx_GD_sine-1} below. }

\begin{figure}[!ht]
\centering %
\includegraphics[width=0.6\textwidth]{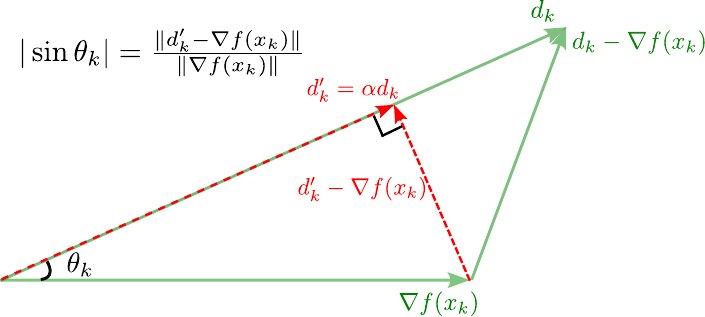}
\caption{This figure illustrates how for any $x_{k}\in\mathbb{R}^{n}$, search
direction $d_{k}\in\mathbb{R}^{n}$, and $x_{k+1}\in\mathbb{R}^{n}$
satisfying \eqref{eq:GD_ELS}, one can scale $d_{k}$ appropriately
without altering $x_{k},x_{k+1}$ such that the scaled search direction
$d_{k}^{\prime} = \alpha d_k$ with $\alpha = \nicefrac{\left\langle \nabla f(x_{k});\,d_{k}\right\rangle }{\|d_{k}\|^{2}}$ ensures $|\sin\theta_{k}|=\nicefrac{\|d_{k}^{\prime}-\nabla f(x_{k})\|}{\|\nabla f(x_{k})\|}$
with $\theta_{k}$ being the angle between $\nabla f(x_{k})$ and
$d_{k}$.  \label{fig:scaling_dk}}
\end{figure}

\begin{corollary}
\label{thm:approx_GD_sine-1} Let $f\in\mathcal{F}_{\mu,L}(\mathbb{R}^{n})$,
$x_{\star}\triangleq\argmin_{x\in\mathbb{R}^{n}}f(x)$ be the minimizer of $f$, and $f_{\star}\triangleq f(x_{\star})$. Consider any $x_{k}\in\mathbb{R}^{n}$,
search direction $d_{k}\in\mathbb{R}^{n}$, and $x_{k+1}\in\mathbb{R}^{n}$
computed using \eqref{eq:GD_ELS} such that $|\sin\theta_{k}|\leq\epsilon$
with $\theta_{k}$ being the angle between $\nabla f(x_{k})$ and
$d_{k}$ and $\epsilon\in[0,1)$. Then we have 
\begin{equation}
f(x_{k+1})-f_{\star}\leq\left(\frac{1-\cond_{\epsilon}}{1+\cond_{\epsilon}}\right)^{2}\left(f(x_{k})-f_{\star}\right),\label{eq:ELS_inex}
\end{equation}
where $\cond_{\epsilon}\triangleq\nicefrac{\mu(1-\epsilon)}{L(1+\epsilon)}$. 
\end{corollary}

\begin{proof}

{
The proof sketch is as follows. As \eqref{eq:GD_ELS} uses exact line
search, it is only the direction of $d_{k}$ that influences the convergence, not its
magnitude. Hence, we can appropriately scale the search
direction $d_{k}$ to $d_{k}^{\prime}=\alpha d_{k}$ (with $\alpha=\nicefrac{\left\langle \nabla f(x_{k});\,d_{k}\right\rangle }{\|d_{k}\|^{2}} \neq0$) so that the algorithm iterates remain the same and we can ensure $|\sin \theta_k| = |\sin \angle (\nabla f(x_k), d_k^\prime)| =\nicefrac{\|d_{k}^{\prime}-\nabla f(x_{k})\|}{\|\nabla f(x_{k})\|}$.
Since the iterates remain the same, under the angle condition $|\sin\theta_{k}|\leq\epsilon$
we have the same convergence guarantees that hold for \eqref{eq:rel_dir} in Theorem \ref{thm:approx_GD}.

 }

Now we start the proof earnestly. Consider the following method, where
the search direction $d_{k}$ in \eqref{eq:GD_ELS} is scaled by some
factor $\alpha\neq0$ with the scaled search direction denoted by
$d_{k}^{\prime}=\alpha d_{k}$: 

\begin{equation}
\begin{aligned} & \gamma_{k}^{\prime}=\underset{\gamma^{\prime}}{\argmin}\,f(x_{k}-\gamma^{\prime}d_{k}^{\prime}),\\
 & x_{k+1}^{\prime}=x_{k}-\gamma_{k}^{\prime}d_{k}^{\prime},
\end{aligned}
\tag{\ensuremath{\mathcal{ASD}_{\textup{scaled}}}}\label{eq:GD_ELS-1}
\end{equation}
and we denote $\theta_{k}^{\prime}$ to be the angle between
$\nabla f(x_{k})$ and $d_{k}^{\prime}$. We now show that \eqref{eq:GD_ELS} and \eqref{eq:GD_ELS-1} are\emph{
equivalent in the sense that they generate an identical sequence of
iterates $x_{k},x_{k+1}$ along with $|\sin \theta_k| = |\sin \theta_k^\prime|$}. This is so because
\[
\gamma_{k}^{\prime}=\underset{\gamma^{\prime}}{\argmin}\,f(x_{k}-\gamma^{\prime}d_{k}^{\prime})=\underset{\gamma}{\argmin}\,f\left(x_{k}-\frac{\gamma_{k}}{\alpha}\cdot\alpha d_{k}\right)=\frac{\gamma_{k}}{\alpha},
\]
 i.e., the optimal stepsize $\gamma_{k}^{\prime}$ in \eqref{eq:GD_ELS-1}
is the optimal step-size $\gamma_{k}$ in \ref{eq:GD_ELS} scaled
by $\nicefrac{1}{\alpha}$, leading to
\[
x_{k+1}^{\prime}=x_{k}-\gamma_{k}^{\prime}d_{k}^{\prime}=x_{k}-\gamma_{k}d_{k}=x_{k+1}.
\]
Finally,

\begin{align*}|\sin\theta_{k}^{\prime}| & =\sqrt{1-\cos^{2}\theta_{k}^{\prime}}\\
 & =\sqrt{1-\frac{\left\langle \nabla f(x_{k});\,d_{k}^{\prime}\right\rangle ^{2}}{\|\nabla f(x_{k})\|^{2}\|d_{k}^{\prime}\|^{2}}}\\
 & =\sqrt{1-\frac{\cancel{\alpha}^{2}\left\langle \nabla f(x_{k});\,d_{k}\right\rangle ^{2}}{\cancel{\alpha}^{2}\|\nabla f(x_{k})\|^{2}\|d_{k}\|^{2}}}\\
 & =\sqrt{1-\cos^{2}\theta_{k}}\\
 & =|\sin\theta_{k}|.
\end{align*}

Hence to establish our convergence result \eqref{eq:ELS_inex}, we
can work with \eqref{eq:GD_ELS-1}. Next, we carefully select a nonzero
$\alpha$ that ensures $\left\langle d_{k}^{\prime};\,d_{k}^{\prime}-\nabla f(x_{k})\right\rangle =0$,
i.e., $d_{k}^{\prime}-\nabla f(x_{k})$ would be perpendicular to
$d_{k}^{\prime}$ (see Figure~\ref{fig:scaling_dk}); this yields $\alpha=\nicefrac{\left\langle \nabla f(x_{k});\,d_{k}\right\rangle }{\|d_{k}\|^{2}}$
which is nonzero because $\epsilon\in[0,1)$ implies $\left\langle \nabla f(x_{k});\,d_{k}\right\rangle \neq0$.
For this value of $\alpha$, we have $|\sin\theta_{k}|=\nicefrac{\|d_{k}^{\prime}-\nabla f(x_{k})\|}{\|\nabla f(x_{k})\|}$, which can be shown geometrically  in Figure
\ref{fig:scaling_dk} in the right triangle (colored {\color{red}{red}}) involving $\nabla f(x_{k})$,
$d_{k}^{\prime}$, and $d_{k}^{\prime}-\nabla f(x_{k})$. 

Now we are given that $|\sin\theta_{k}|\leq\epsilon$,
hence setting $\alpha=\nicefrac{\left\langle \nabla f(x_{k});\,d_{k}\right\rangle }{\|d_{k}\|^{2}}$ ensures
that the relative error criterion $\nicefrac{\|d_{k}^{\prime}-\nabla f(x_{k})\|}{\|\nabla f(x_{k})\|}\leq\epsilon$
is satisfied for \eqref{eq:GD_ELS-1}. Finally by applying Theorem \ref{thm:approx_GD} to  \eqref{eq:GD_ELS-1}, we arrive at \eqref{eq:ELS_inex}.
\end{proof}

\section{Base descent properties of NCGMs\label{s:NCG_as_GD}}

In this section, we analyze NCGMs as approximate steepest descent methods satisfying \eqref{eq:GD_ELS} through a computer-assisted approach, where only the generated search directions matter,
and not their magnitudes. This renders the analysis somewhat simpler,
and we argue that this is a reasonable setting for improving the analysis
and understanding of NCGMs.

This section builds on the intuition that when $|\sin\theta_{k}|$, where $\theta_{k}$ is the angle between the gradient and the search
direction $d_k$ at iteration $k$, is upper bounded in an appropriate fashion, one can use  Theorem \ref{thm:approx_GD} for obtaining convergence guarantees. In particular, we get nontrivial convergence guarantees as soon as $\theta_{k}$ can be bounded away from~$\pm\frac{\pi}{2}$, i.e.,
$\sin\theta_{k}$ should be bounded away from $1$ for ensuring that
$d_{k}$'s are descent directions. Of course, viewing NCGMs as approximate steepest descent methods is adversarial by nature, as it misses the point
that the directions of NCGMs are meant to be better than those of
vanilla gradient descent, while such analyses can only provide worse
rates. Additionally, in {Section \ref{subsec:PEP-reason-for-NCGM-as-steep-GD}},
we provide additional justification behind analyzing NCGMs as approximate
steepest descent methods through the lens of performance estimation
problem (PEP), where we formulate the process of computing the worst-case
$\nicefrac{f(x_{k+1})-f_{\star}}{f(x_{k})-f_{\star}}$ as optimization
problems. 

Albeit being pessimistic by construction, the analyses of this section
are, to the best of our knowledge, novel for {\FR} (for which we provide the first non-asymptotic convergence bound) and significantly better than the state-of-the-art
bound for {\PRP}. Furthermore, we show in Section~\ref{sec:num-res-prp}
and Section~\ref{subsec:Numerical-results-for-FR} that there is
actually nearly no room for improving those analyses.

\paragraph{Properties of NCGMs with exact line search.}

Before going into the detailed approach, let us review
a few properties of the iterates of~\eqref{eq:NCG}. Note that the
iterates of \eqref{eq:NCG} satisfy the following equalities: 
\begin{equation}
\begin{aligned} & \left\langle \nabla f(x_{k+1});\,d_{k}\right\rangle =0,\\
 & \left\langle \nabla f(x_{k+1});\,x_{k}-x_{k+1}\right\rangle =0,\\
 & \left\langle \nabla f(x_{k});\,\pg_{k}\right\rangle =\|\nabla f(x_{k})\|^{2},
\end{aligned}
\label{eq:relaxed-line-search-cond}
\end{equation}
where the first two equalities are the same as \eqref{eq:relaxed-line-search-cond-1}
following from exact line search. The last equality in \eqref{eq:relaxed-line-search-cond}
follows from applying the first equality to
\begin{equation}
\left\langle \nabla f(x_{k});\,\pg_{k}\right\rangle =\left\langle \nabla f(x_{k});\,\nabla f(x_{k})+\beta_{k-1}\pg_{k-1}\right\rangle =\|\nabla f(x_{k})\|^{2}.\label{eq:NCG_eq1}
\end{equation}

Combining~\eqref{eq:NCG_eq1} with
$\langle\nabla f(x_{k});d_{k}\rangle=\|\nabla f(x_{k})\|\|d_{k}\|\cos\theta_{k}$,
we obtain that $\nicefrac{\|\nabla f(x_{k})\|}{\|d_{k}\|}=\cos\theta_{k}$,
thereby reaching $\sin^{2}\theta_{k}=1-\nicefrac{\|\nabla f(x_{k})\|^{2}}{\|d_{k}\|^{2}}$.
If we have $\nicefrac{\|d_{k}\|^{2}}{\|\nabla f(x_{k})\|^{2}}\leq c$
($c\geq1$ due to \eqref{eq:relaxed-line-search-cond}), then $\sin^{2}\theta_{k}=1-\left(\nicefrac{\|\nabla f(x_{k})\|^{2}}{\|d_{k}\|^{2}}\right)\leq1-\left(\nicefrac{1}{c}\right)$,
yielding 

\begin{equation}
    |\sin\theta_{k}|\leq\sqrt{1-\nicefrac{1}{c}}. \label{eq:explicit_bound}
\end{equation}

The first two equations of \eqref{eq:relaxed-line-search-cond}, in conjunction with \eqref{eq:explicit_bound}, satisfied by NCGMs, correspond to the same set of conditions required to apply Theorem \ref{thm:approx_GD}. Thus, if we can establish an upper bound for the ratio $\nicefrac{\|d_{k}\|}{||\nabla f(x_{k})||}$ in the context of NCGMs, we can translate this into their worst-case convergence rates using Theorem \ref{thm:approx_GD}.

\paragraph{Section organization.}

In {Section \ref{subsec:PEP-reason-for-NCGM-as-steep-GD}},
we provide PEP-based perspective behind analyzing NCGMs as methods satisfying \eqref{eq:GD_ELS}. Section~\ref{sec:worst-case-search-dir},
first frames the problems of computing the worst-case $\nicefrac{\|d_{k}\|}{\|\nabla f(x_{k})\|}$
for PRP and FR as optimization problems for obtaining the desired
bounds measuring the quality of the angle~$\theta_{k}$ as PEPs. These PEPs are
nonconvex but practically tractable QCQPs and can be solved numerically
to certifiable global optimality using spatial branch-and-bound algorithms
(detailed in Section \ref{app:ncg-pep-alg}), which allows (i)~to construct
``bad'' functions serving as counter-examples
on which the worst-case $\nicefrac{\|d_{k}\|}{\|\nabla f(x_{k})\|}$
for {\PRP} and {\FR} is achieved, and (ii)~to identify closed-form
solutions to the PEPs leading to proofs that can be verified in a
standard and mathematically rigorous way. The convergence rates for
PRP and FR are provided and proved in~Section~\ref{s:NCGs_inexact_GD_results}.

\subsection{A PEP perspective behind viewing NCGMs as approximate steepest descent method \label{subsec:PEP-reason-for-NCGM-as-steep-GD}}

In this section, we provide a PEP-based perspective
behind analyzing NCGMs as approximate steepest descent methods satisfying \eqref{eq:GD_ELS}, i.e., we formulate the
problems of computing the worst-case ratios of $\nicefrac{f(x_{k+1})-f_{\star}}{f(x_{k})-f_{\star}}$
as the following optimization problem:

\begin{equation}
\left(\begin{array}{ll}
\underset{\substack{f, \, x_{k}, \, x_{k+1}, \, d_{k}, \, d_{k+1},\\
\gamma_{k}, \, \beta_{k}, \, n 
}
}{\mbox{maximize}} & \frac{f(x_{k+1})-f_{\star}}{f(x_{k})-f_{\star}}\\
\textrm{subject to} & n\in\mathbb{N},\,f\in\mathcal{F}_{\mu,L}(\mathbb{R}^{n}),\,d_{k},x_{k}\in\mathbb{R}^{n},\\
 & \left\langle \nabla f(x_{k});\,\pg_{k}\right\rangle =\|\nabla f(x_{k})\|^{2},\\
 & \|d_{k}\|^{2}\leq c\|\nabla f(x_{k})\|^{2},\\
 & (x_{k+1},d_{k+1},\beta_{k})\text{ generated by~\eqref{eq:NCG} from }x_{k}\text{ and }d_{k}.
\end{array}\right)\label{eq:PEP_rho_1}
\end{equation}
In Section \ref{subsec:NCG-PEP-ift-formulation}, we will illustrate
how we can formulate and solve~\eqref{eq:PEP_rho_1} by casting it as
a nonconvex~QCQP. Note that in~\eqref{eq:PEP_rho_1}, the second
constraint corresponds to third equation of~\eqref{eq:relaxed-line-search-cond}
and the third constraint $\|d_{k}\|^{2}\leq c\|\nabla f(x_{k})\|^{2}$
models that if $\nabla f(x_{k})=0$ then $d_{k}=0$ for~\eqref{eq:NCG}.
Note that $\nicefrac{\|d_{k}\|^{2}}{\|\nabla f(x_{k})\|^{2}} \geq 1$ because $\|\nabla f(x_{k})\|^{2}\leq\|d_{k}\|^{2}$, which follows from applying Cauchy--Schwarz inequality to \eqref{eq:NCG_eq1}.

While solving the nonconvex QCQPs equivalent to \eqref{eq:PEP_rho_1} for different values of $c$, $\mu$, and $L$, we found that the worst-case $\nicefrac{f(x_{k+1})-f_{\star}}{f(x_{k})-f_{\star}}$ is strictly monotonically increasing in $c$. Naturally, assigning an arbitrary value to $c$ would not be reasonable to get the best bound, because the search direction generated by \eqref{eq:NCG} may not admit such a value. For example, for {\PRP}, $c$ is always upper bounded by $1+\nicefrac{L^2}{\mu^2}$ as $\nicefrac{\|d_{k}\|^{2}}{\|\nabla f(x_{k})\|^{2}}\leq1+\nicefrac{L^2}{\mu^2}$ for {\PRP} \cite[Theorem 2]{polyak1969conjugate}. As we are interested in obtaining the tightest upper bound on $\nicefrac{f(x_{k+1})-f_{\star}}{f(x_{k})-f_{\star}}$, the natural question is: What is the smallest admissible value of $c$, i.e., what is the least upper bound on the ratio $\nicefrac{\|d_{k}\|^{2}}{\|\nabla f(x_{k})\|^{2}}$ generated by \eqref{eq:NCG}? To that end, we numerically computed  the least upper bound on $c$ by solving a problem similar to \eqref{eq:PEP_rho_1}, except we replaced the objective $\nicefrac{f(x_{k+1})-f_{\star}}{f(x_{k})-f_{\star}}$ with $\nicefrac{\|d_{k+1}\|^{2}}{\|\nabla f(x_{k+1})\|^{2}}$ and then replaced the indices $k,k+1$ with $k-1,k$, respectively. In Section~\ref{sec:worst-case-search-dir}, we provide the details on formulating the problems of computing the worst-case ratios of $\nicefrac{\|d_{k}\|^{2}}{\|\nabla f(x_{k})\|^{2}}$ as nonconvex QCQPs. After we computed the least upper bound on $c$ numerically, we put them in \eqref{eq:PEP_rho_1}. We then solved the associated nonconvex QCQP to global optimality, which numerically provided us with the tightest upper bound on worst-case $\nicefrac{f(x_{k+1})-f_{\star}}{f(x_{k})-f_{\star}}$. Remarkably, at this stage, we found that these numerically computed worst-case $\nicefrac{f(x_{k+1})-f_{\star}}{f(x_{k})-f_{\star}}$ for \eqref{eq:NCG} exactly matched the analytical bound prescribed in \Cref{thm:approx_GD_sine-1}. This observation provides us a justification for analyzing NCGMs as approximate steepest descent methods.

\subsection{Computing worst-case search directions}\label{sec:worst-case-search-dir}

In this section, we formulate the problems of computing the worst-case
ratios of $\nicefrac{\|d_{k}\|}{\|\nabla f(x_{k})\|}$. Following the classical steps introduced in~\cite{taylor2017smooth,taylor2017exact},
we show that it can be cast as a nonconvex~QCQP. 
 
For doing that, we assume that
at iteration $k-1$ the NCGM has not reached optimality, so $\nabla f(x_{k-1})\neq0.$
Because $\|\nabla f(x_{k-1})\|^{2}\leq\|d_{k-1}\|^{2}$ (follows from
applying Cauchy--Schwarz inequality to \eqref{eq:NCG_eq1}), without
loss of generality we define the ratio $c_{k-1}\triangleq \nicefrac{\|d_{k-1}\|^{2}}{\|\nabla f(x_{k-1})\|^{2}}$ where $c_{k-1}\geq1$. Then, denoting by $c_{k}$ the worst-case ratio $\nicefrac{\|d_{k}\|^{2}}{\|\nabla f(x_{k})\|^{2}}$
arising when applying~\eqref{eq:NCG} to the minimization of an $L$-smooth
$\mu$-strongly convex function, we will compute $c_k$ as a function of $L$, $\mu$, and $c_{k-1}$. In other words, we use a \emph{Lyapunov}-type point of view and take the stand of somewhat \emph{forgetting} about how $d_{k-1}$ was generated (except through the fact that it satisfies~\eqref{eq:relaxed-line-search-cond}). Then, we compute the worst possible next search direction $d_k$ that the algorithm could generate given that $d_{k-1}$ satisfies a certain quality. Thereby, we obtain an upper bound on the evolution of the \emph{quality} of the search directions (quantified by $c_k$) obtained throughout the iterative procedure. Formally, we compute
\begin{equation}
c_{k}(\mu,L,c_{k-1})\triangleq\left(\begin{array}{ll}
\underset{\substack{f, \, x_{k-1}, \, d_{k-1}, \\ x_k, \, d_k, \, \beta_{k-1}, \, n} }{\mbox{maximize}} & \frac{\|d_{k}\|^{2}}{\|\nabla f(x_{k})\|^{2}}\\
\textrm{subject to} & n\in\mathbb{N},\,f\in\mathcal{F}_{\mu,L}(\mathbb{R}^{n}),\,d_{k-1},x_{k-1}\in\mathbb{R}^{n},\\
 & x_{k},d_{k}\text{ and }\beta_{k-1}\text{ generated by~\eqref{eq:NCG} from }x_{k-1}\text{ and }d_{k-1},\\
 & \left\langle \nabla f(x_{k-1});\,d_{k-1}\right\rangle =\|\nabla f(x_{k-1})\|^{2},\\
 & \|d_{k-1}\|^{2}=c_{k-1}\|\nabla f(x_{k-1})\|^{2}.
\end{array}\right)\label{eq:angle_PEP}
\end{equation}

For computing $c_{k}(\mu,L,c_{k-1})$, we reformulate~\eqref{eq:angle_PEP}
as follows. Denote $I \triangleq \{ k-1, k\}$. An appropriate sampling of the variable $f$ (which
is inconveniently infinite-dimensional) allows us to cast~\eqref{eq:angle_PEP} as:
\begin{equation}
c_{k}(\mu,L,c_{k-1})=\left(\begin{array}{ll}
\underset{\substack{\{d_{i}\}_{i\in I}, \, \gamma_{k-1}, \, \beta_{k-1},\\
\{(x_{i},g_{i},f_{i})\}_{i\in I}, \, n
}
}{\mbox{maximize}} & \frac{\|d_{k}\|^{2}}{\|g_{k}\|^{2}}\\
\textrm{subject to} & n\in\mathbb{N},\,\beta_{k-1}\in\mathbb{R},\,d_{k-1},d_{k}\in\mathbb{R}^{n},\\
 & \{(x_{i},g_{i},f_{i})\}_{i\in I}\subset\mathbb{R}^{n}\times\mathbb{R}^{n}\times\mathbb{R},\\
 & \exists f\in\mathcal{F}_{\mu,L}:\left\{ \begin{array}{l}
f(x_{i})=f_{i}\\
\nabla f(x_{i})=g_{i}
\end{array}\right.\forall i\in I,\\
 & \gamma_{k-1}=\underset{\gamma}{\mathrm{argmin}}\,f(x_{k-1}-\gamma\,d_{k-1}),\\
 & x_{k}=x_{k-1}-\gamma_{k-1}d_{k-1},\\
 & \beta_{k-1}=\frac{\|g_{k}\|^{2}-\eta\left\langle g_{k};\,g_{k-1}\right\rangle }{\|g_{k-1}\|^{2}},\\
 & d_{k}=g_{k}+\beta_{k-1}d_{k-1},\\
 & \left\langle g_{k-1};\,d_{k-1}\right\rangle =\|g_{k-1}\|^{2},\\
 & \|d_{k-1}\|^{2}=c_{k-1}\|g_{k-1}\|^{2}.
\end{array}\right) \label{eq:intermed_pep}
\end{equation}

{

Using \Cref{thm:interp}, the existence constraint can be replaced
by a set of linear/quadratic inequalities~\eqref{eq:interp} for
all pairs of triplets in $\{(x_{i},g_{i},f_{i})\}_{i \in I}$
without changing the objective value. So, applying \Cref{thm:interp} to \eqref{eq:intermed_pep} followed by an homogeneity argument and a few substitutions based on \eqref{eq:relaxed-line-search-cond}, we arrive at:

\begin{equation}
c_{k}(\mu,L,c_{k-1})=\left(\begin{array}{ll}
\underset{\substack{\{d_{i}\}_{i\in I}, \, \gamma_{k-1}, \, \beta_{k-1},\\
\{(x_{i},g_{i},f_{i})\}_{i\in I}, \, n
}
}{\mbox{maximize}} & \|d_{k}\|^{2}\\
\textrm{subject to} & n\in\mathbb{N},  \,d_{k-1},x_{k-1}\in\mathbb{R}^{n},\\
 & f_{i}\geq f_{j}+\langle g_{j};\,x_{i}-x_{j}\rangle+\frac{1}{2(1-\frac{\mu}{L})}\Big(\frac{1}{L}\|g_{i}-g_{j}\|^{2}\\
 & \qquad+\mu\|x_{i}-x_{j}\|^{2}-2\frac{\mu}{L}\left\langle g_{i}-g_{j};\,x_{i}-x_{j}\right\rangle \Big),\quad i,j\in I,\\
 & \left\langle g_{k-1};\,\pg_{k-1}\right\rangle =\|g_{k-1}\|^{2},\\
 & \left\langle g_{k};\,\pg_{k-1}\right\rangle =0,\\
 & \left\langle g_{k};\,x_{k-1}-x_{k}\right\rangle =0,\\
 & x_{k}=x_{k-1}-\gamma_{k-1}\pg_{k-1},\\
 & \beta_{k-1}=\frac{\|g_{k}\|^{2}-\eta\left\langle g_{k};\,g_{k-1}\right\rangle }{\|g_{k-1}\|^{2}},\\
 & \pg_{k}=g_{k}+\beta_{k-1}\pg_{k-1}\\
 & \|d_{k-1}\|^{2}=c_{k-1}\|g_{k-1}\|^{2},\\
 & \|g_{k}\|^{2}=1.
\end{array}\right)\label{eq:c_k_fnt_intract}
\end{equation}

We now show how to transform \eqref{eq:c_k_fnt_intract} into a finite-dimensional
nonconvex QCQP based on PEP methodologies developed in ~\cite{drori2014performance,taylor2017smooth,taylor2017exact}. To that goal, note that \eqref{eq:c_k_fnt_intract} contains
function values, inner product, and norm-squared involving $\{(x_{i},g_{i},f_{i})\}_{i\in I}$ and $\{d_i\}_{i \in I}$,
to model such terms in a compact manner, we introduce the following Grammian matrices:

\begin{equation}
\begin{alignedat}{1}H & =[x_{k-1}\mid g_{k-1}\mid g_{k}\mid\pg_{k-1}]\in\mathbb{R}^{n\times4},\\
G & =H^{\top}H\in\mathbb{S}_{+}^{4},\quad\rank G\leq n,\\
F & =[f_{k-1}\mid f_{k}]\in\mathbb{R}^{1\times2}.
\end{alignedat}
\label{eq:grammian-mats-ncg-Lyapunov-improved}
\end{equation}

We next define the following notation for selecting columns and elements
of $H$ and $F$: 
\begin{equation}
\begin{alignedat}{1} & \mathbf{x}_{k-1}=e_{1},\,\mathbf{g}_{k-1}=e_{2},\,\mathbf{g}_{k}=e_{3},\,\pgb_{k-1}=e_{4},\;(\textrm{all in }\mathbb{R}^{4})\\
 & \,\mathbf{f}_{k-1}=e_{1},\,\mathbf{f}_{k}=e_{2},\;(\textrm{all in }\mathbb{R}^{2}),\\
 & \mathbf{x}_{k}=\mathbf{x}_{k-1}-\gamma_{k-1}\pgb_{k-1},\;(\textrm{all in }\mathbb{R}^{4}),\\
 & \pgb_{k}=\mathbf{g}_{k}+\beta_{k-1}\pgb_{k-1},\;(\textrm{all in }\mathbb{R}^{4}).
\end{alignedat}
\label{eq:bold_vector_worst_case_sd}
\end{equation}
This ensures that $x_{i}=H\mathbf{x}_{i}$, $g_{i}=H\mathbf{g}_{i}$,
$\pg_{i}=H\pgb_{i}$, $f_{i}=F\mathbf{f}_{i},$ for {all
$i\in I$}. Next, for appropriate choices of matrices $A_{i,j}$,
$B_{i,j}$, $C_{i,j}$, $\widetilde{C}_{i,j}$, $D_{i,j}$, $\widetilde{D}_{i,j}$,
$E_{i,j}$, and vector $a_{i,j}$, we can ensure that the following
reformulations hold for all $i,j\in I$: 
\begin{equation}
\begin{alignedat}{1} & \left\langle g_{j};\,x_{i}-x_{j}\right\rangle =\tr GA_{i,j},\\
 & \|x_{i}-x_{j}\|^{2}=\tr GB_{i,j},\\
 & \|g_{i}-g_{j}\|^{2}=\tr GC_{i,j},\;\|g_{i}\|^{2}=\tr GC_{i,\star},\\
 & \|\pg_{i}-\pg_{j}\|^{2}=\tr G\widetilde{C}_{i,j},\;\|\pg_{i}\|^{2}=\tr G\widetilde{C}_{i,\star},\\
 & \left\langle g_{i};\,g_{j}\right\rangle =\tr GD_{i,j},\\
 & \langle g_{i};\,\pg_{j}\rangle=\tr G\widetilde{D}_{i,j},\\
 & \left\langle g_{i}-g_{j};\,x_{i}-x_{j}\right\rangle =\tr GE{}_{i,j},\\
 & f_{j}-f_{i}=Fa_{i,j},
\end{alignedat}
\label{eq:effect-of-grammian-trs-ncg-Lyapunov-improved}
\end{equation}
where, using \eqref{eq:bold_vector_worst_case_sd}, and using symmetric outer product notation $\left(\cdot\odot\cdot\right)\colon \mathbb{R}^{n}\times \mathbb{R}^{n}\to \mathbb{R}^{n\times n}$ such that for any $x,y\in \mathbb{R}^{n}$,
$x\odot y=\nicefrac{1}{2}\left(xy^{\top}+yx^{\top}\right)$, we define 
\begin{equation}
\begin{alignedat}{1} & A_{i,j}=\mathbf{g}_{j}\odot(\mathbf{x}_{i}-\mathbf{x}_{j})\\
 & B_{i,j}=(\mathbf{x}_{i}-\mathbf{x}_{j})\odot(\mathbf{x}_{i}-\mathbf{x}_{j})\\
 & C_{i,j}=(\mathbf{g}_{i}-\mathbf{g}_{j})\odot(\mathbf{g}_{i}-\mathbf{g}_{j}),\;C_{i,\star}=\mathbf{g}_{i}\odot\mathbf{g}_{i},\\
 & \widetilde{C}_{i,j}=(\pgb_{i}-\pgb_{j})\odot(\pgb_{i}-\pgb_{j}),\;\widetilde{C}_{i,\star}=\pgb_{i}\odot\pgb_{i},\\
 & D_{i,j}=\mathbf{g}_{i}\odot\mathbf{g}_{j},\\
 & \widetilde{D}_{i,j}=\mathbf{g}_{i}\odot\pgb_{j},\\
 & E_{i,j}=(\mathbf{g}_{i}-\mathbf{g}_{j})\odot(\mathbf{x}_{i}-\mathbf{x}_{j}),\\
 & a_{i,j}=\mathbf{f}_{j}-\mathbf{f}_{i}.
\end{alignedat}
\label{eq:ABCa-mat-vec-ncg-Lyapunov-improved}
\end{equation}

Using \eqref{eq:ABCa-mat-vec-ncg-Lyapunov-improved}, we can write
\eqref{eq:c_k_fnt_intract} as a finite-dimensional optimization
problem with a positive-semidefinite constraint:

\begin{equation}
c_{k}(\mu,L,c_{k-1})=\left(\begin{array}{ll}
\underset{G, \, F, \, \gamma_{k-1}, \, \beta_{k-1}, \, n}{\mbox{maximize}} & \tr G\widetilde{C}_{k,\star}\\
\textrm{subject to} & \tr G\widetilde{D}_{k-1,k-1}=\tr GC_{k-1,\star},\\
 & \tr G\widetilde{D}_{k,k-1}=0,\\
 & \tr GA_{k-1,k}=0,\\
 & \beta_{k-1}\times\tr GC_{k-1,\star}=\tr G\left(C_{k,\star}-\eta D_{k,k-1}\right),\\
 & \tr G\widetilde{C}_{k-1,\star}\leq c_{k-1}\tr GC_{k-1,\star},\\
 & Fa_{i,j}+\tr G\Big[A_{i,j}\\
 & \quad+\frac{1}{2(1-\frac{\mu}{L})}\left(\frac{1}{L}C_{i,j}+\mu B_{i,j}-2\frac{\mu}{L}E{}_{i,j}\right)\Big]\leq0,\;i,j\in I,\\
 & \tr GC_{k,\star}=1,\\
 & G\in\mathbb{S}_{+}^{4},\;\rank G\leq n. \; 
\end{array}\right)\label{eq:two-point-adapep-qcqp-1}
\end{equation}
In the optimization problem above, the only constraint
involving $n$ is $\rank G\leq n$, where the optimal value of the problem is monotonically nondecreasing in $n$. As $G\in\mathbb{S}_{+}^{4}$ (implying $\rank G \leq 4$), at the optimal solution, we have $\rank G \leq n$ satisfied automatically without impacting the optimal objective value, and the worst-case function would have a dimension of less than or equal to $4$.

Next, we model the positive semidefinite constraint $G\in\mathbb{S}_{+}^{4}$
using Cholesky factorization. Recall that a matrix $Z\in\mathbb{S}^{m}$ is positive semidefinite if and only if it has a Cholesky factorization $P ^\top P=Z$, where $P\in \mathbb{R}^{m \times m}$ \cite[Corollary 7.2.9]{horn2012matrix}. Hence, positive semidefiniteness of $G$
can be reformulated as $G=\widetilde{H}^{\top}\widetilde{H}$, where
$\widetilde{H}\in\mathbb{R}^{4\times4}$, i.e., for $G = H^\top H$ in \eqref{eq:grammian-mats-ncg-Lyapunov-improved},
we can let $H\in\mathbb{R}^{4\times4}$. Thus, we can write \eqref{eq:two-point-adapep-qcqp-1}
as the following nonconvex QCQP: 
\begin{equation}
c_{k}(\mu,L,c_{k-1})=\left(\begin{array}{ll}
\underset{\substack{G, \, F, \, H, \, \gamma_{k-1}, \, \beta_{k-1}, \\ \Theta, \, \{\Theta_{i,j}\}_{i,j\in I}}}{\mbox{maximize}} & \tr G\Theta\\
\textrm{subject to} & \tr G\widetilde{D}_{k-1,k-1}=\tr GC_{k-1,\star},\\
 & \tr G\widetilde{D}_{k,k-1}=0,\\
 & \tr GA_{k-1,k}=0,\\
 & \beta_{k-1}\times\tr GC_{k-1,\star}=\tr G\left(C_{k,\star}-\eta D_{k,k-1}\right),\\
 & \tr G\widetilde{C}_{k-1,\star}\leq c_{k-1}\tr GC_{k-1,\star},\\
 & Fa_{i,j}+\tr G\Big[A_{i,j}\\
 & \quad+\frac{1}{2(1-\frac{\mu}{L})}\left(\frac{1}{L}C_{i,j}+\mu\Theta_{i,j}-2\frac{\mu}{L}E{}_{i,j}\right)\Big]\leq0,\;i,j\in I,\\
 & \Theta=\widetilde{C}_{k,\star},\quad\Theta_{i,j}=B_{i,j},\;i,j\in I,\\
 & G=H^{\top}H,\\
 & \tr GC_{k,\star}=1.
\end{array}\right)\tag{\ensuremath{\mathcal{D}}}\label{eq:two-point-adapep-qcqp}
\end{equation}
Note that in the problem above, $\Theta$ and $\{\Theta_{i,j}\}_{i,j\in I_{N}^{\star}}$
are introduced as separate decision variables to formulate the cubic
constraints arising from $\widetilde{C}_{k,\star}$ and $B_{i,j}$
as quadratic constraints, respectively. This nonconvex QCQP can
be solved to certifiable global optimality using a custom spatial
branch-and-bound algorithm described in Section \ref{app:ncg-pep-alg}.

Finally, we recall that numerical solutions to~\eqref{eq:two-point-adapep-qcqp} correspond
to worst-case functions that can be obtained through the reconstruction
procedure from Theorem~\ref{thm:scvx-extension}. In addition, numerical
solutions can serve as inspirations for devising rigorous mathematical
proofs, as presented next. 

}

\subsection{Worst-case bounds for {\PRP} and {\FR}} \label{s:NCGs_inexact_GD_results}

In this section, we provide explicit solutions to~\eqref{eq:two-point-adapep-qcqp} for PRP and FR. Those results are then used for deducing simple convergence bounds through a straightforward application of Theorem~\ref{thm:approx_GD}.

{The main benefit of our proof structures is that they are verifiable through both calculations by hand and also by symbolic computer algebra systems.  Our proofs to the lemmas in this section (Lemmas \ref{thm:PRP_1}, \ref{Bound-on-beta}, \ref{thm:FR_1}) are obtained through Lagrangian relaxation by linearly combining the constraints of associated performance estimation problems with appropriate weights, where the weights themselves correspond to dual variables for the performance estimation problems. This makes our proofs independently verifiable programmatically using open-source symbolic computation libraries such as
		\texttt{SymPy} \cite{sympy}  and \texttt{Wolfram Language} \cite{wolframLanguage}. We have provided notebooks for the symbolic verifications of our proofs to Lemmas \ref{thm:PRP_1}, \ref{Bound-on-beta}, \ref{thm:FR_1} in the \texttt{Symbolic\_Verifications} folder of our open-source code available at \texttt{\url{https://github.com/Shuvomoy/NCG-PEP-code}}.}

\subsubsection{A worst-case bound for Polak-Ribière-Polyak (PRP)} 

Solving~\eqref{eq:two-point-adapep-qcqp} with $\eta = 1$ to global optimality allows obtaining the following worst-case bound for {\PRP} quantifying the \emph{quality} of the search direction with respect to the gradient direction. 

\begin{lemma}[Worst-case search direction for {\PRP}]\label{thm:PRP_1} 
Let
$f\in\mathcal{F}_{\mu,L}$, and let $x_{k-1},d_{k-1}\in\mathbb{R}^{n}$
and $x_{k}$, $d_{k}$ be generated by the {\PRP} method (i.e.,~\eqref{eq:NCG}
with $\eta=1$). It holds that: 
\begin{equation}
\frac{\|d_{k}\|^{2}}{\|\nabla f(x_{k})\|^{2}}\leq\frac{(1+\cond)^{2}}{4\cond},\label{eq:PRP_angle}
\end{equation}
with $\cond\triangleq\nicefrac{\mu}{L}$. Equivalently, 
$|\sin\theta_{k}|\leq\epsilon$ holds, where $\theta_{k}$ is
the angle between $\nabla f(x_{k})$ and $d_{k}$
and $\epsilon=\nicefrac{(1-\cond)}{(1+\cond)}$. 
\end{lemma} 
\begin{proof}
Recall that $x_{k}=x_{k-1}-\gamma_{k-1}\,d_{k-1}$ and $d_{k}=\nabla f(x_{k})+\beta_{k-1}d_{k-1}$.
The proof consists of the following weighted sum of inequalities: 
\begin{itemize}
\item optimality condition of the line search, with weight $\lambda_{1}=-\beta_{k-1}^{2}\frac{1+\cond}{L\gamma_{k-1}\cond}$:
\[
\langle\nabla f(x_{k});d_{k-1}\rangle=0,
\]
\item smoothness and strong convexity of $f$ between $x_{k-1}$ and $x_{k}$,
with weight $\lambda_{2}=\frac{\beta_{k-1}^{2}(1+\cond)^{2}}{L\gamma_{k-1}^{2}(1-\cond)\cond}$:
\[
\begin{aligned}f(x_{k-1})\geq & f(x_{k})+\langle\nabla f(x_{k});x_{k-1}-x_{k}\rangle+\tfrac{1}{2L}\|\nabla f(x_{k-1})-\nabla f(x_{k})\|^{2}\\
 & \quad+\tfrac{\mu}{2(1-\mu/L)}\|x_{k-1}-x_{k}-\tfrac{1}{L}(\nabla f(x_{k-1})-\nabla f(x_{k}))\|^{2}\\
= & f(x_{k})+\gamma_{k-1}\langle\nabla f(x_{k});\,d_{k-1}\rangle+\tfrac{1}{2L}\|\nabla f(x_{k-1})-\nabla f(x_{k})\|^{2}\\
 & \quad+\tfrac{\mu}{2(1-\mu/L)}\|\gamma_{k-1}d_{k-1}-\tfrac{1}{L}(\nabla f(x_{k-1})-\nabla f(x_{k}))\|^{2}
\end{aligned}
\]
\item smoothness and strong convexity of $f$ between $x_{k}$ and $x_{k-1}$,
with weight $\lambda_{3}=\lambda_{2}$: 
\[
\begin{aligned}f(x_{k})\geq & f(x_{k-1})+\langle\nabla f(x_{k-1});\,x_{k}-x_{k-1}\rangle+\tfrac{1}{2L}\|\nabla f(x_{k-1})-\nabla f(x_{k})\|^{2}\\
 & \quad+\tfrac{\mu}{2(1-\mu/L)}\|x_{k-1}-x_{k}-\tfrac{1}{L}(\nabla f(x_{k-1})-\nabla f(x_{k}))\|^{2}\\
= & f(x_{k-1})-\gamma_{k-1}\langle\nabla f(x_{k-1}),d_{k-1}\rangle+\tfrac{1}{2L}\|\nabla f(x_{k-1})-\nabla f(x_{k})\|^{2}\\
 & \quad+\tfrac{\mu}{2(1-\mu/L)}\|\gamma_{k-1}d_{k-1}-\tfrac{1}{L}(\nabla f(x_{k-1})-\nabla f(x_{k}))\|^{2}
\end{aligned}
\]
\item definition of $\beta_{k-1}$ with weight $\lambda_{4}=\frac{\beta_{k-1}(1+\cond)}{L\gamma_{k-1}\cond}$:
\[
\begin{aligned}0 & =\langle\nabla f(x_{k-1});\,\nabla f(x_{k})\rangle-\|\nabla f(x_{k})\|^{2}+\beta_{k-1}\|\nabla f(x_{k-1})\|^{2}\\
 & =\langle\nabla f(x_{k-1});\,\nabla f(x_{k})\rangle-\|\nabla f(x_{k})\|^{2}+\beta_{k-1}\langle\nabla f(x_{k-1});\,d_{k-1}\rangle.
\end{aligned}
\]
\end{itemize}
We arrive at the following weighted sum: 
\[
\begin{aligned}0\geq & \lambda_{1}\langle\nabla f(x_{k});d_{k-1}\rangle\\
 & +\lambda_{2}\bigg[f(x_{k})-f(x_{k-1})+\gamma_{k-1}\langle\nabla f(x_{k});\,d_{k-1}\rangle+\tfrac{1}{2L}\|\nabla f(x_{k-1})-\nabla f(x_{k})\|^{2}\\
 & \quad\quad\quad+\tfrac{\mu}{2(1-\mu/L)}\|\gamma_{k-1}d_{k-1}-\tfrac{1}{L}(\nabla f(x_{k-1})-\nabla f(x_{k}))\|^{2}\bigg]\\
 & +\lambda_{3}\bigg[f(x_{k-1})-f(x_{k})-\gamma_{k-1}\langle\nabla f(x_{k-1});\,d_{k-1}\rangle+\tfrac{1}{2L}\|\nabla f(x_{k-1})-\nabla f(x_{k})\|^{2}\\
 & \quad\quad\quad+\tfrac{\mu}{2(1-\mu/L)}\|\gamma_{k-1}d_{k-1}-\tfrac{1}{L}(\nabla f(x_{k-1})-\nabla f(x_{k}))\|^{2}\bigg]\\
 & +\lambda_{4}\big[\langle\nabla f(x_{k-1});\,\nabla f(x_{k})\rangle-\|\nabla f(x_{k})\|^{2}+\beta_{k-1}\langle\nabla f(x_{k-1});\,d_{k-1}\rangle\big]
\end{aligned}
\]
which can be reformulated exactly as ({by expanding both expressions and then
observing that all terms match, detailed calculations for this reformulation are provided in \appref{app:prp_reform}}) 
\[
\begin{aligned}0\geq & \|d_{k}\|^{2}-\frac{(1+\cond)^{2}}{4\cond}\|\nabla f(x_{k})\|^{2}\\
 & \quad+\frac{4\beta_{k-1}^{2}\cond}{(1-\cond)^{2}}\left\Vert d_{k-1}-\tfrac{1+\cond}{2L\gamma_{k-1}\cond}\nabla f(x_{k-1})+\tfrac{2\beta_{k-1}(1+\cond)-L\gamma_{k-1}(1-\cond)^{2}}{4\beta_{k-1}L\gamma_{k-1}\cond}\nabla f(x_{k})\right\Vert ^{2},\\
\geq & \|d_{k}\|^{2}-\frac{(1+\cond)^{2}}{4\cond}\|\nabla f(x_{k})\|^{2},
\end{aligned}
\]
thereby arriving at \eqref{eq:PRP_angle}. Finally, using \eqref{eq:explicit_bound},
we have 
$|\sin\theta_{k}|\leq\epsilon$ 
where $\epsilon=\nicefrac{(1-\cond)}{(1+\cond)}$.   \end{proof}

In Appendix~\ref{sec:num_vec_wcsd}, we numerically showcase the tightness of the worst-case bounds~\eqref{eq:PRP_angle} for {\PRP}. By tightness, we mean that we verified numerically that there exist $n\in\mathbb{N}$, functions $f\in\mathcal{F}_{\mu,L}$ and $x_{k-1},d_{k-1}\in\mathbb{R}^n$ such that ${\|d_{k}\|^{2}}=\left(\nicefrac{(1+\cond)^{2}}{4\cond}\right)\|\nabla f(x_{k})\|^{2}$. This is done by exhibiting feasible points to~\eqref{eq:two-point-adapep-qcqp} (obtained by solving~\eqref{eq:two-point-adapep-qcqp} numerically for $\eta = 1$) for different values of the inverse condition number $q$ and $c_{k-1}$. Those feasible points were verified through other ({existing}) software~\cite{goujaud2022pepit,taylor2017performance}.

The following rate is a direct consequence of \Cref{thm:PRP_1} and \Cref{thm:approx_GD}. 


\begin{theorem}[Worst-case bound for {\PRP}]\label{thm:PRP_bound} Let $f\in\mathcal{F}_{\mu,L}$, and $x_{k},d_{k}\in\mathbb{R}^n$ and $x_{k+1}$, $d_{k+1}\in\mathbb{R}^n$ be generated by respectively $k\geq 0$ and $k+1$ iterations of the {\PRP} method (i.e.,~\eqref{eq:NCG} with $\eta=1$). It holds that
    \[ f(x_{k+1})-f_\star\leq \displaystyle \left(\frac{1 - \cond^{2}}{1 + \cond^{2}}\right)^{2} \left(f(x_{k})-f_\star\right),\]
with $\cond\triangleq \nicefrac{\mu}{L}$. 
\end{theorem}\begin{proof}The desired claim is a direct consequence of \Cref{thm:approx_GD_sine-1} with $\epsilon=\frac{1-\cond}{1+\cond}$. That is, the {\PRP} scheme can be seen as a descent method with direction $d_k$ satisfying $\|d_{k}-\nabla f(x_{k})\|\leq \epsilon\|\nabla f(x_{k})\|$.
  \end{proof}

As a take-away from this theorem, we obtained an improved bound on the convergence rate of PRP, but possibly not in the most satisfying way: this analysis strategy does not allow beating steepest descent. Furthermore, this bound is tight for one iteration assuming that the current search direction satisfies $\nicefrac{\|d_k\|^2}{\|\nabla f(x_k)\|^2}=\nicefrac{(1+\cond)^{2}}{4\cond}$. However, it does not specify whether such an angle can be achieved on the same worst-case instances as those where~\Cref{thm:approx_GD} is achieved. In other words, there might be no worst-case instances where the bounds~\eqref{eq:ELS_inex} and~\eqref{eq:PRP_angle} are tight simultaneously, possibly leaving room for improvement in the analysis of {\PRP}. We show in~\Cref{sec:analysis} that we could indeed slightly improve this bound by taking into account the \emph{history} of the method in a more appropriate way by examining multiple iterations  of~\eqref{eq:NCG} rather than a single one.

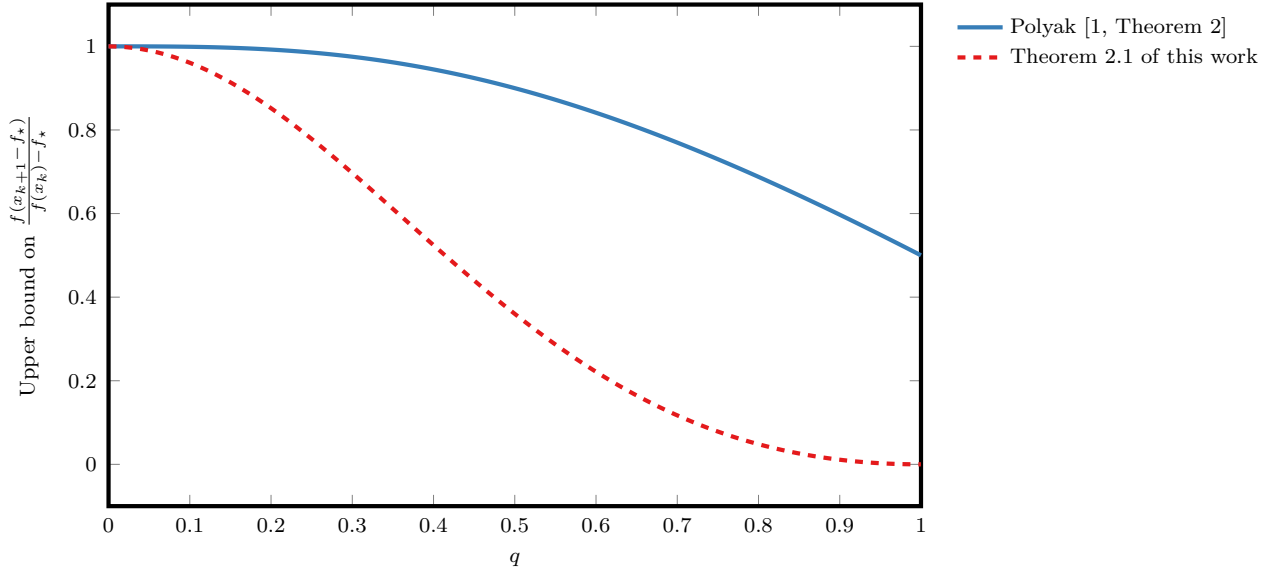
\begin{figure}[!ht]
\centering %
\input{polyak_vs_ncgpep.tex}
\caption{Comparison between the upper bounds on $\nicefrac{f(x_{k+1})-f_\star}{f(x_k)-f_\star}$ vs. condition number $q \triangleq \nicefrac{\mu}{L}$ for {\PRP} by Polyak \cite[Theorem 2]{polyak1969conjugate} and Theorem \ref{thm:PRP_bound} of this paper.\label{fig:PolyakVsPEP}}
\end{figure}

\begin{remark} The only worst-case complexity result that we are
aware of in the context of {\PRP} for smooth strongly
convex problems was provided by Polyak in~\cite[Theorem 2]{polyak1969conjugate}:

\[
f(x_{k+1})-f_{\star}\leq{\displaystyle \left(1 - \frac{\cond}{1+\frac{1}{\cond^{2}}} \right)\left(f(x_{k})-f_{\star}\right).}
\]

Figure \ref{fig:PolyakVsPEP} shows that the upper bound on $\nicefrac{f(x_{k+1})-f_\star}{f(x_k)-f_\star}$ for {\PRP} (for different values of $\cond$) provided by ~\cite[Theorem 2]{polyak1969conjugate} is significantly worse compared to that of Theorem \ref{thm:PRP_bound}.
From what we can tell,
this is due to two main weaknesses in the proof of Polyak~\cite[Theorem 2]{polyak1969conjugate}:
a weaker analysis of gradient descent, and a weaker analysis of the
direction (and in particular that $\nicefrac{\|d_{k}\|^{2}}{\|\nabla f(x_{k})\|^{2}}\leq1+\nicefrac{1}{\cond^{2}}$).
That is, whereas gradient descent with exact line searches is guaranteed
to achieve an accuracy $f(x_{k})-f_{\star}\leq\varepsilon$ in $O(\nicefrac{1}{\cond}\log\nicefrac{1}{\varepsilon})$,
our analysis provides an $O(\nicefrac{1}{\cond^{2}}\log\nicefrac{1}{\varepsilon})$
guarantee for {\PRP}, where Polyak's guarantee for {\PRP} is
$O(\nicefrac{1}{\cond^{3}}\log\nicefrac{1}{\varepsilon})$. As a reference,
note that the lower complexity bound (achieved by a few methods, including
many variations of Nesterov's accelerated gradients) is of order $O(\sqrt{\nicefrac{1}{\cond}}\log\nicefrac{1}{\varepsilon})$.
\end{remark}

\subsubsection{A worst-case bound for Fletcher-Reeves (FR)}

Similar to the obtaining of the bound for PRP, our bound for FR follows from solving~\eqref{eq:two-point-adapep-qcqp} (for $\eta=0$) in closed-form. We start by quantifying the \emph{quality} of the search direction with respect to the steepest descent direction. Unlike {\PRP}, where the worst-case ratio $\nicefrac{\|d_{k}\|^{2}}{\|\nabla f(x_{k})\|^{2}}$ depends only on the condition number $q$, in {\FR}, the ratio $\nicefrac{\|d_{k}\|^{2}}{\|\nabla f(x_{k})\|^{2}}$ depends also on the previous ratio $\nicefrac{\|d_{k-1}\|^{2}}{\|\nabla f(x_{k-1})\|^{2}}$.
To show this dependence, we first establish the following bound on the FR update parameter
$\beta_{k-1}$ in terms of $\nicefrac{\|d_{k-1}\|^{2}}{\|\nabla f(x_{k-1})\|^{2}}$ and $q$.

\begin{lemma}[Bound on $\beta_{k-1}$ for {\FR}\label{Bound-on-beta}]
 Let $f\in\mathcal{F}_{\mu,L}$, and let $x_{k-1},d_{k-1}\in\mathbb{R}^{n}$
and $x_{k}$, $d_{k}$ be generated by the {\FR} method (i.e.,~\eqref{eq:NCG}
where $c_{k-1}>1$, it holds that: 
\begin{equation}
0\leq\beta_{k-1}\leq\frac{1}{c_{k-1}}\frac{\left(1-q+2\sqrt{(c_{k-1}-1)q}\right)^{2}}{4q},\label{eq:case_2_beta_FR}
\end{equation}
where $q \triangleq \nicefrac{\mu}{L}$.
\end{lemma}

 \begin{proof}
First, note that $\beta_{k-1}\geq0$ by definition. The other part of the proof consists of the following
weighted sum of inequalities:
\begin{itemize}
\item relation between $\nabla f(x_{k-1})$ and $\pg_{k-1}$ with weight
$\lambda_{1}=\gamma_{k-1}(L+\mu)-\frac{2\sqrt{\beta_{k-1}}}{\sqrt{(c_{k-1}-1)c_{k-1}}}$: 

\[
0=\left\langle \nabla f(x_{k-1});\,\pg_{k-1}\right\rangle -\|\nabla f(x_{k-1})\|^{2},
\]

\item optimality condition of the line search with weight $\lambda_{2}=\frac{2}{c_{k-1}}-\gamma_{k-1}(L+\mu)$: 

\[
0=\left\langle \nabla f(x_{k});\,d_{k-1}\right\rangle ,
\]

\item definition of $\beta_{k-1}$ with weight $\lambda_{3}=\frac{\sqrt{c_{k-1}-1}}{\sqrt{\beta_{k-1}c_{k-1}}}$: 

\[
0=\|\nabla f(x_{k})\|^{2}-\beta_{k-1}\|\nabla f(x_{k-1})\|^{2},
\]

\item initial condition on the ratio $\frac{\|d_{k-1}\|^{2}}{\|\nabla f(x_{k-1})\|^{2}}$
with weight $\lambda_{4}=-\gamma_{k-1}^{2}L \mu +\frac{\sqrt{\beta_{k-1}}}{c_{k-1}\sqrt{(c_{k-1}-1)c_{k-1}}}$
: 
\[
0=\|\pg_{k-1}\|^{2}-c_{k-1}\|g_{k-1}\|^{2}
\]
\item smoothness and strong convexity of $f$ between $x_{k-1}$ and $x_{k}$,
with weight $\lambda_{5}=L-\mu$: 
\[
\begin{aligned}0\geq & -f(x_{k-1})+f(x_{k})+\langle\nabla f(x_{k});\,x_{k-1}-x_{k}\rangle+\tfrac{1}{2L}\|\nabla f(x_{k-1})-\nabla f(x_{k})\|^{2}\\
 & \quad+\tfrac{\mu}{2(1-\mu/L)}\|x_{k-1}-x_{k}-\tfrac{1}{L}(\nabla f(x_{k-1})-\nabla f(x_{k}))\|^{2}\\
= & f(x_{k}) { -f(x_{k-1})} +\gamma_{k-1}\langle\nabla f(x_{k});\,d_{k-1}\rangle+\tfrac{1}{2L}\|\nabla f(x_{k-1})-\nabla f(x_{k})\|^{2}\\
 & \quad+\tfrac{\mu}{2(1-\mu/L)}\|\gamma_{k-1}d_{k-1}-\tfrac{1}{L}(\nabla f(x_{k-1})-\nabla f(x_{k}))\|^{2}
\end{aligned}
\]
where going from the first line to the second, we used $x_{k-1}-x_{k}=\gamma_{k-1}d_{k-1}$,
\item smoothness and strong convexity of $f$ between $x_{k}$ and $x_{k-1}$,
with weight $\lambda_{6}=\lambda_{5}$: 
\[
\begin{aligned}0\geq & -f(x_{k})+f(x_{k-1})+\langle\nabla f(x_{k-1});\,x_{k}-x_{k-1}\rangle+\tfrac{1}{2L}\|\nabla f(x_{k-1})-\nabla f(x_{k})\|^{2}\\
 & \quad+\tfrac{\mu}{2(1-\mu/L)}\|x_{k-1}-x_{k}-\tfrac{1}{L}(\nabla f(x_{k-1})-\nabla f(x_{k}))\|^{2}\\
= & f(x_{k-1}) { -f(x_{k})}-\gamma_{k-1}\langle\nabla f(x_{k-1});\,d_{k-1}\rangle+\tfrac{1}{2L}\|\nabla f(x_{k-1})-\nabla f(x_{k})\|^{2}\\
 & \quad+\tfrac{\mu}{2(1-\mu/L)}\|\gamma_{k-1}d_{k-1}-\tfrac{1}{L}(\nabla f(x_{k-1})-\nabla f(x_{k}))\|^{2}
\end{aligned}
\]
where going from the first line to the second, we again used $x_{k-1}-x_{k}=\gamma_{k-1}d_{k-1}$,
\end{itemize}
The final weighted sum of inequalities is:
\begin{align*}
0 & \geq\lambda_{1}\left[\left\langle \nabla f(x_{k-1});\,\pg_{k-1}\right\rangle -\|\nabla f(x_{k-1})\|^{2}\right]+\lambda_{2}\left[\left\langle \nabla f(x_{k});\,d_{k-1}\right\rangle \right]\\
 & +\lambda_{3}\left[\|\nabla f(x_{k})\|^{2}-\beta_{k-1}\|\nabla f(x_{k-1})\|^{2}\right]+\lambda_{4}\left[\|\pg_{k-1}\|^{2}-c_{k-1}\|g_{k-1}\|^{2}\right]\\
 & +\lambda_{5}\Big[f(x_{k})-f(x_{k-1})+\gamma_{k-1}\langle\nabla f(x_{k});\,d_{k-1}\rangle+\tfrac{1}{2L}\|\nabla f(x_{k-1})-\nabla f(x_{k})\|^{2}\\
 & \quad\quad+\tfrac{\mu}{2(1-\mu/L)}\|\gamma_{k-1}d_{k-1}-\tfrac{1}{L}(\nabla f(x_{k-1})-\nabla f(x_{k}))\|^{2}\Big]\\
 & +\lambda_{6}\Big[f(x_{k-1})-f(x_{k})-\gamma_{k-1}\langle\nabla f(x_{k-1});\,d_{k-1}\rangle+\tfrac{1}{2L}\|\nabla f(x_{k-1})-\nabla f(x_{k})\|^{2}\\
 & \quad\quad+\tfrac{\mu}{2(1-\mu/L)}\|\gamma_{k-1}d_{k-1}-\tfrac{1}{L}(\nabla f(x_{k-1})-\nabla f(x_{k}))\|^{2}\Big],
\end{align*}
which can be reformulated exactly as ({by expanding both expressions and then
observing that all terms match, detailed calculations for this reformulation are provided in \appref{app:reform_lemma_2_2}}):

\begin{align*}
0\geq & \|\nabla f(x_{k})\|^{2}-\nu(\beta_{k-1},\gamma_{k-1},c_{k-1},\mu,L)\|\nabla f(x_{k-1})\|^{2}\\
 & +\left\Vert \sqrt[4]{\frac{\beta_{k-1}}{(c_{k-1}-1)c_{k-1}^{3}}}d_{k-1}-\sqrt[4]{\frac{\beta_{k-1}c_{k-1}}{c_{k-1}-1}}\nabla f(x_{k-1})+\sqrt[4]{\frac{c_{k-1}-1}{\beta_{k-1}c_{k-1}}}\nabla f(x_{k})\right\Vert ^{2}\\
\geq & \|\nabla f(x_{k})\|^{2}-\nu(\beta_{k-1},\gamma_{k-1},c_{k-1},\mu,L)\|\nabla f(x_{k-1})\|^{2},
\end{align*}
where 
\[
\nu(\beta_{k-1},\gamma_{k-1},c_{k-1},\mu,L)=2\sqrt{1-\frac{1}{c_{k-1}}}\sqrt{\beta_{k-1}}-c_{k-1}\gamma_{k-1}^{2}L\mu+\gamma_{k-1}(L+\mu)-1.
\]
So, we have: 
\begin{equation*}
    \begin{aligned}
  \beta_{k-1}&\leq\nu(\beta_{k-1},\gamma_{k-1},c_{k-1},\mu,L)\\
&\Leftrightarrow  \beta_{k-1}-2\sqrt{1-\frac{1}{c_{k-1}}}\sqrt{\beta_{k-1}}\leq-c_{k-1}\gamma_{k-1}^{2}L\mu +\gamma_{k-1}(L+\mu)-1\\
&\Rightarrow  \beta_{k-1}-2\sqrt{1-\frac{1}{c_{k-1}}}\sqrt{\beta_{k-1}}\leq\max_{\gamma}\left(-c_{k-1}\gamma_{k-1}^{2}L\mu +\gamma_{k-1}(L+\mu)-1\right).
    \end{aligned}
\end{equation*}
Because, $-c_{k-1}\gamma_{k-1}^{2}L\mu +\gamma_{k-1}(L+\mu)-1$ is
a concave function in $\gamma_{k-1},$ its maximum can be achieved
by differentiating the term with respect to $\gamma_{k-1},$ equating
it to $0$, and then solving for $\gamma_{k-1}$. The corresponding maximum value is equal to $\nicefrac{(L+\mu)^{2}}{4c_{k-1}L \mu}-1$ and achieved at $\gamma_{k-1}=\nicefrac{(L+\mu)}{(2c_{k-1}L\mu)}$. Hence,
the last inequality becomes:
\begin{equation*}
    \begin{aligned}
          \beta_{k-1}-&2\sqrt{1-\frac{1}{c_{k-1}}}\sqrt{\beta_{k-1}}-\frac{(L+\mu)^{2}}{4c_{k-1}L\mu}+1\leq0\\
&\Leftrightarrow \left(\sqrt{\beta_{k-1}}\right)^{2}-2\sqrt{1-\frac{1}{c_{k-1}}}\sqrt{\beta_{k-1}}+\left(\sqrt{1-\frac{1}{c_{k-1}}}\right)^{2}-\frac{(L+\mu)^{2}}{4c_{k-1}L\mu}-\left(\sqrt{1-\frac{1}{c_{k-1}}}\right)^{2}+1\leq0\\
&\Leftrightarrow  \left(\sqrt{\beta_{k-1}}-\sqrt{1-\frac{1}{c_{k-1}}}\right)^{2}\leq\frac{(L+\mu)^{2}}{4c_{k-1}L\mu}+\cancel{1}-\frac{1}{c_{k-1}}-\cancel{1}=\frac{1}{c_{k-1}}\left(\frac{(L+\mu)^{2}}{4L\mu}-1\right)\\
&\Leftrightarrow  \sqrt{\beta_{k-1}}\leq\sqrt{1-\frac{1}{c_{k-1}}}+\sqrt{\frac{(L+\mu)^{2}}{4c_{k-1}L\mu }-\frac{1}{c_{k-1}}}.
    \end{aligned}
\end{equation*}
Thereby, squaring both sides (which are nonnegative) of the last inequality and then through some algebra, we reach
\begin{align*}
\beta_{k-1} & \leq1+\frac{(L-\mu)}{c_{k-1}}\sqrt{\frac{(c_{k-1}-1)}{\mu L}}+\frac{\mu^{2}-6\mu L+L^{2}}{4c_{k-1}\mu L}\\
 & =\frac{1}{c_{k-1}}\frac{\left(1-q+2\sqrt{(c_{k-1}-1)q}\right)^{2}}{4q},
\end{align*}

{ which completes the proof. } 

  \end{proof}

Next, we prove a bound quantifying the quality of the search directions of {\FR}.

\begin{lemma}[Worst-case search direction for FR]\label{thm:FR_1}
Let $f\in\mathcal{F}_{\mu,L}$, and let $x_{k-1},d_{k-1}\in\mathbb{R}^{n}$
and $x_{k}$, $d_{k}$ be generated by the {\FR} method (i.e.,~\eqref{eq:NCG}
with $\eta=0$). For any $c_{k-1}\in\mathbb{R}$ such that $\nicefrac{\|d_{k-1}\|^{2}}{\|\nabla f(x_{k-1})\|^{2}}=c_{k-1}$,
where $c_{k-1}>1$, it holds that: 
\begin{equation}
\frac{\|d_{k}\|^{2}}{\|\nabla f(x_{k})\|^{2}}\leq c_{k}\triangleq1+\frac{\left(1-\cond+2\sqrt{(c_{k-1}-1)\cond}\right)^{2}}{4\cond},\label{eq:FR_angle}
\end{equation}
with $\cond\triangleq\nicefrac{\mu}{L}$. 

Equivalently, 
$|\sin\theta_{k}|\leq\epsilon$ holds, where $\theta_{k}$ is
the angle between $\nabla f(x_{k})$ and $d_{k}$
holds with $\epsilon=\sqrt{1-\nicefrac{1}{c_{k}}}$. \end{lemma} 
 \begin{proof}
The proof consists of the following weighted sum of {equalities}: 
\begin{itemize}
\item optimality condition of the line search with weight $\lambda_{1}=2\beta_{k-1}$:
\[
0=\langle\nabla f(x_{k});d_{k-1}\rangle,
\]
\item the quality of the search direction with weight $\lambda_{2}=\beta_{k-1}^{2}$:
\[
0=\|d_{k-1}\|^{2}-c_{k-1}\|\nabla f(x_{k-1})\|^{2},
\]
\item definition of $\beta_{k-1}$ with weight $\lambda_{3}=-c_{k-1}\beta_{k-1}$:
\[
\begin{aligned}0=\|\nabla f(x_{k})\|^{2}-\beta_{k-1}\|\nabla f(x_{k-1})\|^{2}.\end{aligned}
\]
\end{itemize}
The weighted sum can be simplified as (calculation shown in \appref{app:reform_lemma_2_3})

{
\begin{align*}0 = & \lambda_{1}\left[\langle\nabla f(x_{k});d_{k-1}\rangle\right]+\lambda_{2}\left[\|d_{k-1}\|^{2}-c_{k-1}\|\nabla f(x_{k-1})\|^{2}\right] +\lambda_{3}\left[\|\nabla f(x_{k})\|^{2}-\beta_{k-1}\|\nabla f(x_{k-1})\|^{2}\right]\\
=\, & \|d_{k}\|^{2}-(1+c_{k-1}\beta_{k-1})\|\nabla f(x_{k})\|^{2}.
\end{align*}
Hence,
\[
\begin{aligned}\|d_{k}\|^{2} & =(1+c_{k-1}\beta_{k-1})\|\nabla f(x_{k})\|^{2}\\
 & \leq\left(1+\frac{\left(1-\cond+2\sqrt{(c_{k-1}-1)\cond}\right)^{2}}{4\cond}\right)\|\nabla f(x_{k})\|^{2},
\end{aligned}
\]
}
where in the last line we have used the upper bound on $\beta_{k-1}$ from~\eqref{eq:case_2_beta_FR}. This gives us  \eqref{eq:FR_angle}. Finally, using \eqref{eq:explicit_bound},
we have 
$|\sin\theta_{k}|\leq\epsilon$, where $\epsilon=\sqrt{1-\nicefrac{1}{c_{k}}}$.
\end{proof}

Similar to PRP, in \appref{sec:num_vec_wcsd}, we compare this last bound with the worst example that we were able to find numerically (i.e., worst feasible points to \eqref{eq:two-point-adapep-qcqp}). Thereby, we conclude tightness of the bound on the quality of the search direction~\eqref{eq:FR_angle}. That is, we claim that for all values of $\cond$ and $c_{k-1}$, there exist $n\in\mathbb{N}$, functions $f\in\mathcal{F}_{\mu,L}$ and $x_{k-1},d_{k-1}\in\mathbb{R}^n$ such that the bound from Lemma~\ref{thm:FR_1} is achieved with equality.

That being said, this bound only allows obtaining unsatisfactory convergence results for FR, although not letting much room for improvements, as showed in the next sections.

\begin{theorem}[Worst-case bound]\label{thm:FR_2} Let $f\in\mathcal{F}_{\mu,L}$, and $x_{k},d_{k}\in\mathbb{R}^n$ and $x_{k+1}$, $d_{k+1}\in\mathbb{R}^n$ be generated by respectively $k\geq 0$ and $k+1$ iterations of the {\FR} method (i.e.,~\eqref{eq:NCG} with $\eta=0$). It holds that
    \[ f(x_{k+1})-f_\star\leq \displaystyle \left(\frac{1 - \cond \frac{1-\epsilon_k}{1+\epsilon_k}}{1+\cond\frac{1-\epsilon_k}{1+\epsilon_k}}\right)^{2} \left(f(x_{k})-f_\star\right),\]
with $\epsilon_k=\sqrt{\nicefrac{(1-\cond)^2 (k-1)^2}{4 \cond + (1-\cond)^2 (k-1)^2}}$.
\end{theorem}
\begin{proof}The desired claim is a direct consequence of \Cref{thm:approx_GD_sine-1} with  \Cref{thm:FR_1}. Indeed, it follows from 
\[c_{k} \leq 1+\frac{\left(1-\frac{\mu}{L}+2 \sqrt{(c_{k-1}-1)\frac{\mu }{L}}\right)^2}{\frac{4 \mu }{L}}\]
(the guarantee from~\Cref{thm:FR_1} for the quality of the search direction) which we can rewrite as
\[ \sqrt{c_{k+1}-1}\leq \frac{1-\cond+2\sqrt{(c_k-1)\cond}}{2\sqrt{\cond}}\]
with $c_{0}-1=0$, thereby arriving to $c_{k}\leq 1 + k^2 \nicefrac{(1-\cond)^2}{4\cond}$ by recursion. For applying \Cref{thm:approx_GD}, we compute $\epsilon_k=\sqrt{1-\nicefrac{1}{c_k}}\leq\sqrt{\nicefrac{(1-\cond)^2 k^2}{4 \cond + (1-\cond)^2 k^2}}$ and reach the desired statement.
  \end{proof}
It is clear that the statement of \Cref{thm:FR_2} is rather disappointing, as the convergence rate of the FR variation can become arbitrarily close to 1. While this guarantee clearly does not give a total and fair picture of the true behavior of FR in practice, it seems in line with the practical necessity to effectively restart the method as it runs~\cite{hager2006survey}.

The next section is devoted to studying the possibilities for obtaining tighter guarantees for~PRP and~FR beyond the simple single-iteration worst-case analyses of this section (which are tight for one iteration, but not beyond), showing that we cannot hope to improve the convergence rates for those methods without further assumptions on the problems at hand.

\section{Obtaining better worst-case bounds for NCGMs} \label{sec:analysis}

In the previous section, we established closed-form bounds on ratios between consecutive function values for NCGMs by characterizing worst-case search directions. Albeit being tight for the analysis of NCGMs for one iteration, the bounds that we obtained are disappointingly inferior to those of the vanilla gradient descent. In this section, we investigate the possibility of obtaining better worst-case guarantees for NCGMs. For doing this using our framework, one natural possibility for us is to go beyond the study of a single iteration (since our results appear to be tight for this situation). Therefore, in contrast with the previous section, we now proceed only numerically and provide worst-case bounds without closed-forms. 

More precisely, we solve the corresponding PEPs in two regimes. In short, the difference between the two regimes resides in the type of bounds under consideration. 
\begin{enumerate}
    \item The first type of bounds {can be thought of as} a ``Lyapunov'' approach which studies $N$ iterations of~\eqref{eq:NCG} starting at some iterate $(x_k, d_k)$ (for which we ``neglect'' how it was generated). In this first setup, we numerically compute worst-case bounds on $\nicefrac{f(x_{k+N} ) - f_\star}{f(x_{k})-f_\star}$ for different values of $N$ (namely $N = 1,2,3,4$). As for the results of Section~\ref{s:NCG_as_GD}, we quantify the quality of the couple $(x_k,d_k)$ by requiring that $\| d_k \|^2 \leq c_k \| \nabla f(x_k)\|^2$. When $N=1$, this setup corresponds to that of Section~\ref{s:NCG_as_GD}. Stemming from the fact {that} the worst-case behaviors observed for $N=1$ might not be compatible between consecutive iterations, we expect the quality of the bounds to improve with $N$. Of course, the main weakness of this approach is the fact that we neglect how $(x_k,d_k)$ was generated.
    
    \item As a natural complementary alternative, the second type of bounds studies $N$ iterations of~\eqref{eq:NCG} initiated at $x_0$ (with $d_0=\nabla f(x_0)$). Whereas the first type of bounds is by construction more conservative, it has the advantage of being \emph{recursive}: it is valid for all $k\geq 0$. On the other side, the second type of bounds is only valid for the first $N$ iterations (the bound cannot be used recursively), but it cannot be improved at all. That is, we study \emph{exact} worst-case ratio $\nicefrac{f(x_{N} ) - f_\star}{f(x_{0})-f_\star}$ for a few different values of $N$ (namely $N\in \{1,2,3,4\}$). In this setup, we obtain worst-case bounds that are only valid close to initialization. However, it has the advantage of being unimprovable, as we do not neglect how the search direction is generated.
\end{enumerate}

\paragraph{Section organization.}
This section is organized as follows. First, in Section~\ref{subsec:NCG-PEP-ift-formulation}
we present the performance estimation problems for~\eqref{eq:NCG}
specifically for computing the worst-case ratios $\nicefrac{f(x_{k+N})-f_{\star}}{f(x_{k})-f_{\star}}$
and $\nicefrac{f(x_{N})-f_{\star}}{f(x_{0})-f_{\star}}$. In Section
\ref{sec:ncqcqp_for_Lyapunov_and_exact}, we describe the steps to arrive at the nonconvex
QCQP formulations for the performance estimation problems considered.
Then, Section~\ref{sec:num-res-prp} and Section~\ref{subsec:Numerical-results-for-FR}
presents our findings for respectively~{\PRP} and~{\FR}. In
Appendix \ref{subsec:counter-example}, we discuss how to generate
the counter-examples from the solutions to the nonconvex QCQPs.

\subsection{Computing numerical worst-case scenarios}
\label{subsec:NCG-PEP-ift-formulation}

Similar to~\eqref{eq:angle_PEP}, the problem of computing the worst-case ratio $\nicefrac{f(x_{k+N} ) - f_\star}{f(x_{k})-f_\star}$ is framed as the following nonconvex maximization problem (for $c\geq 1$ and $\cond\triangleq\nicefrac{\mu}{L}$):
\begin{equation}
\rho_N(\cond,c)\triangleq\left(\begin{array}{ll}
\underset{\substack{f, \, \{x_{k+i}\}_i, \, \{d_{k+i}\}_{i},\\
\{\gamma_{k+i}\}_{i}, \, \{\beta_{k+i}\}_{i}, \, n
}
}{\mbox{maximize}} & \frac{f(x_{k+N})-f_{\star}}{f(x_{k})-f_{\star}}\\
\textrm{subject to} & n\in\mathbb{N},\, {f\in\mathcal{F}_{\mu, L}(\mathbb{R}^{n})} ,\,d_{k},x_{k}\in\mathbb{R}^{n},\\
 & \left\langle \nabla f(x_{k});\,d_{k}\right\rangle =\|\nabla f(x_{k})\|^{2},\\
 &  \|d_{k}\|^{2}\leq c\|\nabla f(x_{k})\|^{2},\\
 & \begin{pmatrix}x_{k+1}\\d_{k+1}\\\beta_{k}\end{pmatrix},\ldots,\begin{pmatrix}x_{k+N}\\d_{k+N}\\\beta_{k+N-1}\end{pmatrix}\text{ generated by~\eqref{eq:NCG} from }x_{k}\text{ and }d_{k}.
\end{array}\right)\tag{$\mathcal{B}_\textup{Lyapunov}$}\label{eq:PEP_rho_N}
\end{equation}
We proceed similarly for $\nicefrac{f(x_{N} ) - f_\star}{f(x_{0})-f_\star}$:
\begin{equation}
\rho_{N,0}(\cond)\triangleq\left(\begin{array}{ll}
\underset{\substack{f, \, \{x_{k+i}\}_i, \, \{d_{k+i}\}_{i},\\
\{\gamma_{k+i}\}_{i}, \, \{\beta_{k+i}\}_{i}, \, n
}
}{\mbox{maximize}} & \frac{f(x_{N})-f_{\star}}{f(x_{0})-f_{\star}}\\
\textrm{subject to} & n\in\mathbb{N},\,f\in {\mathcal{F}_{\mu, L}(\mathbb{R}^{n})},\,x_{0}\in\mathbb{R}^{n},\\
 & d_0=\nabla f(x_0),\\
 & \begin{pmatrix}x_{1}\\d_{1}\\\beta_{0}\end{pmatrix},\ldots,\begin{pmatrix}x_{N}\\d_{N}\\\beta_{N-1}\end{pmatrix}\text{ generated by~\eqref{eq:NCG} from }x_{k}\text{ and }d_{k}.
\end{array}\right)\tag{$\mathcal{B}_\textup{exact}$}\label{eq:PEP_rho_N0}
\end{equation}

{

Obviously, $\rho_{N}(\cond,c)\geq\rho_{N,0}(\cond)$ for any $c\geq1$.
We solve the nonconvex QCQP reformulations of~\eqref{eq:PEP_rho_N} and~\eqref{eq:PEP_rho_N0} numerically
to high precision (reformulation details shown in Section \ref{sec:ncqcqp_for_Lyapunov_and_exact})
for $N\in\{1,2,3,4\}$ and report the corresponding results in what
follows. In the numerical experiments, we fix the values of $c$ using
Lemma~\ref{thm:PRP_1} for {\PRP} in~\eqref{eq:PEP_rho_N}, thereby
computing $\rho_{N}\left(\cond,\nicefrac{(1+\cond)^{2}}{4\cond}\right)$
whose results are provided in Figure~\ref{fig:PRP_Lyapunov} of Section \ref{sec:num-res-prp}. For
{\FR}, $c$ can become arbitrarily bad and we therefore only compute
$\rho_{N,0}(\cond)$ via~\eqref{eq:PEP_rho_N0}. The numerical values
for $\rho_{N,0}(\cond)$ respectively {\PRP} and {\FR} are provided
in Figure~\ref{fig:PRP_exact} and Figure~\ref{fig:FR_exact}, located in Section \ref{sec:num-res-prp} and Section \ref{subsec:Numerical-results-for-FR}, respectively. 

 In the next section, we describe the nonconvex
QCQP formulations for ~\eqref{eq:PEP_rho_N} and~\eqref{eq:PEP_rho_N0}.
Readers interested in the findings of our numerical experiments by solving the nonconvex QCQPs can skip to Section~\ref{sec:num-res-prp} (for ~{\PRP}) and Section~\ref{subsec:Numerical-results-for-FR} (for~{\FR}).

\subsection{Nonconvex QCQP reformulations of ~\eqref{eq:PEP_rho_N} and~\eqref{eq:PEP_rho_N0}\label{sec:ncqcqp_for_Lyapunov_and_exact}}

Similar to the reformulations from \eqref{eq:two-point-adapep-qcqp},~
\eqref{eq:PEP_rho_N} and~\eqref{eq:PEP_rho_N0} can be cast as nonconvex
QCQPs, where the number of nonconvex constraints grows quadratically
with $N$. Thereby, solving them to global optimality in reasonable
time for $N=3,4$ is already challenging.

Therefore, rather than solving the nonconvex QCQP reformulations of
\eqref{eq:PEP_rho_N} and~\eqref{eq:PEP_rho_N0} directly, we compute
upper bounds and lower bounds by solving more tractable nonconvex
QCQP formulations. We then show that the relative gap between the
upper and lower bounds is less than $10\%$ which thereby indicates
that there is essentially no room for further improvement.

\subsubsection{Nonconvex QCQP reformulation of \eqref{eq:PEP_rho_N}\label{subsec:Bounds-for-Lyapunov-ratios}}

This section presents the nonconvex QCQP formulations for our upper bound $\overline{\rho}_{N}(\cond,c)$
and lower bound $\underline{\rho}_{N}(\cond,c)$ on $\rho_{N}(\cond,c)$. We use the notation $[a:b] = \{a, a+1, a+2, \ldots, b-1, b\}$ where $a,b$ are integers.

\paragraph{Computing $\overline{\rho}_{N}(\cond,c)$\label{sec:comp-upper-bd}}

Using \eqref{eq:relaxed-line-search-cond}, we have the following
relaxation of~\eqref{eq:PEP_rho_N}, which provides upper bounds
on ${\rho}_{N}(\cond,c)$:

\begin{equation}
\left(\begin{array}{ll}
\underset{\substack{\{x_{k+i}\}_{i\in[0:N]}, \\ \{d_{k+i}\}_{i\in[0:N]}, \\ f, \, n  }}{\mbox{maximize}} & \frac{f(x_{k+N})-f_{\star}}{f(x_{k})-f_{\star}}\\
\textrm{subject to} & n\in\mathbb{N},\,f\in\mathcal{F}_{\mu,L}(\mathbb{R}^{n}),\\
 & x_{k+i},\,d_{k+i}\in\mathbb{R}^{n},\quad i\in[0:N]\\
 & \|\pg_{k}\|^{2}\leq c\|\nabla f(x_{k})\|^{2},\\
 & \left\langle \nabla f(x_{k+i+1});\,\pg_{k+i}\right\rangle =0,\quad i\in[0:N-1],\\
 & \left\langle \nabla f(x_{k+i+1});\,x_{k+i}-x_{k+i+1}\right\rangle =0,\quad i\in[0:N-1],\textrm{ }\\
 & \left\langle \nabla f(x_{k+i});\,\pg_{k+i}\right\rangle =\|\nabla f(x_{k+i})\|^{2},\quad i\in[0:N-1],\\
 & \pg_{k+i+1}=g_{k+i+1}+\beta_{k+i}\pg_{k+i},\quad i\in[0:N-2],\\
 & \beta_{k+i}=\frac{\|g_{k+i+1}\|^{2}-\eta\left\langle g_{k+i+1};\,g_{k+i}\right\rangle }{\|g_{k+i}\|^{2}},\quad i\in[0:N-2].
\end{array}\right)\label{eq:N-point-ncg-pep-relaxed}
\end{equation}
Using the notation $g_{i}\triangleq\nabla f(x_{i})$ and $f_{i}\triangleq f(x_{i})$
again, and then applying an homogeneity argument, we write \eqref{eq:N-point-ncg-pep-relaxed}
as: 
\begin{equation}
\overline{\rho}_{N}(\cond,c)=\left(\begin{array}{ll}
\underset{\substack{\{x_{k+i}\}_{i\in[0:N]}, \\ \{\pg_{k+i}\}_{i\in[0:N]}, \\ f, \, n} }{\mbox{maximize}} & f_{k+N}-f_{\star}\\
\textrm{subject to} & n\in\mathbb{N},\,f\in\mathcal{F}_{\mu,L}(\mathbb{R}^{n}),\\
 & x_{k+i},\,d_{k+i}\in\mathbb{R}^{n},\quad i\in[0:N]\\
 & \|\pg_{k}\|^{2}\leq c\|g_{k}\|^{2},\\
 & \left\langle g_{k+i+1};\,\pg_{k+i}\right\rangle =0,\quad i\in[0:N-1],\\
 & \left\langle g_{k+i+1};\,x_{k+i}-x_{k+i+1}\right\rangle =0,\quad i\in[0:N-1],\textrm{ }\\
 & \left\langle g_{k+i} ;\, \pg_{k+i}\right\rangle =\|g_{k+i}\|^{2},\quad i\in[0:N-1],\\
 & \pg_{k+i+1}=g_{k+i+1}+\beta_{k+i}\pg_{k+i},\quad i\in[0:N-2],\\
 & \beta_{k+i-1}=\frac{\|g_{k+i}\|^{2}-\eta\left\langle g_{k+i};\,g_{k+i-1}\right\rangle }{\|g_{k+i-1}\|^{2}},\quad i\in[1:N-1],\\
 & f_{k}-f_{\star}=1.
\end{array}\right)\label{eq:overline_rho_0_shifted_inft_dm}
\end{equation}
Define $I_{N}^{\star}=\{\star,k,k+1,\ldots,k+N\}$.
Next, note that the equation $\pg_{k+i+1}=g_{k+i+1}+\beta_{k+i}\pg_{k+i}$
for $i\in[0:N-2],$ can be written equivalently as the following set
of equations:

\begin{equation}
\begin{aligned} & \chi_{j,i}=\chi_{j,i-1}\beta_{k+i-1},\quad i\in[1:N-1],j\in[0:i-2],\\
 & \chi_{i-1,i}=\beta_{k+i-1},\quad i\in[1:N-1],\\
 & \pg_{k+i}=g_{k+i}+\sum_{j=1}^{i-1}\chi_{j,i}g_{k+j}+\chi_{0,i}\pg_{k},\quad i\in[1:N-1],
\end{aligned}
\label{eq:eqs-for-pseudograds}
\end{equation}
where we have introduced the intermediate variables $\chi_{j,i}$,
which will aid us in formulating \eqref{eq:overline_rho_0_shifted_inft_dm}
as a nonconvex QCQP down the line. In absence of these intermediate
variables in \eqref{eq:eqs-for-pseudograds}, the resultant constraints
in the final optimization problem will involve polynomials of degree
three or more in the decision variables, and such optimization problems
present a significantly greater challenge in solving to global optimality
compared to a QCQP. Next, using \eqref{eq:eqs-for-pseudograds} and
Theorem \ref{thm:interp}, we can equivalently write \eqref{eq:overline_rho_0_shifted_inft_dm}
as: 
\begin{equation}
\overline{\rho}_{N}(\cond,c)=\left(\begin{array}{ll}
\underset{\substack{\{x_{k+i},g_{k+i},f_{k+i}\}_{i}, \, n, \\ \{\pg_{k+i}\}_{i}, \, \{\beta_{k+i}\}_{i},\{\chi_{j,i}\}_{j,i}}}{\mbox{maximize}} & f_{k+N}-f_{\star}\\
\textrm{subject to} & n\in\mathbb{N},\\
 & f_{i}\geq f_{j}+\langle g_{j};\,x_{i}-x_{j}\rangle+\frac{1}{2(1-\frac{\mu}{L})}\Big(\frac{1}{L}\|g_{i}-g_{j}\|^{2}\\
 & \quad+\mu\|x_{i}-x_{j}\|^{2}-2\frac{\mu}{L}\left\langle g_{i}-g_{j};\,x_{i}-x_{j}\right\rangle \Big),\quad i,j\in I_{N}^{\star},\\
 & \|\pg_{k}\|^{2}\leq c\|g_{k}\|^{2},\\
 & \left\langle g_{k+i+1} ;\, \pg_{k+i}\right\rangle =0,\quad i\in[0:N-1],\\
 & \left\langle g_{k+i+1} ;\, x_{k+i}-x_{k+i+1}\right\rangle =0,\quad i\in[0:N-1],\textrm{ }\\
 & \left\langle g_{k+i} ;\, \pg_{k+i}\right\rangle =\|g_{k+i}\|^{2},\quad i\in[0:N-1],\\
 & \beta_{k+i-1}=\frac{\|g_{k+i}\|^{2}-\eta\left\langle g_{k+i} ;\, g_{k+i-1}\right\rangle }{\|g_{k+i-1}\|^{2}},\quad i\in[1:N-1],\\
 & \chi_{j,i}=\chi_{j,i-1}\beta_{k+i-1},\quad i\in[1:N-1],j\in[0:i-2],\\
 & \chi_{i-1,i}=\beta_{k+i-1},\quad i\in[1:N-1],\\
 & \pg_{k+i}=g_{k+i}+\sum_{j=1}^{i-1}\chi_{j,i}g_{k+j}+\chi_{0,i}\pg_{k},\quad i\in[1:N-1],\\
 & f_{k}-f_{\star}=1,\\
 & g_{\star}=0,\,x_{\star}=0,\,f_{\star}=0,\\
 & \{x_{i},g_{i},f_{i}\}_{i\in I_{N}^{\star}}\subset\mathbb{R}^{n}\times\mathbb{R}^{n}\times\mathbb{R},\;\{\pg_{i}\}_{i\in [k+1:k+N-1]}\subset\mathbb{R}^{n},\\
 & \{\beta_{k+i}\}_{i\in[0:N-2]}\subset\mathbb{R},\;\{\chi_{j,i}\}_{j\in[0:N-2],i\in[0:N-1]}\subset\mathbb{R}.
\end{array}\right)\label{eq:overline_rho_0_shifted_fnt_dm_intract}
\end{equation}
Note that we have set $g_{\star}=0$,
$x_{\star}=0$, and $f_{\star}=0$ without loss of generality, because
both the objective and the function class are closed and invariant
under shifting variables and function values. We introduce Grammian
matrices again: 
\begin{equation}
\begin{alignedat}{1}H & =[\pg_{k}\mid g_{k}\mid g_{k+1}\mid g_{k+2}\mid\cdots\mid g_{k+N}\mid x_{k}\mid x_{k+1}\mid x_{k+2}\mid\cdots\mid x_{k+N}]\in\mathbb{R}^{n\times(2N+3)},\\
G & =H^{\top}H\in\mathbb{S}_{+}^{(2N+3)},\;\rank G\leq n,\\
F & =[f_{k}\mid f_{k+1}\mid\ldots\mid f_{k+N}]\in\mathbb{R}^{1\times(N+1)}.
\end{alignedat}
\label{eq:grammian-mats-ncg-Lyapunov-1}
\end{equation}

Using the same arguments described in Section \ref{sec:worst-case-search-dir}, we can ignore the constraint $\rank G\leq n$,
and confine $H$ to be in $\mathbb{R}^{(2N+3)\times(2N+3)}$ without
loss of generality.
Next, define the following notation for selecting columns and elements
of $H$ and $F$: 
\begin{equation}
\begin{alignedat}{1} & \mathbf{x}_{\star}=\mathbf{0}\in\mathbb{R}^{2N+3},\,\pgb_{k}=e_{1}\in\mathbb{R}^{2N+3},\,\mathbf{g}_{k+i}=e_{i+2}\in\mathbb{R}^{2N+3},\\
 & \mathbf{x}_{k+i}=e_{(N+2)+(i+1)}\in\mathbb{R}^{2N+3},\\
 & \mathbf{f}_{\star}=\mathbf{0},\;\mathbf{f}_{k+i}=e_{i+1}\in\mathbb{R}^{(N+1)},\\
 & \pgb_{k+i}=\mathbf{g}_{k+i}+\sum_{j=1}^{i-1}\chi_{j,i}\mathbf{g}_{k+j}+\chi_{0,i}\pgb_{k}\in\mathbb{R}^{2N+3},
\end{alignedat}
\label{eq:rho_overline_bold_vectors}
\end{equation}
where $i\in[0:N]$. This ensures that we have $x_{i}=H\mathbf{x}_{i}$,
$g_{i}=H\mathbf{g}_{i}$, $\pg_{i}=H\pgb_{i}$, $f_{i}=F\mathbf{f}_{i}$
for all $i\in I_{N}^{\star}$. For appropriate choices of matrices
$A_{i,j}$,$B_{i,j}$, $C_{i,j}$, $\widetilde{C}_{i,j}$, $D_{i,j}$,
$\widetilde{D}_{i,j}$, $E_{i,j}$, and vector $a_{i,j}$ as defined
in \eqref{eq:effect-of-grammian-trs-ncg-Lyapunov-improved}, where
$\mathbf{x}_{i}$, $\mathbf{g}_{i}$, $\mathbf{f}_{i}$, $\pgb_{i}$
are taken from \eqref{eq:rho_overline_bold_vectors} now, we can ensure
that the identities in \eqref{eq:ABCa-mat-vec-ncg-Lyapunov-improved}
hold for all $i,j\in I_{N}^{\star}$. Using those identities and using
the definition of $G=H^{\top}H$, where $H\in\mathbb{R}^{(2N+3)\times(2N+3)}$,
we can write \eqref{eq:overline_rho_0_shifted_fnt_dm_intract} as
the following nonconvex QCQP: 
\begin{equation}
\overline{\rho}_{N}(\cond,c)=\left(\begin{array}{ll}
\underset{\substack{F, \, G, \, H, \\ \{\chi_{j,i}\}_{j,i}, \, \{\beta_{k+i}\}_{i}}}{\mbox{maximize}} & Fa_{\star,k+N}\\
\textrm{subject to} & Fa_{i,j}+\tr G\left[A_{i,j}+\frac{1}{2(1-\frac{\mu}{L})}\left(\frac{1}{L}C_{i,j}+\mu B_{i,j}-2\frac{\mu}{L}E{}_{i,j}\right)\right]\leq0,\quad i,j\in I_{N}^{\star},\\
 & \tr G\widetilde{C}_{k,\star}\leq c\tr GC_{k,\star},\\
 & \tr G\widetilde{D}_{k+i+1,k+i}=0,\quad i\in[0:N-1],\\
 & \tr GA_{k+i,k+i+1}=0,\quad i\in[0:N-1],\\
 & \tr G\widetilde{D}_{k+i,k+i}=\tr GC_{k+i,\star}\quad i\in[0:N-1],\\
 & \beta_{k+i-1}\times\tr GC_{k+i-1,\star}=\tr G\left(C_{k+i,\star}-\eta D_{k+i,k+i-1}\right),\quad i\in[1:N-1],\\
 & \chi_{j,i}=\chi_{j,i-1}\beta_{k+i-1},\quad i\in[1:N-1],\,j\in[0:i-2],\\
 & \chi_{i-1,i}=\beta_{k+i-1},\quad i\in[1:N-1],\\
 & Fa_{\star,k}=1,\\
 & G=H^{\top}H,\\
 & F\in\mathbb{R}^{N+1},\;G\in\mathbb{S}^{2N+3},\;H\in\mathbb{R}^{(2N+3)\times(2N+3)},\\
 & \{\beta_{k+i}\}_{i\in[0:N-2]}\subset\mathbb{R},\;\{\chi_{j,i}\}_{j\in[0:N-2],i\in[0:N-1]}\subset\mathbb{R}.
\end{array}\right)\label{eq:NCG-PEP-upper-bound-final}
\end{equation}

\paragraph{Computing $\underline{\rho}_{N}(\cond,c)$ and corresponding counter-examples
\label{sec:comp-lower-bd}}

We now discuss how we can calculate $\underline{\rho}_{N}(\cond,c)$
and construct the corresponding ``bad'' function. This function
serves as a counter-example, illustrating scenarios where~\eqref{eq:NCG}
performs poorly. Once we have solved \eqref{eq:NCG-PEP-upper-bound-final},
it provides us with the corresponding CG update parameters, which
we denote by $\overline{\beta}_{i}$. If we can solve \eqref{eq:PEP_rho_N}
with the CG update parameters fixed to the $\overline{\beta}_{i}$
found from \eqref{eq:NCG-PEP-upper-bound-final}, then it will provide
us with the lower bound $\underline{\rho}_{N}(\mu,L,c)$. This process
also yields a ``bad'' function that acts as a counter-example, which
we explain next. Using the notation $g_{i}\triangleq\nabla f(x_{i})$
and $f_{i}\triangleq f(x_{i})$, then applying the homogeneity argument,
we can compute $\underline{\rho}_{N}(\cond,c)$ by finding a feasible
solution to the following optimization problem: 
\begin{equation}
\left(\begin{array}{ll}
\underset{\substack{\{x_{k+i}\}_i, \, \{d_{k+i}\}_{i}, \\ \{\gamma_{k+i}\}_{i}, \, f, \, n   } }{\mbox{maximize}} & f_{k+N}-f_{\star}\\
\textrm{subject to} & n\in\mathbb{N},\,f\in\mathcal{F}_{\mu,L}(\mathbb{R}^{n}),\\
 & x_{k+i},\,d_{k+i}\in\mathbb{R}^{n},\quad i\in[0:N]\\
 & \|\pg_{k}\|^{2}\leq c\|g_{k}\|^{2},\\
 & \gamma_{k+i}=\textrm{argmin}_{\gamma}f(x_{k+i}-\gamma\pg_{k+i}),\quad i\in[0:N-1],\\
 & x_{k+i+1}=x_{k+i}-\gamma_{k+i}\pg_{k+i},\quad i\in[0:N-1],\\
 & \pg_{k+i+1}=g_{k+i+1}+\overline{\beta}_{k+i}\pg_{k+i},\quad i\in[0:N-2],\\
 & \overline{\beta}_{k+i-1}=\frac{\|g_{k+i}\|^{2}-\eta\left\langle g_{k+i};\,g_{k+i-1}\right\rangle }{\|g_{k+i-1}\|^{2}},\quad i\in[1:N-1],\\
 & f_{k}-f_{\star}=1,
\end{array}\right)\label{eq:N-point-ncg-pep-lower-bound}
\end{equation}
Next, note that the NCGM iteration scheme
in \eqref{eq:N-point-ncg-pep-lower-bound} can be equivalently written
as:

\begin{equation}
\begin{aligned} & \chi_{j,i}=\chi_{j,i-1}\overline{\beta}_{k+i-1},\quad i\in[1:N-1],j\in[0:i-2]\\
 & \chi_{i-1,i}=\overline{\beta}_{k+i-1},\quad i\in[1:N-1]\\
 & \alpha_{i,i-1}=\gamma_{k+i-1},\quad i\in[1:N],\\
 & \alpha_{i,j}=\gamma_{k+j}+\sum_{\ell=j+1}^{i-1}\gamma_{k+\ell}\chi_{j,\ell},\quad i\in[1:N],\;j\in[0:i-2],\\
 & x_{k+i}=x_{k}-\sum_{j=1}^{i-1}\alpha_{i,j}g_{k+j}-\alpha_{i,0}\pg_{k},\quad i\in[1:N],\\
 & \pg_{k+i}=g_{k+i}+\sum_{j=1}^{i-1}\chi_{j,i}g_{k+j}+\chi_{0,i}\pg_{k},\quad i\in[1:N-1].
\end{aligned}
\label{eq:eqs-for-ncg-pep-full}
\end{equation}
where we have introduced intermediate variables $\chi_{j,i}$ and
$\alpha_{i,j}$ which will aid us in formulating \eqref{eq:N-point-ncg-pep-lower-bound}
as a nonconvex QCQP. Define $I_{N}^{\star}=\{\star,k,k+1,\ldots,k+N\}$.
Now using \eqref{eq:eqs-for-ncg-pep-full}, Theorem \ref{thm:interp},
and \eqref{eq:relaxed-line-search-cond}, we can equivalently write
\eqref{eq:overline_rho_0_shifted_inft_dm} as: 
\begin{equation}
\left(\begin{array}{ll}
\underset{\substack{\{x_{k+i},g_{k+i},f_{k+i}\}_{i}, \, n, \\ \{\gamma_{k+i}\}_{i}, \, \{\chi_{j,i}\}_{j,i}, \, \{\alpha_{i,j}\}_{i,j}}}{\mbox{maximize}} & f_{k+N}-f_{\star}\\
\textrm{subject to} & n\in\mathbb{N},\\
 & f_{i}\geq f_{j}+\langle g_{j} ;\, x_{i}-x_{j}\rangle+\frac{1}{2(1-\frac{\mu}{L})}\Big(\frac{1}{L}\|g_{i}-g_{j}\|^{2}\\
 & \;+\mu\|x_{i}-x_{j}\|^{2}-2\frac{\mu}{L}\left\langle g_{i}-g_{j};\,x_{i}-x_{j}\right\rangle \Big),\quad i,j\in I_{N}^{\star},\\
 & \|\pg_{k}\|^{2}\leq c\|g_{k}\|^{2},\\
 & \left\langle g_{k+i+1} ;\, \pg_{k+i}\right\rangle =0,\quad i\in[0:N-1],\\
 & \left\langle g_{k+i+1} ;\, x_{k+i}-x_{k+i+1}\right\rangle =0,\quad i\in[0:N-1],\textrm{ }\\
 & \left\langle g_{k+i} ;\, \pg_{k+i}\right\rangle =\|g_{k+i}\|^{2},\quad i\in[0:N-1],\\
 & \chi_{j,i}=\chi_{j,i-1}\overline{\beta}_{k+i-1},\quad i\in[1:N-1],j\in[0:i-2]\\
 & \chi_{i-1,i}=\overline{\beta}_{k+i-1},\quad i\in[1:N-1]\\
 & \alpha_{i,i-1}=\gamma_{k+i-1},\quad i\in[1:N],\\
 & \alpha_{i,j}=\gamma_{k+j}+\sum_{\ell=j+1}^{i-1}\gamma_{k+\ell}\chi_{j,\ell},\quad i\in[1:N],\;j\in[0:i-2],\\
 & x_{k+i}=x_{k}-\sum_{j=1}^{i-1}\alpha_{i,j}g_{k+j}-\alpha_{i,0}\pg_{k},\quad i\in[1:N],\\
 & \pg_{k+i}=g_{k+i}+\sum_{j=1}^{i-1}\chi_{j,i}g_{k+j}+\chi_{0,i}\pg_{k},\quad i\in[1:N-1].\\
 & \overline{\beta}_{k+i-1}=\frac{\|g_{k+i}\|^{2}-\eta\left\langle g_{k+i};\,g_{k+i-1}\right\rangle }{\|g_{k+i-1}\|^{2}},\quad i\in[1:N-1],\\
 & f_{k}-f_{\star}=1,\\
 & g_{\star}=0,\,x_{\star}=0,\,f_{\star}=0,\\
 & \{x_{i},g_{i},f_{i}\}_{i\in I_{N}^{\star}}\subset\mathbb{R}^{n}\times\mathbb{R}^{n}\times\mathbb{R},\;{\{d_{i}\}_{i\in[k+1:k+N-1]}\subset\mathbb{R}^{n}},\\
 & \{\chi_{j,i}\}_{j\in[0:N-2],i\in[0:N-1]}\subset\mathbb{R},\\
 & \{\gamma_{k+i}\}_{i\in[0:N]}\subset\mathbb{R},\;\{\alpha_{i,j}\}_{i\in[1:N],j\in[0:N-1]}\subset\mathbb{R}.
\end{array}\right)\label{eq:N-point-ncg-pep-lower-bound-fnt-dmnsnl-intractable}
\end{equation}
We introduce the Grammian transformation:

\begin{equation}
\begin{alignedat}{1}H & =[x_{k}\mid g_{k}\mid g_{k+1}\mid\ldots\mid g_{k+N}\mid\pg_{k}]\in\mathbb{R}^{n\times(N+3)},\\
G & =H^{\top}H\in\mathbb{S}_{+}^{N+3},\rank G\leq n,\\
F & =[f_{k}\mid f_{k+1}\mid\ldots\mid f_{k+N}]\in\mathbb{R}^{1\times(N+1)}.
\end{alignedat}
\label{eq:grammian-mats-ncg-1}
\end{equation}

Using the same arguments described in Section \ref{sec:worst-case-search-dir}, we again ignore the constraint $\rank G\leq n$
and can let $H\in\mathbb{R}^{(N+3)\times(N+3)}$ without loss of generality.
We next define the following notation for selecting columns and elements
of $H$ and $F$: 
\begin{equation}
\begin{alignedat}{1} & \mathbf{g}_{\star}=0\in\mathbb{R}^{N+3},\;\mathbf{g}_{k+i}=e_{i+2}\in\mathbb{R}^{N+3},\quad i\in[0:N],\\
 & \pgb_{k}=e_{N+3}\in\mathbb{R}^{N+3},\\
 & \mathbf{x}_{k}=e_{1}\in\mathbb{R}^{N+2},\;\mathbf{x}_{\star}=0\in\mathbb{R}^{N+2},\\
 & \mathbf{x}_{k+i}(\alpha)=\mathbf{x}_{k}-\sum_{j=1}^{i-1}\alpha_{i,j}\mathbf{g}_{k+j}-\alpha_{i,0}\pgb_{k}\in\mathbb{R}^{N+3},\quad i\in[1:N],\\
 & \pgb_{k+i}(\chi)=\mathbf{g}_{k+i}+\sum_{j=1}^{i-1}\chi_{j,i}\mathbf{g}_{k+j}+\chi_{0,i}\pgb_{k},\quad i\in[1:N-1],\\
 & \mathbf{f}_{\star}=0\in\mathbb{R}^{N+1},\;\mathbf{f}_{k+i}=e_{i+1}\in\mathbb{R}^{N+1},\quad i\in[0:N],
\end{alignedat}
\label{eq:bold-vec-ncg-1}
\end{equation}
which ensure $x_{i}=H\mathbf{x}_{i},\;g_{i}=H\mathbf{g}_{i},\;f_{i}=F\mathbf{f}_{i},\;\pg_{i}=H\pgb_{i}$
for $i\in I_{N}^{\star}$. For appropriate choices of matrices $A_{i,j}$,$B_{i,j}$,
$C_{i,j}$, $\widetilde{C}_{i,j}$, $D_{i,j}$, $\widetilde{D}_{i,j}$,
$E_{i,j}$, and vector $a_{i,j}$ as defined in \eqref{eq:effect-of-grammian-trs-ncg-Lyapunov-improved},
where $\mathbf{x}_{i}$, $\mathbf{g}_{i}$, $\mathbf{f}_{i}$, $\pgb_{i}$
are from \eqref{eq:bold-vec-ncg-1}, we can ensure that the identities
in \eqref{eq:ABCa-mat-vec-ncg-Lyapunov-improved} hold for all $i,j\in I_{N}^{\star}$.
Using those identities and using the definition of $G=H^{\top}H$,
where $H\in\mathbb{R}^{(N+3)\times(N+3)}$, we can write \eqref{eq:N-point-ncg-pep-lower-bound-fnt-dmnsnl-intractable}
as the following nonconvex QCQP: 
\begin{equation}
\left(\begin{array}{ll}
\underset{\substack{G, \, F, \, H, \\ \{\Theta_{i,j}\}_{i,j\in I_{N}^{\star}}, \, \gamma, \, \alpha, \, \chi}}{\mbox{maximize}} & Fa_{\star,N}\\
\textrm{subject to} & Fa_{i,j}+\tr G\left[A_{i,j}+\frac{1}{2(1-\frac{\mu}{L})}\left(\frac{1}{L}C_{i,j}+\mu\Theta_{i,j}-2\frac{\mu}{L}E_{i,j}\right)\right]\leq0,\quad i,j\in I_{N}^{\star},\\
 & \Theta_{i,j}=B_{i,j},\quad i,j\in I_{N}^{\star},\\
 & \tr G\widetilde{C}_{k,\star}\leq c\tr GC_{k,\star},\\
 & \tr G\widetilde{D}_{k+i+1,k+i}=0,\quad i\in[0:N-1],\\
 & \tr GA_{k+i,k+i+1}=0,\quad i\in[0:N-1],\\
 & \tr G\widetilde{D}_{k+i,k+i}=\tr GC_{k+i,\star}\quad i\in[0:N-1],\\
 & \chi_{j,i}=\chi_{j,i-1}\overline{\beta}_{k+i-1},\quad i\in[1:N-1],j\in[0:i-2]\\
 & \chi_{i-1,i}=\overline{\beta}_{k+i-1},\quad i\in[1:N-1]\\
 & \alpha_{i,i-1}=\gamma_{k+i-1},\quad i\in[1:N],\\
 & \alpha_{i,j}=\gamma_{k+j}+\sum_{\ell=j+1}^{i-1}\gamma_{k+\ell}\chi_{j,\ell},\quad i\in[1:N],\;j\in[0:i-2],\\
 & \overline{\beta}_{k+i-1}\times\tr GC_{k+i-1,\star}=\tr G\left(C_{k+i,\star}-\eta D_{k+i,k+i-1}\right),\quad i\in[1:N-1],\\
 & Fa_{\star,k}=1,\\
 & G=H^{\top}H,\\
 & F\in\mathbb{R}^{N+1},\;G\in\mathbb{S}^{N+3},\;H\in\mathbb{R}^{(N+3)\times(N+3)},\\
 & \{\chi_{j,i}\}_{j\in[0:N-2],i\in[0:N-1]}\subset\mathbb{R},\\
 & \{\gamma_{k+i}\}_{i\in[0:N]}\subset\mathbb{R},\;\{\alpha_{i,j}\}_{i\in[1:N],j\in[0:N-1]}\subset\mathbb{R}.
\end{array}\right)\label{eq:NCG-PEP-lower-bound}
\end{equation}
Note that $\{\Theta_{i,j}\}_{i,j\in I_{N}^{\star}}$
is introduced as a separate decision variable to formulate the cubic
constraints arising from $B_{i,j}$ as quadratic constraints. Also,
to compute $\underline{\rho}_{N}(\cond,c)$, it suffices to find just
a feasible solution to \eqref{eq:NCG-PEP-lower-bound}, in Section
\ref{app:ncg-pep-alg} we will discuss how to do so using our custom
spatial branch-and-bound algorithm.

\subsubsection{Nonconvex QCQP reformulation of \eqref{eq:PEP_rho_N0}}

\label{sec:reform_rho_N0}

Now we discuss how we compute the upper bound $\overline{\rho}_{N,0}(\cond)$
and lower bound $\underline{\rho}_{N,0}(\cond)$ to $\rho_{N,0}(\cond)$
defined in \eqref{eq:PEP_rho_N0}. The bound computation process is
very similar to that of \eqref{eq:PEP_rho_N}. Observe that, in \eqref{eq:PEP_rho_N},
if we remove the constraint $\|d_{k}\|^{2}\leq c\|\nabla f(x_{k})\|^{2},$
set $k\triangleq0$ , and then add the constraint $d_{0}=\nabla f(x_{0})$,
then it is identical to \eqref{eq:PEP_rho_N0} (the constraint $\left\langle \nabla f(x_{0});\,d_{0}\right\rangle =\|\nabla f(x_{0})\|^{2}$
in \eqref{eq:PEP_rho_N} is a valid but redundant constraint for \eqref{eq:PEP_rho_N0}).

So, to compute the upper bound $\overline{\rho}_{N,0}(\cond)$, we
can follow a transformation process very similar to \Cref{sec:comp-upper-bd}
but with a few changes. In \eqref{eq:overline_rho_0_shifted_inft_dm}
and \eqref{eq:overline_rho_0_shifted_fnt_dm_intract}, we remove the
constraint $\|\pg_{k}\|^{2}\leq c\|g_{k}\|^{2}$, and then add the
constraint $g_{k}=\pg_{k}$. Second, the Grammian matrices defined
in \eqref{eq:grammian-mats-ncg-Lyapunov-1} stays the same, and in
\eqref{eq:rho_overline_bold_vectors} the vectors $\{\mathbf{x}_{i},\mathbf{g}_{i},\mathbf{f}_{i}\}_{i\in I_{N}^{\star}}$
stays the same except we set $\mathbf{d}_{k}=\mathbf{g}_{k}=e_{2}\in\mathbb{R}^{2N+3}$,
which ensures that $d_{k}=F\mathbf{d}_{k}=g_{k}$. We then remove
the constraint $\tr G\widetilde{C}_{k,\star}\leq c\tr GC_{k,\star}$
from \eqref{eq:NCG-PEP-upper-bound-final} and finally set $k\triangleq0$
in the resultant QCQP. The solution to the nonconvex QCQP will provide
us the upper bound $\overline{\rho}_{N,0}(\cond)$ in \eqref{eq:PEP_rho_N0}.

To compute the lower bound $\underline{\rho}_{N,0}(\cond)$, we follow
the same set of changes described in the last paragraph but to \eqref{eq:N-point-ncg-pep-lower-bound}
in \Cref{sec:comp-lower-bd}.

\subsubsection{The relative gap between the lower bounds and upper bounds}

\label{sec:rel_gap_tables}

Tables \ref{tab:PRP_default}, \ref{tab:PRP_restarted}, \ref{tab:FR_restarted}
record the relative gap between lower bounds and upper bounds for
a few representative values of $q$ obtained by solving the aforementioned
nonconvex QCQPs associated with \eqref{eq:PEP_rho_N} and \eqref{eq:PEP_rho_N0}
using our custom spatial branch-and-bound algorithm described in Section
\ref{app:ncg-pep-alg}. Note that the tables contain a few negative
entries close to zero which are due to the absolute gap being of the
same order as the accuracy of the solver ($1\textrm{e}-6$). For the
full list for all values, we refer to our open-source code, which
also allows for computing these bounds for a user-specified value
of $q$ as well. In all cases, the relative gap is less than $10\%$.
In most cases, it is significantly better.

\begin{table}[!ht]
\centering{}%
\begin{tabular}{cccccccccc}
\specialrule{2pt}{1pt}{1pt} {$\cond=$} & 0.001 & 0.005 & 0.02 & 0.04 & 0.06 & 0.08 & 0.1 & 0.3 & 0.5\tabularnewline
\specialrule{2pt}{1pt}{1pt} $N=1$ & $3\textrm{e}{-8}$ & $-1\textrm{e}{-6}$ & $3\textrm{e}{-9}$ & $6\textrm{e}{-8}$ & $9\textrm{e}{-8}$ & $2\textrm{e}{-7}$ & $2\textrm{e}{-7}$ & $1\textrm{e}{-6}$ & $3\textrm{e}{-7}$\tabularnewline
$N=2$ & $2\textrm{e}{-6}$ & $6\textrm{e}{-7}$ & $-3\textrm{e}{-8}$ & $9\textrm{e}{-8}$ & $1\textrm{e}{-7}$ & $8\textrm{e}{-8}$ & $3\textrm{e}{-7}$ & $8\textrm{e}{-3}$ & $4\textrm{e}{-4}$\tabularnewline
$N=3$ & $5\textrm{e}{-6}$ & $5\textrm{e}{-4}$ & $7\textrm{e}{-3}$ & $2\textrm{e}{-2}$ & $3\textrm{e}{-2}$ & $4\textrm{e}{-2}$ & $2\textrm{e}{-2}$ & $5\textrm{e}{-2}$ & $-3\textrm{e}{-7}$\tabularnewline
$N=4$ & $2\textrm{e}{-4}$ & $3\textrm{e}{-3}$ & $2\textrm{e}{-2}$ & $7\textrm{e}{-2}$ & $1\textrm{e}{-1}$ & $3\textrm{e}{-2}$ & $4\textrm{e}{-2}$ & $4\textrm{e}{-2}$ & $4\textrm{e}{-2}$\tabularnewline
\specialrule{2pt}{1pt}{1pt} &  &  &  &  &  &  &  &  & \tabularnewline
\end{tabular}

\caption{Relative gaps $\nicefrac{\overline{\rho}_{N}(\cond,c)-\underline{\rho}_{N}(\cond,c)}{\overline{\rho}_{N}(\cond,c)}$
for {\PRP} with $c=\nicefrac{(1+\cond)^{2}}{4\cond}$.}
\label{tab:PRP_default}
\end{table}

\begin{table}[!ht]
\centering{}%
\begin{tabular}{cccccccccc}
\specialrule{2pt}{1pt}{1pt} {$\cond=$} & 0.001 & 0.005 & 0.02 & 0.04 & 0.06 & 0.08 & 0.1 & 0.3 & 0.5\tabularnewline
\specialrule{2pt}{1pt}{1pt} $N=2$ & $7\textrm{e}{-6}$ & $2\textrm{e}{-4}$ & $2\textrm{e}{-3}$ & $7\textrm{e}{-3}$ & $1\textrm{e}{-2}$ & $1\textrm{e}{-2}$ & $2\textrm{e}{-2}$ & $1\textrm{e}{-2}$ & $1\textrm{e}{-6}$\tabularnewline
$N=3$ & $5\textrm{e}{-5}$ & $9\textrm{e}{-4}$ & $1\textrm{e}{-2}$ & $3\textrm{e}{-2}$ & $5\textrm{e}{-2}$ & $6\textrm{e}{-2}$ & $6\textrm{e}{-2}$ & $5\textrm{e}{-3}$ & $-7\textrm{e}{-6}$\tabularnewline
$N=4$ & $3\textrm{e}{-4}$ & $4\textrm{e}{-3}$ & $3\textrm{e}{-2}$ & $4\textrm{e}{-2}$ & $9\textrm{e}{-2}$ & $9\textrm{e}{-2}$ & $7\textrm{e}{-2}$ & $3\textrm{e}{-2}$ & $7\textrm{e}{-2}$\tabularnewline
\specialrule{2pt}{1pt}{1pt} &  &  &  &  &  &  &  &  & \tabularnewline
\end{tabular}

\caption{Relative gap $\nicefrac{\overline{\rho}_{N,0}(q)-\underline{\rho}_{N,0}(q)}{\overline{\rho}_{N,0}(q)}$
for {\PRP} where $N=2,3,4$. The case $N=1$ is omitted, as {\PRP}
is equivalent to {\GDEL} in this case.}
\label{tab:PRP_restarted}
\end{table}

\begin{table}[!ht]
\centering{}%
\begin{tabular}{cccccccccc}
\specialrule{2pt}{1pt}{1pt} {$\cond=$} & 0.001 & 0.005 & 0.02 & 0.04 & 0.06 & 0.08 & 0.1 & 0.3 & 0.5\tabularnewline
\specialrule{2pt}{1pt}{1pt} $N=2$ & $9\textrm{e}{-6}$ & $2\textrm{e}{-4}$ & $1\textrm{e}{-3}$ & $7\textrm{e}{-3}$ & $1\textrm{e}{-2}$ & $1\textrm{e}{-2}$ & $2\textrm{e}{-2}$ & $1\textrm{e}{-2}$ & $8\textrm{e}{-7}$\tabularnewline
$N=3$ & $7\textrm{e}{-5}$ & $1\textrm{e}{-3}$ & $1\textrm{e}{-2}$ & $2\textrm{e}{-2}$ & $3\textrm{e}{-2}$ & $3\textrm{e}{-2}$ & $3\textrm{e}{-2}$ & $3\textrm{e}{-7}$ & $-1\textrm{e}{-7}$\tabularnewline
$N=4$ & $2\textrm{e}{-4}$ & $3\textrm{e}{-3}$ & $2\textrm{e}{-2}$ & $3\textrm{e}{-2}$ & $3\textrm{e}{-2}$ & $2\textrm{e}{-2}$ & $1\textrm{e}{-2}$ & $1\textrm{e}{-6}$ & $4\textrm{e}{-2}$\tabularnewline
\specialrule{2pt}{1pt}{1pt} &  &  &  &  &  &  &  &  & \tabularnewline
\end{tabular}

\caption{The relative gap $\nicefrac{\overline{\rho}_{N,0}(q)-\underline{\rho}_{N,0}(q)}{\overline{\rho}_{N,0}(q)}$
for {\FR} where $N=2,3,4$. The case $N=1$ is omitted again, as
in this case {\FR} is equivalent to {\GDEL}.}
\label{tab:FR_restarted}
\end{table}

The next sections discuss and draw a few conclusions from the numerical
worst-case convergence results for~PRP and~FR.

}

\subsection{Improved worst-case bounds for PRP}\label{sec:num-res-prp}

Figure~\ref{fig:PRP_Lyapunov} reports the worst-case values of the ``Lyapunov'' ratio $\nicefrac{f(x_{k+N} ) - f_\star}{f(x_{k})-f_\star}$ as a function of the inverse condition number $q\triangleq \nicefrac{\mu}{L}$ and for $c=\nicefrac{(1+\cond)^{2}}{4\cond}$ and $N=1,2,3,4$. This worst-case ratio seems to improve as $N$ grows, but does not outperform gradient descent with exact line search ({\GDEL}). The diminishing improvements with $N$ also suggests the worst-case performance of {\PRP} in this regime might not outperform {\GDEL} even for larger values of $N\geq 4$, albeit probably getting close to the same asymptotic worst-case convergence rate.

\begin{figure}[!ht]
\centering %
\input{prp_Lyapunov_master.tex}
\caption{This figure reports the worst-case values for the ``Lyapunov'' ratio $\sqrt[N]{\nicefrac{f(x_{k+N} ) - f_\star}{f(x_{k})-f_\star}}$ vs.~the condition number $\cond\triangleq \nicefrac{\mu}{L}$ for {\PRP}. We compute $\rho_N(\cond,c)$ with $c=\nicefrac{(1+\cond)^{2}}{4\cond}$ for
$N=1,2,3,4$. As $N$ increases, the worst-case $\sqrt[N]{\nicefrac{f_{k+N}-f_{\star}}{f_{k}-f_{\star}}}$
improves, but remains worse than that of gradient descent with exact line search ({\GDEL}). The curve $(1-\sqrt{q})^2$ (orange) corresponds to the rate of the lower complexity bounds for this class of problems~\cite{drori2022oracle}.\label{fig:PRP_Lyapunov}}
\end{figure}
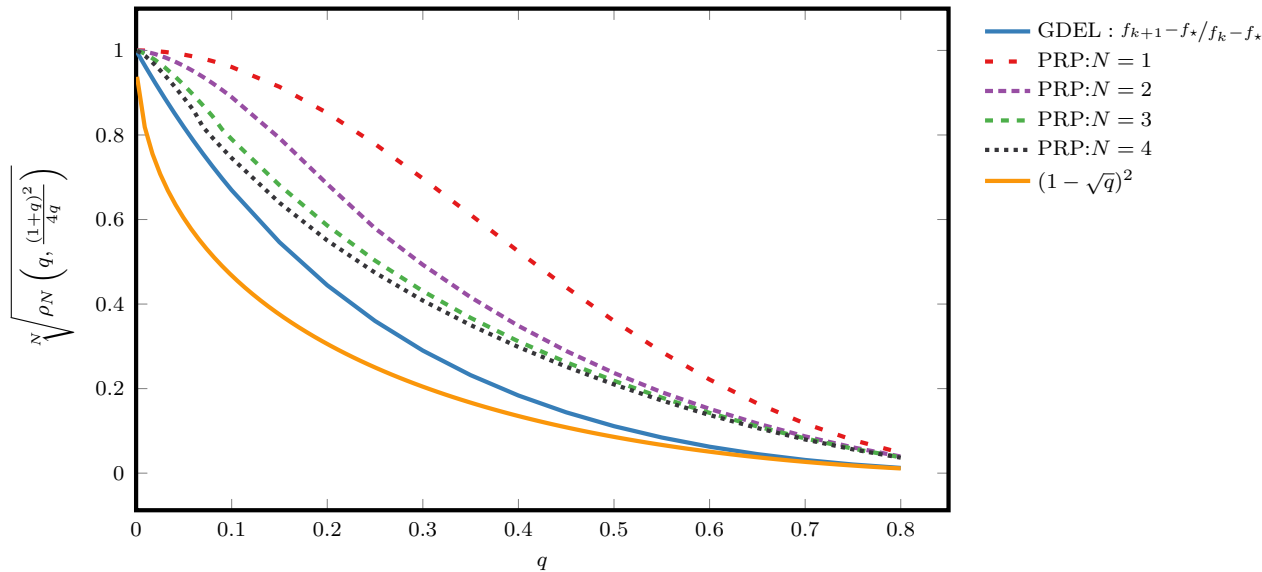 

As a complement, Figure~\ref{fig:PRP_exact} shows how {\PRP}'s worst-case ratio $\nicefrac{f_{N}-f_{\star}}{f_{0}-f_{\star}}$ evolves as a function of $q$ for $N=1,2,3,4$. The worst-case performance of {\PRP} in this setup seems to be similar to that of {\GDEL}. Further, for small $q$ (which is typically the only regime of interest for large-scale optimization), {\PRP}'s worst-case performance seems to be slightly better {than that} of~{\GDEL}. On the other hand, for larger~$q$, {\PRP} performs slightly worse than {\GDEL}.

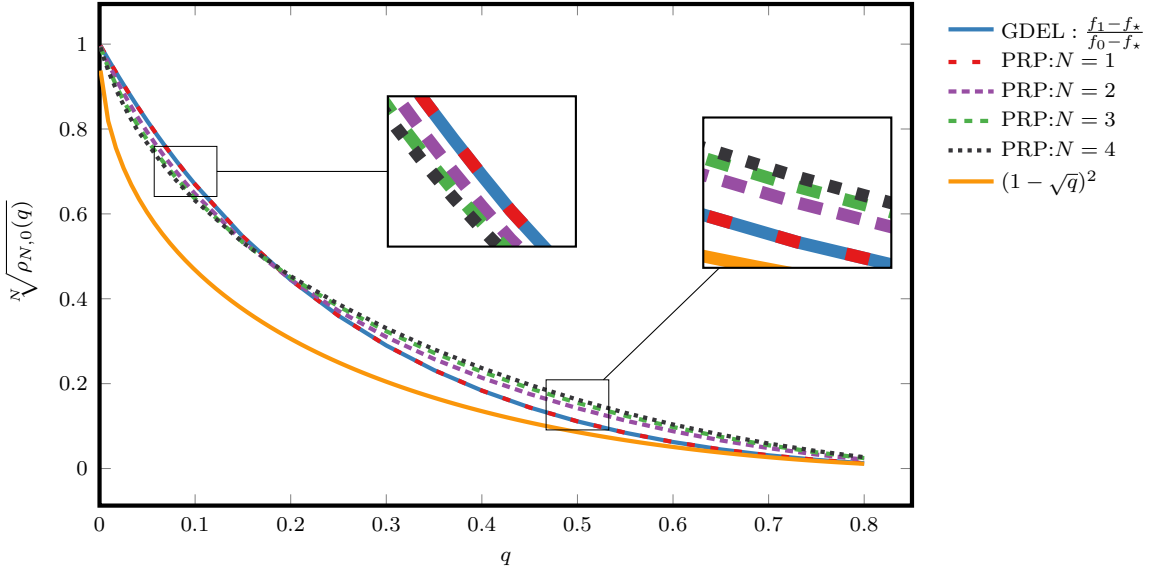
\begin{figure}[!ht]
\centering %
\input{prp_restarted_master.tex}
\caption{This figure reports the worst-case values for the ratio $\sqrt[N]{\nicefrac{f_{N}-f_{\star}}{f_{0}-f_{\star}}}$ vs.~$q$ for {\PRP} for $N=1,2,3,4$. For $N=1$, {\PRP} and {\GDEL} perform the same iteration. For $N=2,3,4$, the worst-case ratio of {\PRP} is better than that of {\GDEL} for $ \cond \leq 0.1$. The curve $(1-\sqrt{q})^2$ (orange) corresponds to the rate of the lower complexity bounds for this class of problems~\cite{drori2022oracle}.\label{fig:PRP_exact}}
\end{figure} 

As a conclusion, we believe there is no hope to prove uniformly better worst-case bounds for~PRP than those for~GDEL for smooth strongly convex minimization. However, we might be able to prove improvements for small values of $q$ at the cost of possibly very technical proofs. As for the Lyapunov approach, the numerical results from this section could be improved by further increasing $N$, but we believe that the transient behavior  does not suggest this direction to be promising. We recall that we computed the bounds by solving an optimization problem whose feasible points correspond to worst-case examples. Therefore, the numerical results provided in this section are backed-up by numerically constructed examples on which PRP behaves ``badly''.

\subsection{Improved worst-case bounds for {\FR}} \label{subsec:Numerical-results-for-FR}

Figure~\ref{fig:FR_exact} reports the worst-case values for the ratio $\nicefrac{f_{N}-f_{\star}}{f_{0}-f_{\star}}$ as a function of~$q$, for $N\in\{1,2,3,4\}$. The convergence bounds appears to be marginally better than {\GDEL} for some sufficiently small inverse condition numbers. Further, the range of values of $q$ for which there is an improvement appears to be decreasing with~$N\geq 2$. Beyond this range, the worst-case values become significantly worse than that of {\GDEL}. Though apparently not as dramatic as the worst-case bound from  Theorem~\ref{thm:FR_2}, the quality of the bound appears to be decreasing with $N$, which stands in line with the practical need to restart the method~\cite{hager2006survey}.

\begin{figure}[!ht]
\centering %
\input{fr_restarted_master.tex}
\caption{This figure reports the worst-case values for the ratio $\sqrt[N]{\nicefrac{f_{N}-f_{\star}}{f_{0}-f_{\star}}}$ vs.~$q$ for {\FR} for $N=1,2,3,4$. For $N=1$, {\FR} and {\GDEL} perform the same iteration. For $N=2,3,4$, the worst-case bound for {\FR} is slightly better than that of {\GDEL} for small enough values of~$\cond$,
and gets larger than {\GDEL} for larger values of~$\cond$. The range of $\cond$ for which {\FR} is better
than {\GDEL} gets smaller as $N\geq 2$ increases. The curve $(1-\sqrt{q})^2$ (orange) corresponds to the rate of the lower complexity bounds for this class of problems~\cite{drori2022oracle}.\label{fig:FR_exact}}
\end{figure}
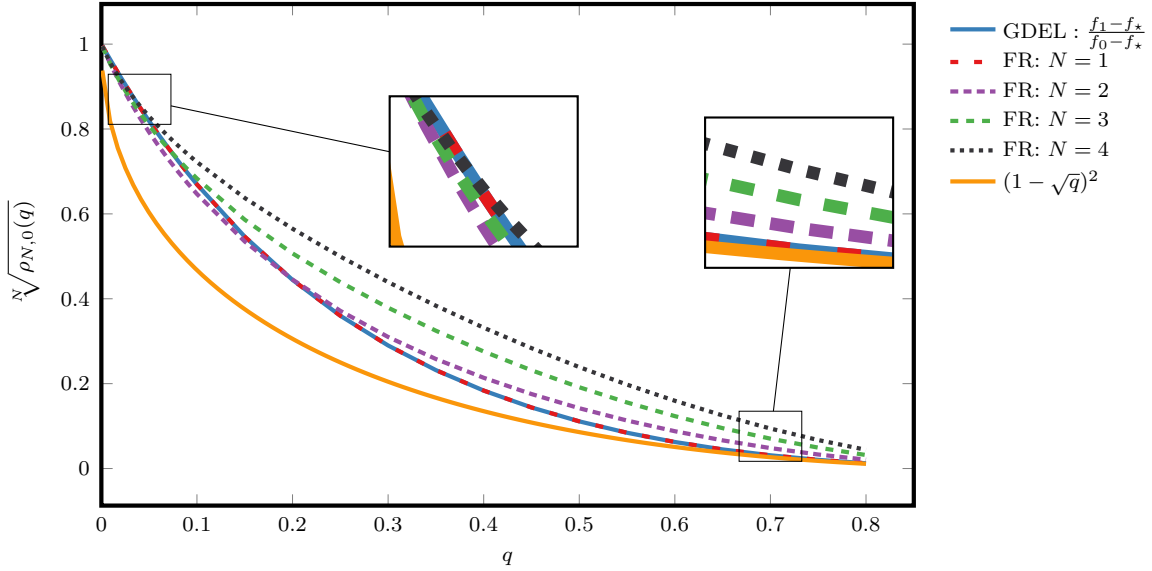 

As in the previous section, we recall that those curves were obtained by numerically constructing ``bad'' worst-case counter-examples satisfying our assumptions. In other words, there is no hope to obtain better results without adding assumptions or changing the types of bounds under consideration.

{

\section{Custom spatial branch-and-bound algorithm}\label{app:ncg-pep-alg}

This section discusses implementation details for solving
the nonconvex QCQPs of this paper (namely~\eqref{eq:two-point-adapep-qcqp},~\eqref{eq:NCG-PEP-upper-bound-final}, or~\eqref{eq:NCG-PEP-lower-bound}) using a custom spatial branch-and-bound method. This strategy proceeds in three stages, as follows. 
\begin{itemize}
\item \textbf{Stage 1: Compute a feasible solution.} First,
we construct a feasible solution to the nonconvex QCQP. We do that by generating a random $\mu$-strongly convex and $L$-smooth quadratic
function, and by applying the corresponding nonlinear conjugate gradient method on it. The corresponding iterates, gradient and function values correspond to a feasible point for the nonconvex QCQPs under consideration.

\item \textbf{Stage 2: Compute a locally optimal solution by warm-starting
at Stage 1 solution. }Stage 2 computes a locally optimal solution
to the nonconvex QCQPs using an interior-point algorithm, warm-starting
at the feasible solution produced by Stage 1. When a good warm-starting
point is provided, interior-point algorithms can quickly converge
to a locally optimal solution under suitable regularity conditions
\cite{byrd1997local,wachter2006implementation}, \cite[\S 3.3]{fiacco1990nonlinear}.
In the situation where the interior-point algorithm fails to converge,
we go back to the feasible solution from Stage 1. We
have empirically observed that Stage 2 consistently provides a locally
optimal solution. 
\item \textbf{Stage 3: Compute a globally optimal solution by warm-starting
at Stage 2 solution.} Stage 3 computes a globally optimal solution
to the nonconvex QCQP using a spatial branch-and-bound algorithm \cite{Gurobi,locatelli2013global},
warm-starting at the locally-optimal solution produced by Stage 2.
For details about how spatial branch-and-bound algorithm works, we
refer the reader to \cite[\S 4.1]{gupta2022branch}.
\end{itemize}

\begin{remark}
    In stage 3, the most numerically challenging nonconvex quadratic constraint in
\eqref{eq:two-point-adapep-qcqp}, \eqref{eq:NCG-PEP-upper-bound-final}
or \eqref{eq:NCG-PEP-lower-bound} is $G=H^{\top} H$. To solve those problems in reasonable times, we use the \emph{lazy constraints} approach,
\cite[\S 4.2.5]{gupta2022branch}.

In short, we replace the constraint {$G=H^{\top}H$} by the
infinite set of linear constraints $\tr\left(Gyy^{\top}\right)\geq0$
for all $y$, which we then sample to obtain a finite set of linear constraints (we recursively add additional linear constraints afterwards if need be). More precisely, we use
\begin{equation}
\tr\left(Gyy^{\top}\right)\geq0,\quad y\in Y,\label{eq:starting-set-lazy}
\end{equation}
where the initial $Y$ is generated randomly as a set of 
unit vectors following the methodology described in \cite[$\mathsection$5.1]{benson2003solving}.
By replacing $G=H^{\top} H$ by~\eqref{eq:starting-set-lazy} we obtain a simpler (but relaxed) QCQP. Then, we update the solution
$G$ lazily by repeating the following steps until $G\succeq0$ is
satisfied subject to a termination criterion. Practically speaking, our termination criterion is that the minimal eigenvalue of $G$ is larger than $\epsilon\approx-1\textrm{e}-6$; until then, we repeat the following procedure: 
\begin{enumerate}
\item Solve the relaxation of the nonconvex QCQPs, where~\eqref{eq:starting-set-lazy}
is used instead of $G=H^{\top} H$, which provides us an upper bound on the original nonconvex QCQP. 
\item Compute the minimal eigenvalue $\textup{eig}_{\textup{min}}(G)$ and the corresponding eigenvector $u$ of $G$. If $\textup{eig}_{\textup{min}}(G)\ge0$,
we reached an optimal solution to the nonconvex QCQP and we terminate. 
\item If $\textup{eig}_{\textup{min}}(G)<0$, we add a constraint
$\tr(Guu^{\top})\geq0$ \emph{lazily,} which makes the current $G$
infeasible for the new relaxation. We use the lazy constraint
callback interface of $\textup{\texttt{JuMP}}$ to add constraints lazily, which
means that after adding one additional linear constraint, updating
the solution in step~1 is efficient since $\textup{\texttt{Gurobi}}$ and all
modern solvers based on the simplex algorithm can quickly update a
solution when only one linear constraint is added \cite[pp. 205-207]{bertsimas1997introduction}. 
\end{enumerate}
\end{remark}

}

\section{Conclusion}\label{sec:conclusion}

This works studies the iteration complexity of two variants of nonlinear conjugate gradients, namely the Polak-Ribière-Polyak (PRP) and the Fletcher-Reeves (FR) methods. We provide novel complexity bounds for both those methods, and show that albeit unsatisfying, not much can a priori be gained from a worst-case perspective, as both methods appear to behave similar or worse to regular steepest descent in the worst-case. Further, those results suggest that explaining the good practical performances of NCGMs might be out of reach for traditional worst-case complexity analyses on classical classes of problems.

This work considers only somewhat ``ideal'' variants of nonlinear conjugate gradient methods, as we make explicit use of exact line search procedures. However, there is a priori no reason to believe that inexact line search procedures would improve the possibly bad worst-case behaviors. Further, the \emph{performance estimation} methodology allows taking such inexact line search procedures into account, so the same methodology could be applied for tackling those questions. We leave such investigations for future work.

\section*{Acknowledgments}
S. Das Gupta and R.~M.~Freund acknowledge support by AFOSR Grant No.~FA9550-22-1-0356. A. Taylor acknowledges support from the European Research Council (grant SEQUOIA 724063). This work was partly funded by the French government under management of Agence Nationale de la Recherche as part of the ``Investissements d’avenir'' program, reference ANR-19-P3IA-0001 (PRAIRIE 3IA Institute).

The authors thank Nizar Bousselmi and Ian Ruffolo for careful reading of the manuscript and constructive feedback.

\bibliographystyle{unsrt}
\bibliography{bib}

\clearpage
\appendix

\section*{Organization of the appendix}

In what follows, we report detailed numerical results and computations that are not presented in the core of the paper. Table~\ref{tab:appendix_summary} details the organization of this additional material.

			\begin{table}[!ht]
				\begin{center}
					{\renewcommand{\arraystretch}{1.4}
						\begin{tabular}{@{}ll@{}}
							\specialrule{2pt}{1pt}{1pt}
							Section & Content \\
							\specialrule{2pt}{1pt}{1pt}
							\multirow{2}{.2\linewidth} 
                                {\appref{sec:num_vec_wcsd} } 
                                &  Numerical illustration of tightness of the worst-case search direction  \\
							& \eqref{eq:PRP_angle} for {\PRP} and \eqref{eq:FR_angle} for {\FR}.\\
                                \hline
             \appref{app:reformulation_lemma} & Reformulations for weighted sum of inequalities for Lemmas \ref{thm:PRP_1}, \ref{Bound-on-beta}, \ref{thm:FR_1} \\
                             \hline
             \appref{subsec:counter-example} & Constructing counter-examples \\
                               
							\specialrule{2pt}{1pt}{1pt}
						\end{tabular}
					}
				\end{center}
				\caption{Organization of the appendix.}\label{tab:appendix_summary}
			\end{table}

\section{Tightness of the worst-case search directions}

\label{sec:num_vec_wcsd}

{Table~\ref{table:PRP_ratio_comp} and Table~\ref{table:FR_ratio_comp} illustrate
the tightness of the bounds~\eqref{eq:PRP_angle} and~\eqref{eq:FR_angle} for~PRP and~FR respectively. That is, we compare the absolute relative difference between the numerical bounds and closed-form bounds for a few different values of $\cond$ and $c_{k-1}$, where numerical bounds are obtained by solving \eqref{eq:two-point-adapep-qcqp} with $\eta=1$ for {\PRP} and $\eta=0$
for {\FR}. These numerical examples strongly suggest
that our bounds cannot be improved in general. }

\begin{table}[!ht]
	\centering{}%
	\begin{tabular}{cccccccccc}
		\specialrule{2pt}{1pt}{1pt} {$\cond=$} & 0.1 & 0.2 & 0.3 & 0.4 & 0.5 & 0.6 & 0.7 & 0.8 \tabularnewline
		\specialrule{2pt}{1pt}{1pt} $c_{k-1} = 1.01$ & $7\textrm{e}{-9}$ & $3\textrm{e}{-7}$ & $1\textrm{e}{-8}$ & $4\textrm{e}{-8}$ & $1\textrm{e}{-8}$ & $2\textrm{e}{-8}$ & $5\textrm{e}{-8}$ & $1\textrm{e}{-8}$ \tabularnewline
		
		$c_{k-1} = 10$ & $4\textrm{e}{-8}$ & $1\textrm{e}{-7}$ & $ 9\textrm{e}{-8}$ & $3\textrm{e}{-8}$ & $7\textrm{e}{-7}$ & $3\textrm{e}{-8}$ & $7\textrm{e}{-9}$ & $4\textrm{e}{-8}$ \tabularnewline
		
		$c_{k-1} = 50$ & $2\textrm{e}{-8}$ & $1\textrm{e}{-7}$ & $2\textrm{e}{-8}$ & $1\textrm{e}{-8}$ & $6\textrm{e}{-8}$ & $7\textrm{e}{-8}$ & $1\textrm{e}{-9}$ & $9\textrm{e}{-8}$ \tabularnewline
		\specialrule{2pt}{1pt}{1pt} &  &  &  &  &  &  &  &  \tabularnewline
	\end{tabular}
	
	\caption{Absolute relative differences in the worst-case analytical bound~\eqref{eq:PRP_angle} and numerical bounds from~\eqref{eq:two-point-adapep-qcqp} with $\eta=1$ (for {\PRP}) for different values of $q$ and
		$c_{k-1}$.  \label{table:PRP_ratio_comp}}
\end{table}

\begin{table}[!ht]
	\centering{}%
	\begin{tabular}{cccccccccc}
		\specialrule{2pt}{1pt}{1pt} {$\cond=$} & 0.1 & 0.2 & 0.3 & 0.4 & 0.5 & 0.6 & 0.7 & 0.8 \tabularnewline
		\specialrule{2pt}{1pt}{1pt} $c_{k-1} = 1.01$ & $3\textrm{e}{-8}$ & $6\textrm{e}{-8}$ & $3\textrm{e}{-8}$ & $9\textrm{e}{-10}$ & $3\textrm{e}{-9}$ & $9\textrm{e}{-10}$ & $6\textrm{e}{-8}$ & $5\textrm{e}{-8}$ \tabularnewline
		
		$c_{k-1} = 10$ & $3\textrm{e}{-9}$ & $2\textrm{e}{-8}$ & $ 2\textrm{e}{-8}$ & $2\textrm{e}{-8}$ & $1\textrm{e}{-8}$ & $7\textrm{e}{-9}$ & $6\textrm{e}{-8}$ & $4\textrm{e}{-8}$ \tabularnewline
		
		$c_{k-1} = 50$ & $8\textrm{e}{-9}$ & $2\textrm{e}{-9}$ & $9\textrm{e}{-10}$ & $1\textrm{e}{-8}$ & $1\textrm{e}{-8}$ & $9\textrm{e}{-7}$ & $4\textrm{e}{-7}$ & $5\textrm{e}{-7}$ \tabularnewline
		\specialrule{2pt}{1pt}{1pt} &  &  &  &  &  &  &  &  \tabularnewline
	\end{tabular}
	
	\caption{Absolute relative differences in the worst-case analytical bound~\eqref{eq:FR_angle} and  numerical bounds from~\eqref{eq:two-point-adapep-qcqp} with $\eta=0$ (for {\FR}) for different values of $q$ and
		$c_{k-1}$.  \label{table:FR_ratio_comp}}
\end{table}

\clearpage

{

\section{Reformulations for weighted sum of inequalities for Lemmas~\ref{thm:PRP_1},~\ref{Bound-on-beta},~\ref{thm:FR_1}}\label{app:reformulation_lemma}

\subsection{Reformulation for weighted sum of inequalities for Lemma \ref{thm:PRP_1}}\label{app:prp_reform}

For notational ease, define $f(x_{k})\triangleq f_{k}$, $f(x_{k-1})\triangleq f_{k-1}$,$\nabla f({x_{k}})\triangleq g_{k}$,
$\nabla f(x_{k-1})\triangleq g_{k-1}$, $\beta_{k-1}\triangleq\beta$,
$\gamma_{k-1}\triangleq\gamma$, and $c_{k-1}\triangleq c$. We want
to show that 

\begin{align} & -\frac{\beta^{2}(q+1)}{\gamma Lq}\langle g_{k} ;\, d_{k-1}\rangle \nonumber \\
 & +\frac{\beta^{2}(q+1)^{2}}{\gamma^{2}L(1-q)q}\bigg[f_{k}-f_{k-1}+\gamma\langle g_{k};\,d_{k-1}\rangle+\tfrac{1}{2L}\|g_{k-1}-g_{k}\|^{2}  \nonumber\\
 & +\tfrac{\mu}{2(1-\mu/L)}\|\gamma d_{k-1}-\tfrac{1}{L}(g_{k-1}-g_{k})\|^{2}\bigg]  \nonumber \\
 & +\frac{\beta^{2}(q+1)^{2}}{\gamma^{2}L(1-q)q}\bigg[f_{k-1}-f_{k}-\gamma\langle g_{k-1};\,d_{k-1}\rangle+\tfrac{1}{2L}\|g_{k-1}-g_{k}\|^{2} \nonumber\\
 & +\tfrac{\mu}{2(1-\mu/L)}\|\gamma d_{k-1}-\tfrac{1}{L}(g_{k-1}-g_{k})\|^{2}\bigg] \nonumber\\
 & +\frac{\beta(q+1)}{\gamma Lq}\big[\langle g_{k-1};\,g_{k}\rangle-\|g_{k}\|^{2}+\beta\langle g_{k-1};\,d_{k-1}\rangle\big]  \label{eq:prp_big_term_1}
\end{align}

is equal to 
\begin{align} & \|d_{k}\|^{2}-\frac{(1+q)^{2}}{4q}\|g_{k}\|^{2} \nonumber \\
 & +\frac{4\beta^{2}q}{(1-q)^{2}}\left\Vert d_{k-1}-\tfrac{1+q}{2L\gamma_{k-1}q}g_{k-1}+\tfrac{2\beta(1+q)-L\gamma(1-q)^{2}}{4\beta L\gamma q}g_{k}\right\Vert ^{2}. \label{eq:prp_simp_term}
\end{align}

We show this by expanding both terms and then doing term-by-term matching. 

\paragraph{Expand the second summand of (\ref{eq:prp_big_term_1}).}

First, we expand the second summand of (\ref{eq:prp_big_term_1})
as follows.

\begin{align} & \frac{\beta^{2}(q+1)^{2}}{\gamma^{2}L(1-q)q}\bigg[f_{k}-f_{k-1}+\gamma\langle g_{k};\,d_{k-1}\rangle+\tfrac{1}{2L}\|g_{k-1}-g_{k}\|^{2}\nonumber\\
 & +\tfrac{\mu}{2(1-\mu/L)}\|\gamma d_{k-1}-\tfrac{1}{L}(g_{k-1}-g_{k})\|^{2}\bigg]\nonumber\\
= & \frac{\beta^{2}(q+1)^{2}}{\gamma^{2}L(1-q)q}\bigg[f_{k}-f_{k-1}+\gamma\langle g_{k};\,d_{k-1}\rangle+\tfrac{1}{2L}(\|g_{k-1}\|^{2}-2\left\langle g_{k};\,g_{k-1}\right\rangle +\|g_{k}\|^{2})\nonumber\\
 & +\tfrac{\mu}{2(1-\mu/L)}\Big(\gamma^{2}\|d_{k-1}\|^{2}+\frac{1}{L^{2}}\|g_{k-1}\|^{2}+\frac{1}{L^{2}}\|g_{k}\|^{2}-\frac{2}{L^{2}}\left\langle g_{k};\,g_{k-1}\right\rangle \nonumber\\
 & +\frac{2\gamma}{L}\left\langle d_{k-1};\,g_{k}\right\rangle -\frac{2\gamma}{L}\left\langle d_{k-1};\,g_{k-1}\right\rangle \Big)\bigg]\nonumber\\
= & \frac{\beta^{2}(\mu+L)^{2}}{\gamma^{2}\mu L(L-\mu)}f_{k}-\frac{\beta^{2}(\mu+L)^{2}}{\gamma^{2}\mu L(L-\mu)}f_{k-1}\nonumber\\
 & +\frac{\beta^{2}(\mu+L)^{2}}{2(L-\mu)^{2}}\|d_{k-1}\|^{2}+\frac{\beta^{2}(\mu+L)^{2}}{2\gamma^{2}\mu L(L-\mu)^{2}}\|g_{k-1}\|^{2}+\frac{\beta^{2}(\mu+L)^{2}}{2\gamma^{2}\mu L(L-\mu)^{2}}\|g_{k}\|^{2}\nonumber\\
 & -\frac{\beta^{2}(\mu+L)^{2}}{\gamma^{2}\mu L(L-\mu)^{2}}\left\langle g_{k-1};\,g_{k}\right\rangle +\frac{\beta^{2}(\mu+L)^{2}}{\gamma\mu(L-\mu)^{2}}\left\langle g_{k};\,d_{k-1}\right\rangle -\frac{\beta^{2}(\mu+L)^{2}}{\gamma L(L-\mu)^{2}}\left\langle g_{k-1};\,d_{k-1}\right\rangle ,\label{eq:prp_summand_2_expanded}
\end{align}

 where on the second line we expand the squares and on the third line
we collect the terms.

\paragraph{Expand the third summand of (\ref{eq:prp_big_term_1}).}

Next, we expand the third summand of (\ref{eq:prp_big_term_1}) as
follows:
\begin{align} & \frac{\beta^{2}(q+1)^{2}}{\gamma^{2}L(1-q)q}\bigg[f_{k-1}-f_{k}-\gamma\langle g_{k-1};\,d_{k-1}\rangle+\tfrac{1}{2L}\|g_{k-1}-g_{k}\|^{2}\nonumber\\
 & +\tfrac{\mu}{2(1-\mu/L)}\|\gamma d_{k-1}-\tfrac{1}{L}(g_{k-1}-g_{k})\|^{2}\bigg]\nonumber\\
= & \frac{\beta^{2}(q+1)^{2}}{\gamma^{2}L(1-q)q}\bigg[f_{k-1}-f_{k}-\gamma\langle g_{k-1};\,d_{k-1}\rangle+\tfrac{1}{2L}(\|g_{k-1}\|^{2}-2\left\langle g_{k};\,g_{k-1}\right\rangle +\|g_{k}\|^{2})\nonumber\\
 & +\tfrac{\mu}{2(1-\mu/L)}\Big(\gamma^{2}\|d_{k-1}\|^{2}+\frac{1}{L^{2}}\|g_{k-1}\|^{2}+\frac{1}{L^{2}}\|g_{k}\|^{2}-\frac{2}{L^{2}}\left\langle g_{k};\,g_{k-1}\right\rangle \nonumber\\
 & +\frac{2\gamma}{L}\left\langle d_{k-1};\,g_{k}\right\rangle -\frac{2\gamma}{L}\left\langle d_{k-1};\,g_{k-1}\right\rangle \Big)\bigg]\nonumber\\
= & \frac{\beta^{2}(\mu+L)^{2}}{\gamma^{2}\mu L(L-\mu)}f_{k-1}-\frac{\beta^{2}(\mu+L)^{2}}{\gamma^{2}\mu L(L-\mu)}f_{k}\nonumber\\
 & +\frac{\beta^{2}(\mu+L)^{2}}{2(L-\mu)^{2}}\|d_{k-1}\|^{2}+\frac{\beta^{2}(\mu+L)^{2}}{2\gamma^{2}\mu L(L-\mu)^{2}}\|g_{k-1}\|^{2}+\frac{\beta^{2}(\mu+L)^{2}}{2\gamma^{2}\mu L(L-\mu)^{2}}\|g_{k}\|^{2}\nonumber\\
 & -\frac{\beta^{2}(\mu+L)^{2}}{\gamma^{2}\mu L(L-\mu)^{2}}\left\langle g_{k-1};\,g_{k}\right\rangle +\frac{\beta^{2}(\mu+L)^{2}}{\gamma L(L-\mu)^{2}}\left\langle g_{k};\,d_{k-1}\right\rangle -\frac{\beta^{2}(\mu+L)^{2}}{\gamma\mu(L-\mu)^{2}}\left\langle g_{k-1};\,d_{k-1}\right\rangle \label{eq:prp_summand_3_expanded}
\end{align}

where again on the second line, we expand the squares and on the third
line, we collect the terms.

\paragraph{Expanded form of (\ref{eq:prp_big_term_1}).}

Now putting (\ref{eq:prp_summand_2_expanded}) and (\ref{eq:prp_summand_3_expanded})
in (\ref{eq:prp_big_term_1}), and then collecting the terms, we arrive
at the following expanded form of (\ref{eq:prp_big_term_1}):

\begin{align}
 & \frac{\beta^{2}(\mu+L)^{2}}{\gamma^{2}\mu L(L-\mu)^{2}}\|g_{k-1}\|^{2}+\frac{\beta^{2}(\mu+L)^{2}}{(L-\mu)^{2}}\|d_{k-1}\|^{2}+\frac{\beta(\mu+L)\left(\mu(\beta-\gamma\mu)-\gamma L^{2}+L(\beta+2\gamma\mu)\right)}{\gamma^{2}\mu L(L-\mu)^{2}}\|g_{k}\|^{2}\nonumber \\
 & +\frac{\beta(\mu+L)\left(\mu(\gamma\mu-2\beta)+\gamma L^{2}-2L(\beta+\gamma\mu)\right)}{\gamma^{2}\mu L(L-\mu)^{2}}\left\langle g_{k-1};\,g_{k}\right\rangle +\frac{4\beta^{2}(\mu+L)}{\gamma(L-\mu)^{2}}\left\langle g_{k};\,d_{k-1}\right\rangle  \nonumber\\
 & -\frac{4\beta^{2}(\mu+L)}{\gamma(L-\mu)^{2}}\left\langle g_{k-1};\,d_{k-1}\right\rangle. \label{eq:simplifed_form_summands}
\end{align}

\paragraph{Expand the first two summands of (\ref{eq:prp_simp_term}).}

Now, we expand the first two summands of (\ref{eq:prp_simp_term})
as follows: 
\begin{align}
 & \|d_{k}\|^{2}-\frac{(1+(\mu/L))^{2}}{4(\mu/L)}\|g_{k}\|^{2}\nonumber \\
= & \|\beta d_{k-1}+g_{k}\|^{2}-\frac{(1+(\mu/L))^{2}}{4(\mu/L)}\|g_{k}\|^{2}\nonumber \\
= & \beta^{2}\|d_{k-1}\|^{2}+\|g_{k}\|^{2}+2\beta\left\langle d_{k-1};\,g_{k}\right\rangle -\frac{(1+(\mu/L))^{2}}{4(\mu/L)}\|g_{k}\|^{2}\nonumber \\
= & \beta^{2}\|d_{k-1}\|^{2}-\frac{(L-\mu)^{2}}{4\mu L}\|g_{k}\|^{2}+2\beta\left\langle d_{k-1};\,g_{k}\right\rangle .\label{eq:prp_simplified_second_summand}
\end{align}

\paragraph{Expand the third summand of (\ref{eq:prp_simp_term}).}

Next, we expand the third summand of (\ref{eq:prp_simp_term}) as
follows:
\begin{align}
 & \frac{4\beta^{2}q}{(1-q)^{2}}\left\Vert d_{k-1}-\tfrac{1+q}{2L\gamma_{k-1}q}g_{k-1}+\tfrac{2\beta(1+q)-L\gamma(1-q)^{2}}{4\beta L\gamma q}g_{k}\right\Vert ^{2}\nonumber \\
= & \frac{4\beta^{2}(\mu/L)}{(1-(\mu/L))^{2}}\Big[\|d_{k-1}\|^{2}+\frac{\left(\mu(\gamma\mu-2\beta)+\gamma L^{2}-2L(\beta+\gamma\mu)\right)^{2}}{16\beta^{2}\gamma^{2}\mu^{2}L^{2}}\|g_{k}\|^{2}+\frac{(\mu+L)^{2}}{4\gamma^{2}\mu^{2}L^{2}}\|g_{k-1}\|^{2}\nonumber \\
 & +\frac{\left(\mu(2\beta-\gamma\mu)-\gamma L^{2}+2L(\beta+\gamma\mu)\right)}{2\beta\gamma\mu L}\left\langle g_{k};\,d_{k-1}\right\rangle -\frac{(\mu+L)}{\gamma\mu L}\left\langle g_{k-1};\,d_{k-1}\right\rangle \nonumber \\
 & +\frac{(\mu+L)\left(\mu(\gamma\mu-2\beta)+\gamma L^{2}-2L(\beta+\gamma\mu)\right)}{4\beta\gamma^{2}\mu^{2}L^{2}}\left\langle g_{k-1};\,g_{k}\right\rangle \Big]\nonumber \\
= & \frac{4\beta^{2}\mu L}{(L-\mu)^{2}}\|d_{k-1}\|^{2}+\frac{\left(\mu(\gamma\mu-2\beta)+\gamma L^{2}-2L(\beta+\gamma\mu)\right)^{2}}{4\gamma^{2}\mu L(L-\mu)^{2}}\|g_{k}\|^{2}+\frac{\beta^{2}(\mu+L)^{2}}{\gamma^{2}\mu L(L-\mu)^{2}}\|g_{k-1}\|^{2}\nonumber \\
 & +2\beta\left(\frac{2\beta(\mu+L)}{\gamma(L-\mu)^{2}}-1\right)\left\langle g_{k};\,d_{k-1}\right\rangle -\frac{4\beta^{2}(\mu+L)}{\gamma(L-\mu)^{2}}\left\langle g_{k-1};\,d_{k-1}\right\rangle \nonumber \\
 & +\frac{\beta(\mu+L)\left(\mu(\gamma\mu-2\beta)+\gamma L^{2}-2L(\beta+\gamma\mu)\right)}{\gamma^{2}\mu L(L-\mu)^{2}}\left\langle g_{k-1};\,g_{k}\right\rangle .\label{eq:prp_simplified_third_summand}
\end{align}

\paragraph{Expanded form of (\ref{eq:prp_simp_term}).}

Finally, putting the expanded expressions from (\ref{eq:prp_simplified_second_summand})
(\ref{eq:prp_simplified_third_summand}) in (\ref{eq:prp_simp_term})
and then collecting the terms, we get:
\begin{align*}
 & \beta^{2}\|d_{k-1}\|^{2}-\frac{(L-\mu)^{2}}{4\mu L}\|g_{k}\|^{2}+2\beta\left\langle d_{k-1};\,g_{k}\right\rangle \\
 & +\frac{4\beta^{2}\mu L}{(L-\mu)^{2}}\|d_{k-1}\|^{2}+\frac{\left(\mu(\gamma\mu-2\beta)+\gamma L^{2}-2L(\beta+\gamma\mu)\right)^{2}}{4\gamma^{2}\mu L(L-\mu)^{2}}\|g_{k}\|^{2}+\frac{\beta^{2}(\mu+L)^{2}}{\gamma^{2}\mu L(L-\mu)^{2}}\|g_{k-1}\|^{2}\\
 & +2\beta\left(\frac{2\beta(\mu+L)}{\gamma(L-\mu)^{2}}-1\right)\left\langle g_{k};\,d_{k-1}\right\rangle -\frac{4\beta^{2}(\mu+L)}{\gamma(L-\mu)^{2}}\left\langle g_{k-1};\,d_{k-1}\right\rangle \\
 & +\frac{\beta(\mu+L)\left(\mu(\gamma\mu-2\beta)+\gamma L^{2}-2L(\beta+\gamma\mu)\right)}{\gamma^{2}\mu L(L-\mu)^{2}}\left\langle g_{k-1};\,g_{k}\right\rangle \\
= & \frac{\beta^{2}(\mu+L)^{2}}{\gamma^{2}\mu L(L-\mu)^{2}}\|g_{k-1}\|^{2}+\frac{\beta^{2}(\mu+L)^{2}}{(L-\mu)^{2}}\|d_{k-1}\|^{2}+\frac{\beta(\mu+L)\left(\mu(\beta-\gamma\mu)-\gamma L^{2}+L(\beta+2\gamma\mu)\right)}{\gamma^{2}\mu L(L-\mu)^{2}}\|g_{k}\|^{2}\\
 & +\frac{\beta(\mu+L)\left(\mu(\gamma\mu-2\beta)+\gamma L^{2}-2L(\beta+\gamma\mu)\right)}{\gamma^{2}\mu L(L-\mu)^{2}}\left\langle g_{k-1};\,g_{k}\right\rangle +\frac{4\beta^{2}(\mu+L)}{\gamma(L-\mu)^{2}}\left\langle g_{k};\,d_{k-1}\right\rangle \\
 & -\frac{4\beta^{2}(\mu+L)}{\gamma(L-\mu)^{2}}\left\langle g_{k-1};\,d_{k-1}\right\rangle ,
\end{align*}
 where the last line is identical to (\ref{eq:simplifed_form_summands}).

 The calculation shown above can also be independently verified using open-source symbolic computation libraries
 \texttt{SymPy} \cite{sympy} and \texttt{Wolfram Language} \cite{wolframLanguage} using the following notebooks available at

\url{https://github.com/Shuvomoy/NCG-PEP-code/blob/main/Symbolic_Verifications/Verify_PRP.ipynb}

and 

\url{https://github.com/Shuvomoy/NCG-PEP-code/blob/main/Symbolic_Verifications/Verify_PRP_Wolfram_Language.ipynb},

respectively.

\subsection{Reformulation for weighted sum of inequalities for Lemma \ref{Bound-on-beta}}\label{app:reform_lemma_2_2}

For notational ease, define $f(x_{k})\triangleq f_{k}$, $f(x_{k-1})\triangleq f_{k-1}$,$\nabla f({x_{k}})\triangleq g_{k}$,
$\nabla f(x_{k-1})\triangleq g_{k-1}$, $\beta_{k-1}\triangleq\beta$,
$\gamma_{k-1}\triangleq\gamma$, and $c_{k-1}\triangleq c$. We want
to show that the weighted sum

\begin{align}
 & \left(\gamma(L+\mu)-\frac{2\sqrt{\beta}}{\sqrt{(c-1)c}}\right)\left[\left\langle g_{k-1};\,d_{k-1}\right\rangle -\|g_{k-1}\|^{2}\right]\nonumber \\
 & +\left(\frac{2}{c}-\gamma(L+\mu)\right)\left[\left\langle g_{k};\,d_{k-1}\right\rangle \right]\nonumber \\
 & +\left(\frac{\sqrt{c-1}}{\sqrt{\beta c}}\right)\left[\|g_{k}\|^{2}-\beta\|g_{k-1}\|^{2}\right]\nonumber \\
 & +\left(-\gamma^{2}L\mu+\frac{\sqrt{\beta}}{c\sqrt{(c-1)c}}\right)\left[\|d_{k-1}\|^{2}-c\|g_{k-1}\|^{2}\right]\nonumber \\
 & +\left(L-\mu\right)\Big[f_{k}-f_{k-1}+\gamma\langle g_{k};\,d_{k-1}\rangle+\tfrac{1}{2L}\|g_{k-1}-g_{k}\|^{2}\nonumber \\
 & \quad\quad+\tfrac{\mu}{2(1-\mu/L)}\|\gamma d_{k-1}-\tfrac{1}{L}(g_{k-1}-g_{k})\|^{2}\Big]\nonumber \\
 & +\left(L-\mu\right)\Big[f_{k-1}-f_{k}-\gamma\langle g_{k-1};\,d_{k-1}\rangle+\tfrac{1}{2L}\|g_{k-1}-g_{k}\|^{2}\nonumber \\
 & \quad\quad+\tfrac{\mu}{2(1-\mu/L)}\|\gamma d_{k-1}-\tfrac{1}{L}(g_{k-1}-g_{k})\|^{2}\Big]\label{eq:FR_big_term_1}
\end{align}

is equal to 

\begin{align} & \|g_{k}\|^{2}-\left(2\sqrt{1-\frac{1}{c}}\sqrt{\beta}-c\gamma^{2}L\mu+\gamma(L+\mu)-1\right)\|g_{k-1}\|^{2} \nonumber \\
 & +\left\Vert \sqrt[4]{\frac{\beta}{(c-1)c^{3}}}d_{k-1}-\sqrt[4]{\frac{\beta c}{c-1}}g_{k-1}+\sqrt[4]{\frac{c-1}{\beta c}}g_{k}\right\Vert ^{2} \label{eq:FR_1_simp_term}
\end{align}

We show this by expanding both terms and doing term-by-term matching. 

\paragraph{Expand the fifth summand of (\ref{eq:FR_big_term_1}).}

First, we expand the second summand of (\ref{eq:FR_big_term_1}) as
follows: 
\begin{align} & \left(L-\mu\right)\Big[f_{k}-f_{k-1}+\gamma\langle g_{k};\,d_{k-1}\rangle+\tfrac{1}{2L}\|g_{k-1}-g_{k}\|^{2}+\tfrac{\mu}{2(1-\mu/L)}\|\gamma d_{k-1}-\tfrac{1}{L}(g_{k-1}-g_{k})\|^{2}\Big]\nonumber\\
= & \left(L-\mu\right)\Big[f_{k}-f_{k-1}+\gamma\langle g_{k};\,d_{k-1}\rangle+\tfrac{1}{2L}\left(\|g_{k-1}\|^{2}-2\left\langle g_{k};\,g_{k-1}\right\rangle +\|g_{k}\|^{2}\right)\nonumber\\
 & +\tfrac{\mu}{2(1-\mu/L)}\Big(\gamma^{2}\|d_{k-1}\|^{2}+\frac{1}{L^{2}}\|g_{k-1}\|^{2}+\frac{1}{L^{2}}\|g_{k}\|^{2}-\frac{2}{L^{2}}\left\langle g_{k};\,g_{k-1}\right\rangle \nonumber\\
 & +\frac{2\gamma}{L}\left\langle d_{k-1};\,g_{k}\right\rangle -\frac{2\gamma}{L}\left\langle d_{k-1};\,g_{k-1}\right\rangle \Big)\Big]\nonumber\\
= & (L-\mu)f_{k}+(\mu-L)f_{k-1}+\frac{1}{2}\gamma^{2}\mu L\|d_{k-1}\|^{2}+\frac{1}{2}\|g_{k-1}\|^{2}+\frac{1}{2}\|g_{k}\|^{2}\nonumber\\
 & -\left\langle g_{k-1};\,g_{k}\right\rangle +\gamma L\left\langle g_{k};\,d_{k-1}\right\rangle -\gamma\mu\left\langle g_{k-1};\,d_{k-1}\right\rangle \label{eq:summand_5_fr_1}
\end{align}
where on the second line, we expand the squares, and on the third line, we collect the terms.

\paragraph{Expand the sixth summand of (\ref{eq:FR_big_term_1}).}

Next, we expand the sixth summand of (\ref{eq:FR_big_term_1}) as
follows: 
\begin{align} & \left(L-\mu\right)\Big[f_{k-1}-f_{k}-\gamma\langle g_{k-1};\,d_{k-1}\rangle+\tfrac{1}{2L}\|g_{k-1}-g_{k}\|^{2}+\tfrac{\mu}{2(1-\mu/L)}\|\gamma d_{k-1}-\tfrac{1}{L}(g_{k-1}-g_{k})\|^{2}\Big]\nonumber\\
= & \left(L-\mu\right)\Big[f_{k-1}-f_{k}-\gamma\langle g_{k-1};\,d_{k-1}\rangle+\tfrac{1}{2L}\left(\|g_{k-1}\|^{2}-2\left\langle g_{k};\,g_{k-1}\right\rangle +\|g_{k}\|^{2}\right)\nonumber\\
 & +\tfrac{\mu}{2(1-\mu/L)}\Big(\gamma^{2}\|d_{k-1}\|^{2}+\frac{1}{L^{2}}\|g_{k-1}\|^{2}+\frac{1}{L^{2}}\|g_{k}\|^{2}-\frac{2}{L^{2}}\left\langle g_{k};\,g_{k-1}\right\rangle \nonumber\\
 & +\frac{2\gamma}{L}\left\langle d_{k-1};\,g_{k}\right\rangle -\frac{2\gamma}{L}\left\langle d_{k-1};\,g_{k-1}\right\rangle \Big)\Big]\nonumber\\
= & (L-\mu)f_{k-1}+(\mu-L)f_{k}+\frac{1}{2}\gamma^{2}\mu L\|d_{k-1}\|^{2}+\frac{1}{2}\|g_{k-1}\|^{2}+\frac{1}{2}\|g_{k}\|^{2}\nonumber\\
 & -\left\langle g_{k-1};\,g_{k}\right\rangle +\gamma\mu\left\langle g_{k};\,d_{k-1}\right\rangle -\gamma L\left\langle g_{k-1};\,d_{k-1}\right\rangle \label{eq:summand_6_fr_1}
\end{align}

where again on the second line, we expand the squares and on the third
line, we collect the terms.

\paragraph{Expanded form of (\ref{eq:FR_big_term_1}).}

Now putting (\ref{eq:summand_5_fr_1}) and (\ref{eq:summand_6_fr_1})
in (\ref{eq:FR_big_term_1}), and then collecting the terms, we arrive
at the following expanded form of (\ref{eq:FR_big_term_1}):

\begin{align}
 & \frac{\sqrt{\beta}}{\sqrt{c-1}c^{3/2}}\|d_{k-1}\|^{2}+\left(\frac{\sqrt{c-1}}{\sqrt{c}\sqrt{\beta}}+1\right)\|g_{k}\|^{2}+\left(-\frac{\sqrt{\beta}(c-2)}{\sqrt{c-1}\sqrt{c}}+c\gamma^{2}\mu L-\gamma(\mu+L)+1\right)\|g_{k-1}\|^{2}\nonumber \\
 & \quad-2\left\langle g_{k-1};\,g_{k}\right\rangle +\frac{2}{c}\left\langle g_{k};\,d_{k-1}\right\rangle -\frac{2\sqrt{\beta}}{\sqrt{c-1}\sqrt{c}}\left\langle g_{k-1};\,d_{k-1}\right\rangle \label{eq:simplified_form_weighted_form_FR_part_1}
\end{align}

\paragraph{Expand the third summand of (\ref{eq:FR_1_simp_term}).}

Next, we expand the third summand of (\ref{eq:FR_1_simp_term}) as
follows: 

\begin{align}
 & \left\Vert \sqrt[4]{\frac{\beta}{(c-1)c}}d_{k-1}-\sqrt[4]{\frac{\beta c}{c-1}}g_{k-1}+\sqrt[4]{\frac{c-1}{\beta c}}g_{k}\right\Vert ^{2}\nonumber \\
= & \frac{\sqrt{\beta}}{\sqrt{c-1}c^{3/2}}\|d_{k-1}\|^{2}+\frac{\sqrt{\beta}\sqrt{c}}{\sqrt{c-1}}\|g_{k-1}\|^{2}+\frac{\sqrt{c-1}}{\sqrt{c}\sqrt{\beta}}\|g_{k}\|^{2}\nonumber \\
 & -2\left\langle g_{k-1};\,g_{k}\right\rangle +\frac{2}{c}\left\langle g_{k};\,d_{k-1}\right\rangle -2\frac{\sqrt{\beta}}{\sqrt{c-1}\sqrt{c}}\left\langle g_{k-1};\,d_{k-1}\right\rangle \label{eq:third_summand_expanded_FR_1}
\end{align}

\paragraph{Expanded form of (\ref{eq:FR_1_simp_term}).}

Finally, putting the expanded expressions from (\ref{eq:simplified_form_weighted_form_FR_part_1})
in (\ref{eq:FR_1_simp_term}) and then collecting the terms, we get:
\begin{align*}
 & \|g_{k}\|^{2}-\left(2\sqrt{1-\frac{1}{c}}\sqrt{\beta}-c\gamma^{2}L\mu+\gamma(L+\mu)-1\right)\|g_{k-1}\|^{2}\\
 & +\left\Vert \sqrt[4]{\frac{\beta}{(c-1)c}}d_{k-1}-\sqrt[4]{\frac{\beta c}{c-1}}g_{k-1}+\sqrt[4]{\frac{c-1}{\beta c}}g_{k}\right\Vert ^{2}\\
= & \|g_{k}\|^{2}-\left(2\sqrt{1-\frac{1}{c}}\sqrt{\beta}-c\gamma^{2}L\mu+\gamma(L+\mu)-1\right)\|g_{k-1}\|^{2}\\
 & +\frac{\sqrt{\beta}}{\sqrt{c-1}c^{3/2}}\|d_{k-1}\|^{2}+\frac{\sqrt{\beta}\sqrt{c}}{\sqrt{c-1}}\|g_{k-1}\|^{2}+\frac{\sqrt{c-1}}{\sqrt{c}\sqrt{\beta}}\|g_{k}\|^{2}\\
 & -2\left\langle g_{k-1};\,g_{k}\right\rangle +\frac{2}{c}\left\langle g_{k};\,d_{k-1}\right\rangle -2\frac{\sqrt{\beta}}{\sqrt{c-1}\sqrt{c}}\left\langle g_{k-1};\,d_{k-1}\right\rangle \\
= & \left(\frac{\sqrt{c-1}}{\sqrt{c}\sqrt{\beta}}+1\right)\|g_{k}\|^{2}+\left(c\gamma^{2}L\mu-\gamma(L+\mu)+1+\sqrt{\beta}\frac{\sqrt{c}}{\sqrt{c-1}}-2\sqrt{\beta}\frac{\sqrt{c-1}}{\sqrt{c}}\right)\|g_{k-1}\|^{2}\\
 & +\frac{\sqrt{\beta}}{\sqrt{c-1}c^{3/2}}\|d_{k-1}\|^{2}-2\left\langle g_{k-1};\,g_{k}\right\rangle +\frac{2}{c}\left\langle g_{k};\,d_{k-1}\right\rangle -2\frac{\sqrt{\beta}}{\sqrt{c-1}\sqrt{c}}\left\langle g_{k-1};\,d_{k-1}\right\rangle \\
= & \frac{\sqrt{\beta}}{\sqrt{c-1}c^{3/2}}\|d_{k-1}\|^{2}+\left(\frac{\sqrt{c-1}}{\sqrt{c}\sqrt{\beta}}+1\right)\|g_{k}\|^{2}+\left(-\frac{\sqrt{\beta}(c-2)}{\sqrt{c-1}\sqrt{c}}+c\gamma^{2}\mu L-\gamma(\mu+L)+1\right)\|g_{k-1}\|^{2}\\
 & \quad-2\left\langle g_{k-1};\,g_{k}\right\rangle +\frac{2}{c}\left\langle g_{k};\,d_{k-1}\right\rangle -\frac{2\sqrt{\beta}}{\sqrt{c-1}\sqrt{c}}\left\langle g_{k-1};\,d_{k-1}\right\rangle 
\end{align*}
where the last line is identical to (\ref{eq:simplified_form_weighted_form_FR_part_1}).

This symbolical calculation shown above can be independently verified using open-source symbolic computation libraries
 \texttt{SymPy} \cite{sympy} and \texttt{Wolfram Language} \cite{wolframLanguage} using the following notebooks available at

{\small \url{https://github.com/Shuvomoy/NCG-PEP-code/blob/main/Symbolic_Verifications/Verify_FR.ipynb}}

and 

{\small \url{https://github.com/Shuvomoy/NCG-PEP-code/blob/main/Symbolic_Verifications/Verify_FR.ipynb}}

(in the cells titled \texttt{Lemma 2.2} of the notebooks), respectively.

\subsection{Reformulation for weighted sum of inequalities for Lemma \ref{thm:FR_1}}\label{app:reform_lemma_2_3}
For notational ease, define $\nabla f({x_k}) \triangleq g_k$, $\nabla f(x_{k-1}) \triangleq g_{k-1}$, $\beta_{k-1} \triangleq \beta$, and $c_{k-1} \triangleq c$. The reformulation of the weighted sum is as follows: 
\begin{align*} & 2\beta\left\langle d_{k-1};\,g_{k}\right\rangle +\beta^{2}\left(\|d_{k-1}\|^{2}-c\|g_{k-1}\|^{2}\right)-c\beta\left(\|g_{k}\|^{2}-\beta\|g_{k-1}\|^{2}\right)\\
= & 2\beta\left\langle d_{k-1};\,g_{k}\right\rangle +\beta^{2}\|d_{k-1}\|^{2}-\cancel{c\beta^{2}\|g_{k-1}\|^{2}}-c\beta\|g_{k}\|^{2}+\cancel{c\beta^{2}\|g_{k-1}\|^{2}}\\
= & 2\beta\left\langle d_{k-1};\,g_{k}\right\rangle +\beta^{2}\|d_{k-1}\|^{2}-c\beta\|g_{k}\|^{2}\\
= & 2\left\langle \beta d_{k-1};\,g_{k}\right\rangle +\|\beta d_{k-1}\|^{2}-c\beta\|g_{k}\|^{2}\\
= & 2\left\langle d_{k}-g_{k};\,g_{k}\right\rangle +\|d_{k}-g_{k}\|^{2}-c\beta\|g_{k}\|^{2}\\
= & \cancel{2\left\langle d_{k};\,g_{k}\right\rangle }-2\|g_{k}\|^{2}+\|d_{k}\|^{2}-\cancel{2\left\langle d_{k};\,g_{k}\right\rangle }+\|g_{k}\|^{2}-c\beta\|g_{k}\|^{2}\\
= & \|d_{k}\|^{2}-\|g_{k}\|^{2}-c\beta\|g_{k}\|^{2}\\
= & \|d_{k}\|^{2}-\left(1+c\beta\right)\|g_{k}\|^{2},
\end{align*}
thus arriving at the simplified form used in the proof. 

This symbolical calculation shown above can be independently verified using open-source symbolic computation libraries
 \texttt{SymPy} \cite{sympy} and \texttt{Wolfram Language} \cite{wolframLanguage} using the following notebooks available at

{\small \url{https://github.com/Shuvomoy/NCG-PEP-code/blob/main/Symbolic_Verifications/Verify_FR.ipynb}}

and 

{\small \url{https://github.com/Shuvomoy/NCG-PEP-code/blob/main/Symbolic_Verifications/Verify_FR.ipynb}}

(in the cells titled \texttt{Lemma 2.3} of the notebooks), respectively.

}

{

\section{Constructing counter-examples}\label{subsec:counter-example}

Once we have solved the nonconvex QCQPs associated with \eqref{eq:PEP_rho_N} or \eqref{eq:PEP_rho_N0}, we can
construct the associated triplets $\{x_{i},g_{i},f_{i}\}_{i\in I_{N}^{\star}}$
and then apply Theorem \ref{thm:scvx-extension} to construct the corresponding ``bad''
function. This ``bad'' function serves as a counter-example, illustrating scenarios where \eqref{eq:NCG} performs poorly. One can access the numerically constructed triplets $\{x_{i},g_{i},f_{i}\}_{i\in I_{N}^{\star}}$ associated with the counter-examples by following the instructions provided in our github repository. Next, we provide a concrete example of how to construct a ``bad'' function for \eqref{eq:PEP_rho_N} from our provided code and datasets located in the folder titled `\texttt{Code\_for\_NCG\_PEP}' of the github repository. Constructing counter-examples for \eqref{eq:PEP_rho_N0} is analogous.

\begin{table}[!ht]
\centering{}%
\begin{tabular}{cccc}
\specialrule{2pt}{1pt}{1pt} 
$x_\star$ & $x_0$ & $x_1$ & $x_2$ \\
\specialrule{2pt}{1pt}{1pt} 
$\begin{bmatrix} 0 \\ 0 \\ 0 \\ 0 \end{bmatrix}$ & 
$\begin{bmatrix} 1.67262 \\ 0 \\ 0 \\ 0 \end{bmatrix}$ & 
$\begin{bmatrix} 0.354109 \\ -0.810313 \\ 0.0775561 \\ 0.000222477 \end{bmatrix}$ & $\begin{bmatrix} -0.140817 \\ -0.322955 \\ -0.138845 \\ 0.000244795 \end{bmatrix}$ \\
\specialrule{2pt}{1pt}{1pt} 
\end{tabular}
\caption{Numerical values of $\{x_i\}_{i\in\{\star, 0, 1, 2\}}$ for constructing the counter-example of Example \ref{example:counter_example}.}
\label{tab:xiVal}
\end{table}

\begin{table}[!ht]
\centering{}%
\begin{tabular}{cccc}
\specialrule{2pt}{1pt}{1pt} 
$g_\star$ & $g_0$ & $g_1$ & $g_2$ \\
\specialrule{2pt}{1pt}{1pt} 
$\begin{bmatrix}
0 \\ 
0 \\ 
0 \\ 
0
\end{bmatrix}$ & 
$
\begin{bmatrix}
1.08734 \\ 
0.237212 \\ 
0 \\ 
0
\end{bmatrix}
$ &
$
\begin{bmatrix}
0.303362 \\ 
-0.47567 \\ 
0.187527 \\ 
0
\end{bmatrix}
$ &
$
\begin{bmatrix}
-0.158567 \\ 
-0.205519 \\ 
-0.100196 \\ 
0.000166564
\end{bmatrix}
$ \\
\specialrule{2pt}{1pt}{1pt} 
\end{tabular}
\caption{Numerical values of $\{g_i\}_{i\in\{\star, 0, 1, 2\}}$ for constructing the counter-example of Example \ref{example:counter_example}.}
\label{tab:giVal}
\end{table}

\begin{table}[!ht]
\centering{}%
\begin{tabular}{cccc}
\specialrule{2pt}{1pt}{1pt} 
$f_\star$ & $f_0$ & $f_1$ & $f_2$ \\
\specialrule{2pt}{1pt}{1pt} 
0 & 1 & 0.267353 & 0.056104 \\
\specialrule{2pt}{1pt}{1pt} 
\end{tabular}
\caption{Numerical values of $\{f_i\}_{i\in\{\star, 0, 1, 2\}}$ for constructing the counter-example of Example \ref{example:counter_example}.}
\label{tab:fiVal}
\end{table}

\begin{exmp}[How to construct counter-examples for \eqref{eq:PEP_rho_N}]\label{example:counter_example}
Suppose we are interested in constructing a ``bad'' function aka
counter-example for the worst-case bound on $\nicefrac{f(x_{k+2})-f_{\star}}{f(x_{k})-f_{\star}}$ (steps for other values of $N$ {are identical}) for
{\PRP} with $q\triangleq \nicefrac{\mu}{L}=0.5$. The resultant ``bad'' function from
$\mbox{\ensuremath{\mathbb{R}}}^{4}$ to $\mbox{\ensuremath{\mathbb{R}}}$
is completely characterized by the triplets $\{x_{i},g_{i},f_{i}\}_{i\in\{\star,0,1,2\}}$, where the triplets can be generated or accessed in two ways:  
\begin{enumerate}
\item we can run `\texttt{1.Example\_Julia.ipynb}' with the 
input parameters and generate the function by solving the
nonconvex QCQP directly and generate the triplets, or
\item we can directly access the triplets from the saved datasets in the folder
\texttt{Saved\_Output\_Files} with instructions provided in the file `\texttt{2.Using\_the\_saved\_datasets\_Julia.ipynb}'.
\end{enumerate}
For the sake of completeness, we provide the
numerical values of $\{x_{i}\}_{i\in\{\star,0,1,2\}},\{g_{i}\}_{i\in\{\star,0,1,2\}},$
and $\{f_{i}\}_{i\in\{\star,0,1,2\}}$ of the function in this setup
in Table \ref{tab:xiVal}, Table \ref{tab:giVal}, and Table \ref{tab:fiVal}, respectively. From the numerical values of the triplets, we can construct the ``bad'' function using Theorem \ref{thm:scvx-extension}. For this constructed function, we have the performance guarantee $\nicefrac{f(x_{k+2})-f_{\star}}{f(x_{k})-f_{\star}} \geq 0.056104$, which closely matches the bound provided in Figure \ref{fig:PRP_Lyapunov}. Additionally, this guarantee can be verified through other {existing} open-source software~\cite{goujaud2022pepit,taylor2017performance}; we provide code for this independent verification in the file called `\texttt{3.PEPIt\_verification\_Python.ipynb}'.
\end{exmp}
}
\end{document}

%% file: polyak_vs_ncgpep.tex
\begin{tikzpicture}
\begin{axis}[
    legend style= {at={(1,0.5)},anchor=west}, plotOptions7, 
    xlabel = $q$,
    ylabel = {Upper bound on $\frac{f(x_{k+1}-f_\star)}{f(x_k)-f_\star}$},
    domain=0:1,
    samples=100,
]
\addplot [color=colorP1] {1 - x/(1 + 1/x^2)};
\addlegendentry{Polyak \cite[Theorem 2]{polyak1969conjugate}}
\addplot [color = colorP2, style = {dashed}] {(1 - x^2)^2/(1 + x^2)^2};
\addlegendentry{Theorem \ref{thm:PRP_bound} of this work}
\end{axis}
\end{tikzpicture}

%% file: prp_Lyapunov_master.tex
\begin{tikzpicture}
\begin{axis}[legend style= {at={(1,0.5)},anchor=west}, plotOptions6, 
xlabel={$\cond$}, ylabel={$\sqrt[N]{\rho_N\left(q, \tfrac{(1+\cond)^{2}}{4\cond}\right)}$}]
    \addplot+[color=colorP1]
        coordinates {
            (0.001,0.9960079880159803)
            (0.005,0.9801985099378732)
            (0.01,0.9607881580237231)
            (0.02,0.9231064975009611)
            (0.03,0.8868884909039495)
            (0.04,0.8520710059171597)
            (0.05,0.81859410430839)
            (0.06,0.7864008543965822)
            (0.07,0.7554371560835006)
            (0.08,0.7256515775034293)
            (0.09,0.6969952024240383)
            (0.1,0.6694214876033058)
            (0.15,0.5463137996219282)
            (0.2,0.44444444444444453)
            (0.25,0.36)
            (0.3,0.2899408284023668)
            (0.35,0.23182441700960219)
            (0.4,0.1836734693877551)
            (0.45,0.14387633769322236)
            (0.5,0.1111111111111111)
            (0.55,0.08428720083246616)
            (0.6,0.0625)
            (0.65,0.04499540863177227)
            (0.7,0.03114186851211074)
            (0.75,0.02040816326530612)
            (0.8,0.012345679012345671)
        }
        ;
    \addlegendentry {$\textrm{GDEL}:\nicefrac{f_{k+1}-f_{\star}}{f_{k}-f_{\star}}$}
    \addplot+[color=colorP2, style = {loosely dashed}]
        coordinates {
            (0.001,0.9999959998991834)
            (0.005,0.9999000050171425)
            (0.01,0.9996000801824212)
            (0.02,0.9984012778644896)
            (0.03,0.9964064697450803)
            (0.04,0.9936204310899357)
            (0.05,0.990049813955355)
            (0.06,0.9857031240494374)
            (0.07,0.9805906788781874)
            (0.08,0.9747245611464818)
            (0.09,0.9681185732670824)
            (0.1,0.9607881600874622)
            (0.15,0.913917304145395)
            (0.2,0.8520710097341999)
            (0.25,0.7785467323890536)
            (0.3,0.6969952095943083)
            (0.35,0.6111130527517736)
            (0.4,0.5243757528091876)
            (0.45,0.43983645713241304)
            (0.5,0.36000002294173783)
            (0.55,0.28676951469927775)
            (0.6,0.22145334537569836)
            (0.65,0.1648160654972316)
            (0.7,0.11715690843736894)
            (0.75,0.07840002062682902)
            (0.8,0.04818560580594592)
        }
        ;
       \addlegendentry {\PRP:$N=1$}
    \addplot+[color=colorP3, style = {densely dashed}]
        coordinates {
              (0.001,0.9999195011376952)
            (0.005,0.999009063590895)
            (0.01,0.9970821808146944)
            (0.02,0.9914077318026863)
            (0.03,0.9837334085049082)
            (0.04,0.9743364306142192)
            (0.05,0.9633903107088588)
            (0.06,0.9510268677610879)
            (0.07,0.9373580561845642)
            (0.08,0.9224856867632967)
            (0.09,0.9065063486137265)
            (0.1,0.8895139420378209)
            (0.15,0.7925739261566849)
            (0.2,0.683843728012363)
            (0.25,0.5791562763552603)
            (0.3,0.49291713818701866)
            (0.35,0.41616258698454134)
            (0.4,0.34844750727415114)
            (0.45,0.2889741166413951)
            (0.5,0.2369056068447073)
            (0.55,0.1914729280588297)
            (0.6,0.15200220836267864)
            (0.65,0.11738355922541831)
            (0.7,0.08697618894267078)
            (0.75,0.06089936454550705)
            (0.8,0.03927916278418407)
        }
        ;
        \addlegendentry {\PRP:$N=2$}
    \addplot+[color=colorP4, style = {dashed}]
        coordinates {
            (0.001,0.9996576163911381)
            (0.005,0.9966584745036428)
            (0.01,0.9913378451551582)
            (0.02,0.9773071319411828)
            (0.03,0.9599565726952046)
            (0.04,0.9400543040842295)
            (0.05,0.9181269773701963)
            (0.06,0.8945982360090287)
            (0.07,0.8698349054519716)
            (0.08,0.8441635743759579)
            (0.09,0.8133076180237087)
            (0.1,0.7900361818146422)
            (0.15,0.6812904259324499)
            (0.2,0.5857960743891981)
            (0.25,0.5028318246679637)
            (0.3,0.4306081108345249)
            (0.35,0.36737544967616304)
            (0.4,0.3117044869280307)
            (0.45,0.2624309350658199)
            (0.5,0.2184865116506535)
            (0.55,0.17863158172931579)
            (0.6,0.14260859322420588)
            (0.65,0.1104305989661618)
            (0.7,0.08213962124255439)
            (0.75,0.05780527604128296)
            (0.8,0.03752310229295579)
        }
        ;
       \addlegendentry {\PRP:$N=3$}
    \addplot+[color=colorP5, style = {dotted}]
        coordinates {
            (0.001,0.9991881239672794)
            (0.005,0.9940522537494056)
            (0.01,0.9861381934147073)
            (0.02,0.9653685506088278)
            (0.03,0.9415851451756735)
            (0.04,0.915429081370345)
            (0.05,0.8876425149330757)
            (0.06,0.8588401500550329)
            (0.07,0.8188156265201514)
            (0.08,0.7933363698427217)
            (0.09,0.7689563243945814)
            (0.1,0.7454857360180419)
            (0.15,0.6397026534849503)
            (0.2,0.5503895769781993)
            (0.25,0.47423389575307046)
            (0.3,0.40828233198618513)
            (0.35,0.35021372269075984)
            (0.4,0.29838457541638946)
            (0.45,0.25181389893431383)
            (0.5,0.2098901669198688)
            (0.55,0.17187704360005873)
            (0.6,0.1374923565266467)
            (0.65,0.1067286283145876)
            (0.7,0.07961736539034035)
            (0.75,0.05621719502524717)
            (0.8,0.03659892129037445)
        }
             ;
    \addlegendentry {\PRP:$N=4$}
    \addplot[no marks, style={{ultra thick}}, color={colorP6}]
        coordinates {
            (0.001,0.9377544467966324)
            (0.009070707070707073,0.8185901876393797)
            (0.017141414141414145,0.7552909671957262)
            (0.025212121212121217,0.7076456166048424)
            (0.03328282828282829,0.6684111890296339)
            (0.04135353535353536,0.634642161755327)
            (0.04942424242424243,0.6047929685613148)
            (0.0574949494949495,0.5779328596882196)
            (0.06556565656565658,0.5534498237088445)
            (0.07363636363636365,0.5309159433963891)
            (0.08170707070707071,0.5100180968512836)
            (0.08977777777777778,0.49051897633126484)
            (0.09784848484848487,0.4722336357648498)
            (0.10591919191919193,0.45501461544367827)
            (0.113989898989899,0.4387420956172438)
            (0.12206060606060606,0.4233171450948348)
            (0.13013131313131315,0.4086569528625282)
            (0.1382020202020202,0.3946913741396516)
            (0.1462727272727273,0.38136037341059414)
            (0.15434343434343434,0.36861209516079907)
            (0.16241414141414143,0.35640138366364177)
            (0.1704848484848485,0.3446886303175612)
            (0.17855555555555555,0.3334388640731176)
            (0.18662626262626264,0.3226210250813445)
            (0.19469696969696973,0.31220737837711793)
            (0.2027676767676768,0.30217303594934114)
            (0.21083838383838385,0.2924955636706998)
            (0.2189090909090909,0.2831546553671783)
            (0.226979797979798,0.27413186052115396)
            (0.23505050505050507,0.26541035520004624)
            (0.24312121212121213,0.2569747481084629)
            (0.2511919191919192,0.24881091539761505)
            (0.2592626262626263,0.24090585918612417)
            (0.2673333333333333,0.2332475857604673)
            (0.2754040404040404,0.22582500020927662)
            (0.2834747474747475,0.2186278148600219)
            (0.2915454545454546,0.2116464693705363)
            (0.29961616161616167,0.20487206071191538)
            (0.3076868686868687,0.19829628158624132)
            (0.3157575757575758,0.1919113660694783)
            (0.32382828282828285,0.1857100414697097)
            (0.3318989898989899,0.17968548555356706)
            (0.339969696969697,0.17383128842687398)
            (0.34804040404040404,0.16814141846510652)
            (0.3561111111111111,0.16261019177990066)
            (0.3641818181818182,0.15723224478313103)
            (0.3722525252525253,0.15200250947293054)
            (0.38032323232323234,0.14691619111869578)
            (0.38839393939393946,0.14196874806645438)
            (0.3964646464646465,0.13715587342342522)
            (0.4045353535353536,0.13247347841236476)
            (0.4126060606060606,0.12791767721332425)
            (0.4206767676767677,0.12348477313352707)
            (0.42874747474747477,0.11917124596585127)
            (0.43681818181818177,0.11497374041339718)
            (0.44488888888888894,0.11088905547227422)
            (0.452959595959596,0.10691413467740586)
            (0.4610303030303031,0.103046057127129)
            (0.46910101010101013,0.09928202921191448)
            (0.47717171717171714,0.09561937698084748)
            (0.48524242424242425,0.09205553908677613)
            (0.4933131313131313,0.0885880602574015)
            (0.5013838383838384,0.08521458524516302)
            (0.5094545454545455,0.08193285321369483)
            (0.5175252525252526,0.07874069252295314)
            (0.5255959595959596,0.07563601587894837)
            (0.5336666666666666,0.07261681581739926)
            (0.5417373737373737,0.06968116049363524)
            (0.5498080808080807,0.06682718975374116)
            (0.5578787878787879,0.064053111464317)
            (0.565949494949495,0.06135719808034647)
            (0.5740202020202021,0.05873778343256193)
            (0.5820909090909092,0.05619325971738821)
            (0.5901616161616162,0.05372207467406883)
            (0.5982323232323233,0.05132272893494058)
            (0.6063030303030303,0.04899377353605107)
            (0.6143737373737373,0.046733807576416396)
            (0.6224444444444445,0.04454147601521375)
            (0.6305151515151516,0.04241546759710283)
            (0.6385858585858587,0.04035451289668257)
            (0.6466565656565657,0.038357382473827514)
            (0.6547272727272728,0.036422885132315036)
            (0.6627979797979798,0.0345498662747615)
            (0.6708686868686868,0.032737206347435356)
            (0.678939393939394,0.03098381936901864)
            (0.687010101010101,0.029288651537843033)
            (0.6950808080808081,0.02765067991254437)
            (0.7031515151515152,0.026068911161461046)
            (0.7112222222222222,0.024542380376447123)
            (0.7192929292929293,0.023070149947091962)
            (0.7273636363636364,0.021651308491628083)
            (0.7354343434343436,0.020284969841077135)
            (0.7435050505050506,0.01897027207342895)
            (0.7515757575757577,0.017706376594874576)
            (0.7596464646464647,0.01649246726532145)
            (0.7677171717171718,0.015327749565609273)
            (0.7757878787878789,0.01421144980402251)
            (0.7838585858585859,0.013142814359855617)
            (0.791929292929293,0.012121108961938598)
            (0.8,0.011145618000168249)
        }
        ;
    \addlegendentry {$(1 - \sqrt{q})^2$}
\end{axis}
\end{tikzpicture}

%% file: prp_restarted_master.tex
\begin{tikzpicture}[spy using outlines=
	{rectangle, magnification=3, connect spies}]
\begin{axis}[legend style= {at={(1,0.5)},anchor=west}, plotOptions6, 
xlabel={$\cond$}, ylabel={$\sqrt[N]{\rho_{N,0}(q)}$}]
    \addplot+[color=colorP1]
            coordinates {
                        (0.001,0.9960079880159803)
            (0.005,0.9801985099378732)
            (0.01,0.9607881580237231)
            (0.02,0.9231064975009611)
            (0.03,0.8868884909039495)
            (0.04,0.8520710059171597)
            (0.05,0.81859410430839)
            (0.06,0.7864008543965822)
            (0.07,0.7554371560835006)
            (0.08,0.7256515775034293)
            (0.09,0.6969952024240383)
            (0.1,0.6694214876033058)
            (0.15,0.5463137996219282)
            (0.2,0.44444444444444453)
            (0.25,0.36)
            (0.3,0.2899408284023668)
            (0.35,0.23182441700960219)
            (0.4,0.1836734693877551)
            (0.45,0.14387633769322236)
            (0.5,0.1111111111111111)
            (0.55,0.08428720083246616)
            (0.6,0.0625)
            (0.65,0.04499540863177227)
            (0.7,0.03114186851211074)
            (0.75,0.02040816326530612)
            (0.8,0.012345679012345671)
        }
        ;
    \addlegendentry {$\textrm{GDEL}:\frac{f_{1}-f_{\star}}{f_{0}-f_{\star}}$}
    \addplot+[color=colorP2, style = {loosely dashed}]
            coordinates {
                        (0.001,0.9960079880159803)
            (0.005,0.9801985099378732)
            (0.01,0.9607881580237231)
            (0.02,0.9231064975009611)
            (0.03,0.8868884909039495)
            (0.04,0.8520710059171597)
            (0.05,0.81859410430839)
            (0.06,0.7864008543965822)
            (0.07,0.7554371560835006)
            (0.08,0.7256515775034293)
            (0.09,0.6969952024240383)
            (0.1,0.6694214876033058)
            (0.15,0.5463137996219282)
            (0.2,0.44444444444444453)
            (0.25,0.36)
            (0.3,0.2899408284023668)
            (0.35,0.23182441700960219)
            (0.4,0.1836734693877551)
            (0.45,0.14387633769322236)
            (0.5,0.1111111111111111)
            (0.55,0.08428720083246616)
            (0.6,0.0625)
            (0.65,0.04499540863177227)
            (0.7,0.03114186851211074)
            (0.75,0.02040816326530612)
            (0.8,0.012345679012345671)
        }
        ;
    \addlegendentry {\PRP:$N=1$}
    \addplot+[color=colorP3, style = {densely dashed}]
            coordinates {
                        (0.001,0.994793898139542)
            (0.005,0.9745633185501604)
            (0.01,0.9504876787003999)
            (0.02,0.9058017195000724)
            (0.03,0.8649649534011612)
            (0.04,0.8272968402623526)
            (0.05,0.7923025713246351)
            (0.06,0.759609001316087)
            (0.07,0.728926055472931)
            (0.08,0.7000224558601051)
            (0.09,0.6727098524515769)
            (0.1,0.6468320977721065)
            (0.15,0.5349677962318686)
            (0.2,0.4452538539591089)
            (0.25,0.37153234311194194)
            (0.3,0.3099827386467969)
            (0.35,0.25802779287091954)
            (0.4,0.21381339885474862)
            (0.45,0.17551167000245527)
            (0.5,0.14210325662534137)
            (0.55,0.11303946103426879)
            (0.6,0.08789586282612477)
            (0.65,0.06634158838253774)
            (0.7,0.048119309731451135)
            (0.75,0.03302504020777399)
            (0.8,0.0209051670506468)
        }
        ;
    \addlegendentry {\PRP:$N=2$}
    \addplot+[style={dashed}, color=colorP4]
            coordinates {
                        (0.001,0.9930087199187081)
            (0.005,0.9683970900819063)
            (0.01,0.939861997414265)
            (0.02,0.8898584550367228)
            (0.03,0.8466561019798095)
            (0.04,0.8083057667248276)
            (0.05,0.7736389267011183)
            (0.06,0.741897841151529)
            (0.07,0.7125587233475786)
            (0.08,0.6852412092987742)
            (0.09,0.6596582894159022)
            (0.1,0.6355866453456256)
            (0.15,0.5326650666575706)
            (0.2,0.4503984759926091)
            (0.25,0.38197023341956005)
            (0.3,0.3234351642689205)
            (0.35,0.2725993584528394)
            (0.4,0.2281376791454533)
            (0.45,0.18909907172517004)
            (0.5,0.15467451982579422)
            (0.55,0.1243602084571698)
            (0.6,0.0977897747835797)
            (0.65,0.07469117255848917)
            (0.7,0.054875859848088675)
            (0.75,0.038169613925819676)
            (0.8,0.024516349615339775)
        }
        ;
    \addlegendentry {\PRP:$N=3$}
    \addplot+[style={dotted}, color=colorP5]
            coordinates {
                        (0.001,0.9920965745151636)
            (0.005,0.9634431381304839)
            (0.01,0.93231084872853)
            (0.02,0.8796153191664312)
            (0.03,0.8349699112877436)
            (0.04,0.7956647978178957)
            (0.05,0.7660226106788741)
            (0.06,0.7347844479971598)
            (0.07,0.7060949557418645)
            (0.08,0.6795100556464467)
            (0.09,0.6547027010834734)
            (0.1,0.6314236411806957)
            (0.15,0.532289241163245)
            (0.2,0.45317082867875613)
            (0.25,0.387290203023197)
            (0.3,0.33060275240595943)
            (0.35,0.2808064841650101)
            (0.4,0.23663219405343755)
            (0.45,0.1973206026733438)
            (0.5,0.16232414393492517)
            (0.55,0.13124014232070544)
            (0.6,0.10377235459481696)
            (0.65,0.07970289383319003)
            (0.7,0.058884028135058095)
            (0.75,0.041207668534850944)
            (0.8,0.026636256223944527)
        }
        ;
    \addlegendentry {\PRP:$N=4$}
    \addplot[no marks, style={{ultra thick}}, color={colorP6}]
        coordinates {
            (0.001,0.9377544467966324)
            (0.009070707070707073,0.8185901876393797)
            (0.017141414141414145,0.7552909671957262)
            (0.025212121212121217,0.7076456166048424)
            (0.03328282828282829,0.6684111890296339)
            (0.04135353535353536,0.634642161755327)
            (0.04942424242424243,0.6047929685613148)
            (0.0574949494949495,0.5779328596882196)
            (0.06556565656565658,0.5534498237088445)
            (0.07363636363636365,0.5309159433963891)
            (0.08170707070707071,0.5100180968512836)
            (0.08977777777777778,0.49051897633126484)
            (0.09784848484848487,0.4722336357648498)
            (0.10591919191919193,0.45501461544367827)
            (0.113989898989899,0.4387420956172438)
            (0.12206060606060606,0.4233171450948348)
            (0.13013131313131315,0.4086569528625282)
            (0.1382020202020202,0.3946913741396516)
            (0.1462727272727273,0.38136037341059414)
            (0.15434343434343434,0.36861209516079907)
            (0.16241414141414143,0.35640138366364177)
            (0.1704848484848485,0.3446886303175612)
            (0.17855555555555555,0.3334388640731176)
            (0.18662626262626264,0.3226210250813445)
            (0.19469696969696973,0.31220737837711793)
            (0.2027676767676768,0.30217303594934114)
            (0.21083838383838385,0.2924955636706998)
            (0.2189090909090909,0.2831546553671783)
            (0.226979797979798,0.27413186052115396)
            (0.23505050505050507,0.26541035520004624)
            (0.24312121212121213,0.2569747481084629)
            (0.2511919191919192,0.24881091539761505)
            (0.2592626262626263,0.24090585918612417)
            (0.2673333333333333,0.2332475857604673)
            (0.2754040404040404,0.22582500020927662)
            (0.2834747474747475,0.2186278148600219)
            (0.2915454545454546,0.2116464693705363)
            (0.29961616161616167,0.20487206071191538)
            (0.3076868686868687,0.19829628158624132)
            (0.3157575757575758,0.1919113660694783)
            (0.32382828282828285,0.1857100414697097)
            (0.3318989898989899,0.17968548555356706)
            (0.339969696969697,0.17383128842687398)
            (0.34804040404040404,0.16814141846510652)
            (0.3561111111111111,0.16261019177990066)
            (0.3641818181818182,0.15723224478313103)
            (0.3722525252525253,0.15200250947293054)
            (0.38032323232323234,0.14691619111869578)
            (0.38839393939393946,0.14196874806645438)
            (0.3964646464646465,0.13715587342342522)
            (0.4045353535353536,0.13247347841236476)
            (0.4126060606060606,0.12791767721332425)
            (0.4206767676767677,0.12348477313352707)
            (0.42874747474747477,0.11917124596585127)
            (0.43681818181818177,0.11497374041339718)
            (0.44488888888888894,0.11088905547227422)
            (0.452959595959596,0.10691413467740586)
            (0.4610303030303031,0.103046057127129)
            (0.46910101010101013,0.09928202921191448)
            (0.47717171717171714,0.09561937698084748)
            (0.48524242424242425,0.09205553908677613)
            (0.4933131313131313,0.0885880602574015)
            (0.5013838383838384,0.08521458524516302)
            (0.5094545454545455,0.08193285321369483)
            (0.5175252525252526,0.07874069252295314)
            (0.5255959595959596,0.07563601587894837)
            (0.5336666666666666,0.07261681581739926)
            (0.5417373737373737,0.06968116049363524)
            (0.5498080808080807,0.06682718975374116)
            (0.5578787878787879,0.064053111464317)
            (0.565949494949495,0.06135719808034647)
            (0.5740202020202021,0.05873778343256193)
            (0.5820909090909092,0.05619325971738821)
            (0.5901616161616162,0.05372207467406883)
            (0.5982323232323233,0.05132272893494058)
            (0.6063030303030303,0.04899377353605107)
            (0.6143737373737373,0.046733807576416396)
            (0.6224444444444445,0.04454147601521375)
            (0.6305151515151516,0.04241546759710283)
            (0.6385858585858587,0.04035451289668257)
            (0.6466565656565657,0.038357382473827514)
            (0.6547272727272728,0.036422885132315036)
            (0.6627979797979798,0.0345498662747615)
            (0.6708686868686868,0.032737206347435356)
            (0.678939393939394,0.03098381936901864)
            (0.687010101010101,0.029288651537843033)
            (0.6950808080808081,0.02765067991254437)
            (0.7031515151515152,0.026068911161461046)
            (0.7112222222222222,0.024542380376447123)
            (0.7192929292929293,0.023070149947091962)
            (0.7273636363636364,0.021651308491628083)
            (0.7354343434343436,0.020284969841077135)
            (0.7435050505050506,0.01897027207342895)
            (0.7515757575757577,0.017706376594874576)
            (0.7596464646464647,0.01649246726532145)
            (0.7677171717171718,0.015327749565609273)
            (0.7757878787878789,0.01421144980402251)
            (0.7838585858585859,0.013142814359855617)
            (0.791929292929293,0.012121108961938598)
            (0.8,0.011145618000168249)
        }
        ;
    \addlegendentry {$(1 - \sqrt{q})^2$}
      \coordinate (spypoint) at (0.09,0.7);
      \coordinate (magnifyglass) at (0.4,0.7);
      \coordinate (spypoint2) at (0.5,0.15);
       \coordinate (magnifyglass2) at (0.73,0.65);
\end{axis}
\spy [black, width=2.5cm,height=2cm] on (spypoint)
in node[fill=white] at (magnifyglass);
\spy [black, width=2.5cm,height=2cm] on (spypoint2)
 in node[fill=white] at (magnifyglass2);
\end{tikzpicture}

%% file: fr_restarted_master.tex
\begin{tikzpicture}[spy using outlines=
	{rectangle, magnification=3, connect spies}]
\begin{axis}[legend style= {at={(1,0.5)},anchor=west}, plotOptions6, 
xlabel={$\cond$}, ylabel={$\sqrt[N]{\rho_{N,0}(\cond)}$}]
    \addplot+[color=colorP1]
        coordinates {
            (0.001,0.9960079880159803)
            (0.005,0.9801985099378732)
            (0.01,0.9607881580237231)
            (0.02,0.9231064975009611)
            (0.03,0.8868884909039495)
            (0.04,0.8520710059171597)
            (0.05,0.81859410430839)
            (0.06,0.7864008543965822)
            (0.07,0.7554371560835006)
            (0.08,0.7256515775034293)
            (0.09,0.6969952024240383)
            (0.1,0.6694214876033058)
            (0.15,0.5463137996219282)
            (0.2,0.44444444444444453)
            (0.25,0.36)
            (0.3,0.2899408284023668)
            (0.35,0.23182441700960219)
            (0.4,0.1836734693877551)
            (0.45,0.14387633769322236)
            (0.5,0.1111111111111111)
            (0.55,0.08428720083246616)
            (0.6,0.0625)
            (0.65,0.04499540863177227)
            (0.7,0.03114186851211074)
            (0.75,0.02040816326530612)
            (0.8,0.012345679012345671)
        }
        ;
    \addlegendentry {$\textrm{GDEL}:\frac{f_{1}-f_{\star}}{f_{0}-f_{\star}}$}
    \addplot+[color=colorP2, style = {loosely dashed}]
        coordinates {
            (0.001,0.9960079880159803)
            (0.005,0.9801985099378732)
            (0.01,0.9607881580237231)
            (0.02,0.9231064975009611)
            (0.03,0.8868884909039495)
            (0.04,0.8520710059171597)
            (0.05,0.81859410430839)
            (0.06,0.7864008543965822)
            (0.07,0.7554371560835006)
            (0.08,0.7256515775034293)
            (0.09,0.6969952024240383)
            (0.1,0.6694214876033058)
            (0.15,0.5463137996219282)
            (0.2,0.44444444444444453)
            (0.25,0.36)
            (0.3,0.2899408284023668)
            (0.35,0.23182441700960219)
            (0.4,0.1836734693877551)
            (0.45,0.14387633769322236)
            (0.5,0.1111111111111111)
            (0.55,0.08428720083246616)
            (0.6,0.0625)
            (0.65,0.04499540863177227)
            (0.7,0.03114186851211074)
            (0.75,0.02040816326530612)
            (0.8,0.012345679012345671)
        }
        ;
    \addlegendentry {\FR: $N=1$}
    \addplot+[color=colorP3, style = {densely dashed}]
        coordinates {
            (0.001,0.994794788176151)
            (0.005,0.9745633187596008)
            (0.01,0.9504876789373703)
            (0.02,0.9058017195887705)
            (0.03,0.8649649512558778)
            (0.04,0.8272968420425596)
            (0.05,0.7923025718678359)
            (0.06,0.7596090008625469)
            (0.07,0.7289260549434058)
            (0.08,0.7000224559620427)
            (0.09,0.6727098519085505)
            (0.1,0.6468320974422438)
            (0.15,0.5349677961696735)
            (0.2,0.4452538562596982)
            (0.25,0.3715323506251218)
            (0.3,0.3099827728421002)
            (0.35,0.25802774287802327)
            (0.4,0.2138133943281805)
            (0.45,0.17551167023922562)
            (0.5,0.1421032248247169)
            (0.55,0.11303946278264354)
            (0.6,0.08789585393732917)
            (0.65,0.0663416070658978)
            (0.7,0.048118424881190475)
            (0.75,0.03302551646023749)
            (0.8,0.020904637971570987)
        }
        ;
    \addlegendentry {\FR: $N=2$}
    \addplot+[style={dashed}, color=colorP4]
        coordinates {
            (0.001,0.9945675538490653)
            (0.005,0.9740361892268338)
            (0.01,0.9505565297087464)
            (0.02,0.9088339047104007)
            (0.03,0.8720624007732602)
            (0.04,0.8388557415388964)
            (0.05,0.8083991271674248)
            (0.06,0.780162378576236)
            (0.07,0.7537755613240519)
            (0.08,0.728967091813677)
            (0.09,0.7055297578895868)
            (0.1,0.6833005716062965)
            (0.15,0.5863944717168539)
            (0.2,0.5068663822015108)
            (0.25,0.4391499135968998)
            (0.3,0.378912546550944)
            (0.35,0.32472080463922187)
            (0.4,0.2757689552877994)
            (0.45,0.23150143467562787)
            (0.5,0.1915289280267366)
            (0.55,0.15557916512020417)
            (0.6,0.12346703430139987)
            (0.65,0.09507623272411858)
            (0.7,0.07033954357444377)
            (0.75,0.04923974755653201)
            (0.8,0.03179369795840537)
        }
        ;
    \addlegendentry {\FR: $N=3$}
    \addplot+[style={dotted}, color=colorP5]
        coordinates {
            (0.001,0.9945847333624366)
            (0.005,0.9748358634807474)
            (0.01,0.9531775123207921)
            (0.02,0.9160209896444994)
            (0.03,0.8840186787465147)
            (0.04,0.8554207380281375)
            (0.05,0.8293183693341128)
            (0.06,0.8051621639845378)
            (0.07,0.782586962478332)
            (0.08,0.7613336423669252)
            (0.09,0.741209128398433)
            (0.1,0.7220638797958839)
            (0.15,0.6374728930205857)
            (0.2,0.5649949618809958)
            (0.25,0.49944293862912675)
            (0.3,0.43926333631189507)
            (0.35,0.38357397390375486)
            (0.4,0.3318287079028512)
            (0.45,0.28369444771066443)
            (0.5,0.23898602579645226)
            (0.55,0.19762831277584494)
            (0.6,0.1596378222552749)
            (0.65,0.12510410147563342)
            (0.7,0.09417911616251094)
            (0.75,0.0670871972456026)
            (0.8,0.044092418264659504)
        }
        ;
    \addlegendentry {\FR: $N=4$}
    \addplot[no marks, color={colorP6}]
        coordinates {
            (0.001,0.9377544467966324)
            (0.009070707070707073,0.8185901876393797)
            (0.017141414141414145,0.7552909671957262)
            (0.025212121212121217,0.7076456166048424)
            (0.03328282828282829,0.6684111890296339)
            (0.04135353535353536,0.634642161755327)
            (0.04942424242424243,0.6047929685613148)
            (0.0574949494949495,0.5779328596882196)
            (0.06556565656565658,0.5534498237088445)
            (0.07363636363636365,0.5309159433963891)
            (0.08170707070707071,0.5100180968512836)
            (0.08977777777777778,0.49051897633126484)
            (0.09784848484848487,0.4722336357648498)
            (0.10591919191919193,0.45501461544367827)
            (0.113989898989899,0.4387420956172438)
            (0.12206060606060606,0.4233171450948348)
            (0.13013131313131315,0.4086569528625282)
            (0.1382020202020202,0.3946913741396516)
            (0.1462727272727273,0.38136037341059414)
            (0.15434343434343434,0.36861209516079907)
            (0.16241414141414143,0.35640138366364177)
            (0.1704848484848485,0.3446886303175612)
            (0.17855555555555555,0.3334388640731176)
            (0.18662626262626264,0.3226210250813445)
            (0.19469696969696973,0.31220737837711793)
            (0.2027676767676768,0.30217303594934114)
            (0.21083838383838385,0.2924955636706998)
            (0.2189090909090909,0.2831546553671783)
            (0.226979797979798,0.27413186052115396)
            (0.23505050505050507,0.26541035520004624)
            (0.24312121212121213,0.2569747481084629)
            (0.2511919191919192,0.24881091539761505)
            (0.2592626262626263,0.24090585918612417)
            (0.2673333333333333,0.2332475857604673)
            (0.2754040404040404,0.22582500020927662)
            (0.2834747474747475,0.2186278148600219)
            (0.2915454545454546,0.2116464693705363)
            (0.29961616161616167,0.20487206071191538)
            (0.3076868686868687,0.19829628158624132)
            (0.3157575757575758,0.1919113660694783)
            (0.32382828282828285,0.1857100414697097)
            (0.3318989898989899,0.17968548555356706)
            (0.339969696969697,0.17383128842687398)
            (0.34804040404040404,0.16814141846510652)
            (0.3561111111111111,0.16261019177990066)
            (0.3641818181818182,0.15723224478313103)
            (0.3722525252525253,0.15200250947293054)
            (0.38032323232323234,0.14691619111869578)
            (0.38839393939393946,0.14196874806645438)
            (0.3964646464646465,0.13715587342342522)
            (0.4045353535353536,0.13247347841236476)
            (0.4126060606060606,0.12791767721332425)
            (0.4206767676767677,0.12348477313352707)
            (0.42874747474747477,0.11917124596585127)
            (0.43681818181818177,0.11497374041339718)
            (0.44488888888888894,0.11088905547227422)
            (0.452959595959596,0.10691413467740586)
            (0.4610303030303031,0.103046057127129)
            (0.46910101010101013,0.09928202921191448)
            (0.47717171717171714,0.09561937698084748)
            (0.48524242424242425,0.09205553908677613)
            (0.4933131313131313,0.0885880602574015)
            (0.5013838383838384,0.08521458524516302)
            (0.5094545454545455,0.08193285321369483)
            (0.5175252525252526,0.07874069252295314)
            (0.5255959595959596,0.07563601587894837)
            (0.5336666666666666,0.07261681581739926)
            (0.5417373737373737,0.06968116049363524)
            (0.5498080808080807,0.06682718975374116)
            (0.5578787878787879,0.064053111464317)
            (0.565949494949495,0.06135719808034647)
            (0.5740202020202021,0.05873778343256193)
            (0.5820909090909092,0.05619325971738821)
            (0.5901616161616162,0.05372207467406883)
            (0.5982323232323233,0.05132272893494058)
            (0.6063030303030303,0.04899377353605107)
            (0.6143737373737373,0.046733807576416396)
            (0.6224444444444445,0.04454147601521375)
            (0.6305151515151516,0.04241546759710283)
            (0.6385858585858587,0.04035451289668257)
            (0.6466565656565657,0.038357382473827514)
            (0.6547272727272728,0.036422885132315036)
            (0.6627979797979798,0.0345498662747615)
            (0.6708686868686868,0.032737206347435356)
            (0.678939393939394,0.03098381936901864)
            (0.687010101010101,0.029288651537843033)
            (0.6950808080808081,0.02765067991254437)
            (0.7031515151515152,0.026068911161461046)
            (0.7112222222222222,0.024542380376447123)
            (0.7192929292929293,0.023070149947091962)
            (0.7273636363636364,0.021651308491628083)
            (0.7354343434343436,0.020284969841077135)
            (0.7435050505050506,0.01897027207342895)
            (0.7515757575757577,0.017706376594874576)
            (0.7596464646464647,0.01649246726532145)
            (0.7677171717171718,0.015327749565609273)
            (0.7757878787878789,0.01421144980402251)
            (0.7838585858585859,0.013142814359855617)
            (0.791929292929293,0.012121108961938598)
            (0.8,0.011145618000168249)
        }
        ;
    \addlegendentry {$(1 - \sqrt{q})^2$}
          \coordinate (spypoint) at (0.04,0.87);
      \coordinate (magnifyglass) at (0.4,0.7);
      \coordinate (spypoint2) at (0.7,0.076);
       \coordinate (magnifyglass2) at (0.73,0.65);
\end{axis}
\spy [black, width=2.5cm,height=2cm] on (spypoint)
in node[fill=white] at (magnifyglass);
\spy [black, width=2.5cm,height=2cm] on (spypoint2)
 in node[fill=white] at (magnifyglass2);
\end{tikzpicture}